\documentclass{amsart}

 \usepackage[foot]{amsaddr}
\let\oldtocsection=\tocsection
 
\let\oldtocsubsection=\tocsubsection 
 
\let\oldtocsubsubsection=\tocsubsubsection
 
\renewcommand{\tocsection}[2]{\vspace{0.5em}\hspace{0em}\oldtocsection{#1}{#2}}
\renewcommand{\tocsubsection}[2]{\vspace{0.5em}\hspace{1em}\oldtocsubsection{#1}{#2}}
\renewcommand{\tocsubsubsection}[2]{\vspace{0.5em}\hspace{2em}\oldtocsubsubsection{#1}{#2}}
\usepackage{graphicx,cancel,xcolor,hyperref,comment,graphicx,geometry}

\usepackage{tikz}
\usepackage{graphicx,url,etoolbox}
\usepackage{lipsum}
\usepackage{dsfont}
 \usepackage{amssymb}
\newcommand{\isEquivTo}[1]{\underset{#1}{\sim}}
 
\usepackage{graphicx,cancel,xcolor,hyperref,comment,graphicx,geometry}

\usepackage{tikz}
\usetikzlibrary{decorations.pathreplacing}
\usepackage{graphicx,url,etoolbox}
\usepackage{lipsum}
\usepackage{dsfont}

\usepackage{fancyhdr}
\usepackage{lastpage}

\newtheorem{theoreme}{Theorem}[section]
\newtheorem{pro}[theoreme]{Proposition}
\newtheorem{lemma}[theoreme]{Lemma}
\newtheorem{rem}[theoreme]{Remark}

\newtheorem{definition}[theoreme]{Definition}
\theoremstyle{definition}

\numberwithin{equation}{section}

 \renewenvironment{proof}{{\bfseries \noindent Proof.}}{\demo}
\newcommand\xqed[1]{%
  \leavevmode\unskip\penalty9999 \hbox{}\nobreak\hfill
  \quad\hbox{#1}}
\newcommand\demo{\xqed{$\square$}}

\hypersetup{bookmarks, colorlinks, urlcolor=blue, citecolor=blue, linkcolor=blue, hyperfigures, pagebackref,
    pdfcreator=LaTeX, breaklinks=true, pdfpagelayout=SinglePage, bookmarksopen=true,bookmarksopenlevel=2}

\usepackage{comment}

\def\R{\mathbb R}

\def\HH{\mathcal H}

\def\AA{\mathcal A}

\def\la {{\lambda}}

\newcommand {\nc}   {\newcommand}
\nc {\be}   {\begin{equation}} \nc {\ee}   {\end{equation}} \nc
{\beq}  {\begin{eqnarray}} \nc {\eeq}  {\end{eqnarray}} \nc {\beqs}
{\begin{eqnarray*}} \nc {\eeqs} {\end{eqnarray*}}
\def\edc{\end{document}}

\providecommand{\abs}[1]{\lvert#1\rvert}

\usepackage{tikz}
\usetikzlibrary{decorations.pathmorphing,patterns,scopes,intersections,calc}
\usepackage{caption}
\usepackage{float}

\newcounter{dummy} 
\numberwithin{dummy}{section}

\newtheorem{Lemma}[dummy]{Lemma}

\numberwithin{equation}{section}

\begin{document}
\title[\fontsize{7}{9}\selectfont  ]{Energy Decay of some boundary coupled systems involving wave$\backslash$ Euler-Bernoulli beam with one locally singular fractional Kelvin-Voigt damping}
\author{Mohammad Akil$^{1}$}
\author{Ibtissam Issa$^{2,3}$}
\author{Ali Wehbe$^{2}$}
\address{$^1$ Universit\'e Savoie Mont Blanc - Chamb\'ery - France, Laboratoire LAMA}
\address{$^2$Lebanese University, Faculty of sciences 1, Khawarizmi Laboratory of  Mathematics and Applications-KALMA, Hadath-Beirut, Lebanon.}
\address{$^3$ Universit\'e Aix-Marseilles - Marseille  - France, Laboratoire I2M}
\email{mohammad.akil@univ-smb.fr}
\email{ali.wehbe@ul.edu.lb}
\email{ibtissam.issa@etu.univ-amu.fr}


\begin{abstract}
In this paper, we investigate the energy decay of hyperbolic systems of wave-wave, wave-Euler-Bernoulli beam and beam-beam types. The two equations are coupled through boundary connection with only one localized non-smooth fractional Kelvin-Voigt damping. First, we reformulate each system into an augmented model and using a general criteria of Arendt-Batty, we prove that our models are strongly stable. Next, by using frequency domain approach, combined with multiplier technique and some interpolation inequalities, we establish different types of polynomial energy decay rate which depends on the order of the fractional derivative and the type of the damped equation in the system.

\end{abstract}
\maketitle
\pagenumbering{roman}
\maketitle
\tableofcontents
\pagenumbering{arabic}
\setcounter{page}{1}
\section{Introduction} 
\subsection{Literature}
In recent years, many researches showed interest in studying the stability and controllability of certain system. The wave equation with different kinds of damping was studied extensively. The wave is created when a vibrating source disturbs the medium. In order to restrain those vibrations, several dampings can be added such as Kelvin-Voigt damping. Many researchers were interested in problems involving this kind of damping (local or global) where different types of stability have been showed. We refer to \cite{Huang-falun,chen1,bardos,chen-1998,liu-liu-1998,zhang-2010,Alves2,liu-zhang-2016,Ammari03} and the rich references therein.\\

The beam, or flexural member, is frequently encountered in structures and machines, and its elementary stress analysis constitutes one of the most interesting facts of mechanics of materials. For beams, there was an extensive studying, since 80's, on the stabilization of the beam equations (see \cite{Ch1987,K-R-1987} for the one dimensional system, and \cite{B-T1991} for n-dimensional system). Also, the studies considered the linear and nonlinear boundary feedback acting through shear forces and moments  \cite{LAS1987,LAS1990,LAS1989} and the case control by moment has been studied in \cite{LAS1999}. \\
The studying of the beam equation with different types of damping was extensively considered. In 1998, the author in \cite{liu-liu-1998} considered the longitudinal and transversal vibrations of the Euler-Bernoulli beam with Kelvin-Voigt damping distributed locally on any subinterval of the region occupied by the beam. It was shown that when the viscoelastic damping is distributed only on a subinterval in the interior of the domain, the exponential stability holds for the transversal but not for the longitudinal motion.
In \cite{Raposo}, they considered a transmission problem for the longitudinal displacement of a Euler-Bernoulli beam, where one small part of the beam is made of a viscoelastic material with Kelvin-Voigt constitutive relation and they proved that the semigroup associated to the system is exponentially stable.\\

Another type of damping which was studied extensively in the past few years is the fractional damping. It is widely applied in the domain of science. The fractional-order type is not only important from the theoretical point of view but also for applications.
They naturally arise in physical, chemical, biological, and ecological phenomena see for example \cite{app}, and the rich
references therein. They are used to describe memory and hereditary properties of various materials and processes. For example, in viscoelasticity, due to the nature of the material microstructure, both elastic solid and viscous fluid-like response qualities are involved. Using Boltzmann’s assumption, we end up with a stress strain relationship defined by a time convolution. The viscoelastic response occurs in a variety of materials,
such as soils, concrete, rubber, cartilage, biological tissue, glasses, and polymers (see \cite{Ronald,peter,RL,Bonetti}). Fractional computing in modeling can improve the capturing of the complex dynamics of natural systems, and controls of fractional order type can improve performance not achievable before using controls of integer-order type. For example, systems in many quantum mechanics, nuclear physics and biological phenomena such as fluid flow are indeed fractional (see for example \cite{Torvik,benotti,Igor}).\\

Fractional calculus includes various extensions of the usual definition of derivative from integer to real order, including the Hadamard, Erdelyi-Kober, Riemann-Liouville, Riesz, Weyl, Grünwald-Letnikov, Jumarie and the Caputo representation. A thorough analysis of fractional dynamical systems is necessary to achieve an appropriate definition of the fractional derivative. For example, the Riemann-Liouville definition entails physically unacceptable initial conditions (fractional order
initial conditions); conversely, for the Caputo representation, which is introduced by Michele Caputo \cite{caputo} in 1967, the initial conditions are expressed in terms of integer-order derivatives having direct physical significance; this definition is mainly used to include memory effects. Recently, in \cite{Caputo-Fab} a new definition of the fractional derivative was presented without a singular kernel; this derivative possesses very interesting properties, for instance the possibility to describe fluctuations and structures with different scales.
The case of wave equation with boundary fractional damping have been treated in \cite{Mbodje1,Mbodje2} where they proved the strong stability and the lack of uniform stabilization. However, the case of the plate equation or the beam equation with boundary fractional damping was treated in
\cite{benaissa} where they showed that the energy is
polynomially stable. In \cite{akil-mcrf}, they considered a multidimensional wave equation with boundary fractional damping acting on a part of the boundary of the domain. They established a polynomial energy decay rate for smooth solutions, under some geometric conditions. Ammari et al., in \cite{Ammari-fathi}, studied the stabilization for a class of evolution systems with fractional-damping. They proved the polynomial stability of the system.\\

The notion of indirect damping mechanisms has been introduced by Russell \cite{RUSSELL1993} and since that time, it retains the attention of many authors. In particular, the fact that only one equation of the coupled system is damped refers to the so-called class of “indirect” stabilization problems initiated and studied in previous studies \cite{Alabau1999,Alabau2002} and further studied by many authors; see for instance \cite{Alabau-cannarsa-komornik,Alabau2007,Alabau-Cannarsa-Guglielmi2011} and the rich references therein. Over the past few years, the coupled systems received a vast attention due to their potential applications. The coupled systems have many applications in the modeling and control of engineering, such as: aircraft, satellite antennas and road traffic(see \cite{saroj} for example). Most of the work in the coupled system considers the stability of the system with various coupling, damping locations, and damping types.
Many researches studied coupling systems with a Kelvin-Voigt damping such as wave-wave system, heat-wave system, Timoshinko (see \cite{akil2020stability,Zhang-Zuazua,alabau}). In 2012, Tebou in \cite{Teb12} considered the Euler-Bernoulli equation coupled with a wave equation in a
bounded domain. The frictional damping is distributed everywhere
in the domain and acts through one of the equations only. For the case where the dissipation acts through the Euler-Bernoulli equation he showed that the
system is not exponentially stable and that the energy decays polynomially was proved. For the case where the damping acts through the wave equation polynomial stability was proved. \\
Benaissa et al., in \cite{Benissa-time}, considered the large time behavior of one dimensional coupled wave equations with fractional control applied at the coupled point. They showed an optimal decay result.\\
In \cite{Fathi-2015}, Hassine considered a beam and a wave equations coupled on an elastic beam through transmission conditions where the locally distributed damping acts through one of the two equations only. The systems are described as follows
{\small{\begin{equation}\label{Fathi}
\left\{\begin{array}{ll}
u_{tt}-(u_x+D_a u_{xt})_x=0, &\text{in}\,\Omega_1,\\ 
y_{tt}+y_{xxxx}=0,&\text{in}\, \Omega_2,\\ 
u(\ell, t) =y(\ell,t) ,&t>0,\\ 
y_{x}(\ell, t) = 0,&t>0,\\ 
u_x(\ell,t)+y_{xxx}(\ell,t)=0, &t>0,\\ 
u(0,t)=y(L,t)=y_x(L,t)=0, &t>0,\\ 
u(x,0) = u_0(x), u_t(x, 0) = u_1(x), &x\in (0,\ell),\\ 
y(x,0) = y_0(x), y_t(x, 0) = y_1(x), &x\in (\ell,L)
\end{array}\right.
\ \text{and} \ 
\left\{\begin{array}{ll}
y_{tt}-(y_{xx}+D_b y_{xxt})_{xx}=0, &\text{in} \,\Omega_1,\\ 
u_{tt}-u_{xx}=0,&\text{in} \,\Omega_2,\\ 
u(\ell, t) =y(\ell,t) ,&t>0,\\ 
y_{x}(\ell, t) = 0,&t>0,\\ 
u_x(\ell,t)+y_{xxx}(\ell,t)=0, &t>0,\\ 
u(0,t)=y(L,t)=y_x(L,t)=0, &t>0,\\ 
u(x,0) = u_0(x), u_t(x, 0) = u_1(x), &x\in (0,\ell),\\ 
y(x,0) = y_0(x), y_t(x, 0) = y_1(x), &x\in (\ell,L)
\end{array}\right.
\end{equation}
}}
\noindent where $\Omega_1=(0,\ell)\times \R^+$, $\Omega_2=(\ell,L)\times \R^+$
, $D_a=a(x)\chi_{(e,f)}$ and $D_b=b(x)\chi_{(e,f)}$ with $0<e<f<\ell<L$ and $a(x),b(x)\geq c_0>0$ in $(e,f)$. The author proved that for the case when the dissipation acts through the wave, the energy of this coupled system decays polynomially as the time variable goes to infinity. Also, for the case where the damping acts through the beam equation polynomial stability was proved. \\
The case of a  Euller-Bernoulli beam  and a wave equations coupled via the interface by transmission conditions was considered by Hassine in \cite{Fathi-2016-N-dim}, where he supposed that the beam equation is stabilized by a localized distributed feedback. He reached that sufficiently smooth solutions decay logarithmically at infinity even the feedback dissipation affects an arbitrarily small open subset of the interior.
\\
In \cite{lu}, the authors studied the stabilization system of a coupled wave and a Euler-Bernoulli plate equation where only one equation is supposed to be damped with a frictional damping in the multidimensional case. Under some assumption about the damping and the coupling terms, they showed that sufficiently smooth solutions of the system decay logarithmically at infinity without any geometric conditions on the effective damping domain. \\
In \cite{Ammari-serge}, Ammari and Nicaise, considered the stabilization problem for coupling the damped wave equation with a damped Kirchhoff plate equation. They proved an exponential stability result under some geometric condition.
In 2018, the authors considered in \cite{Li-Han-Xu}, a system of 1-d coupled string-beam. They obtained two kinds of energy decay rates of the string-beam system with different locations of the frictional damping. On one hand, if the frictional damping is only actuated in the beam part, the system lacks exponential decay. Specifically, the optimal polynomial decay rate $t^{-1}$ is obtained under smooth initial conditions. On the other hand, if the frictional damping is only effective in the string part, the exponential decay of energy is presented.
In 2020, the authors in \cite{wang}, considered a system of two-dimensional coupled wave-plate with local frictional damping in a bounded domain. The frictional damping is only distributed in the part of the plate's or wave's domain, and the other is stabilized by the transmission through the interface of the plate's and wave's domains. They showed that the  energy of the system decays polynomially under some geometric condition when the frictional damping only acts on the part of the plate, and the energy of the system is exponentially stable when the frictional damping acts only on the other part of the wave. \\  
In 2018, Guo and Ren in \cite{Guo2020}, studied the stabilization for a hyperbolic-hyperbolic coupled system consisting of Euler-Bernoulli beam and wave equations, where the structural damping of the wave equation is taken into account. The coupling is actuated through boundary weak connection. The system is described as follows
\begin{equation}\label{Guo}
\left\{\begin{array}{ll}
w_{tt}+w_{xxxx}=0, &(x,t)\in (0,1)\times \R^+,\\ \\
u_{tt}-u_{xx}-s u_{xxt}=0,&(x,t)\in (0,1)\times \R^+,\\ \\
w(1, t) = w_{xx}(1, t) = w(0, t) = 0,&t>0,\\ \\
w_{xx}(0, t) = r u_t(0, t),\,u(1, t) = 0,&x\in (0,1),\\ \\
s u_{xt}(0,t)+u_x(0,t)=-r w_{xt}(0,t), &x\in (0,1),\\ \\
w(x, 0) = w_0(x), w_t(x, 0) = w_1(x), &x\in (0,1),\\ \\
u(x, 0) = u_0(x), u_t(x, 0) = u_1(x), &x\in (0,1),
\end{array}\right.
\end{equation}
where $(w_0, w_1, u_0, u_1)$ is the initial state and $r\neq 0,s>0$ are constants. They concluded the Riesz basis property
and the exponential stability of the system.

In \cite{Ammari-J-Mehren2009}, the authors considered a stabilization problem for a coupled string-beam system. They proved some decay results of the energy of the system. Moreover, they proved, for the same model but with two control functions, independently of the length of the beam that the energy decay with a polynomial rate for all regular initial data. In \cite{Ammari-mehren2009}, the authors considered a stabilization problem for a string-beams network and proved an exponential decay result. 
In \cite{Ammari-2009-CUBO}, the author considered  a  boundary stabilization  problem  for  the transmission  Bernoulli-Euler plate equation and proved a uniform exponential energy decay under natural conditions. In \cite{Rivera-Santos2011}, a coupled system of wave-plate type with thermal effect was studied and exponentially stability was proved. In \cite{Denk}, the authors considered the transmission problem for a coupled system of undamped and structurally damped plate equations in two sufficiently smooth and bounded subdomains. They showed, independently of the size
of the damped part, that the damping is strong enough to produce uniform exponential decay of the energy of the coupled system. In 2019, Liu and Han \cite{LIU-HAN2019}, considered a system of coupled plate equations where indirect structural or Kelvin-Voigt damping is imposed, i.e., only one equation is directly damped by one of these two damping. They showed that the associated semigroup of the system with indirect structural damping is analytic and exponentially stable. However, with the much stronger indirect Kelvin-Voigt damping, they proved that the semigroup is even not differentiable and that the exponential stability is still maintained.

\subsection{Physical interpretation of the models}
In the first model (EBB)-W$_{FKV}$, we investigate the stability of coupled Euler-Bernoulli beam and wave equations. The coupling is via boundary connections with localized non-regular fractional Kelvin-Voigt damping, where the damping acts through the wave equation only (see Figure \ref{fig1}). The system that describes this model is as follows
\begin{equation*}\label{Sys1}\tag{(EBB)-W$_{FKV}$}
\left\lbrace\begin{array}{ll}
u_{tt}-\left(au_x+d(x)\partial_t^{\alpha,\eta}u_x\right)_x=0,&(x,t)\in (0,L)\times (0,\infty),\\[0.1in]
y_{tt}+by_{xxxx}=0,&(x,t)\in (-L,0)\times (0,\infty),\\[0.1in]
u(L,t)=y(-L,t)=y_{x}(-L,t)=0,&t\in (0,\infty),\\[0.1in]
au_{x}(0,t)+by_{xxx}(0,t)=0,y_{xx}(0,t)=0,&t\in (0,\infty),\\[0.1in]
u(0,t)=y(0,t),&t\in (0,\infty),\\[0.1in]
u(x,0)=u_0(x),\, u_t(x,0)=u_1(x),& x\in (0,L),\\[0.1in]
y(x,0)=y_0(x),\, y_t(x,0)=y_1(x),& x\in (-L,0).
\end{array}
\right.
\end{equation*}
The coefficients $a, b$ are strictly positive constant numbers, $\alpha\in (0,1)$ and $\eta\geq 0$. We suppose that there exists $0< l_0<l_1< L$ and a strictly positive constant $d_0$, such that 
\begin{equation}
d(x)=\left\{
\begin{array}{ll}
d_0,&x\in (l_0,l_1)\\[0.1in]
0,&x\in (0,l_0)\cup (l_1,L).
\end{array}
\right.
\end{equation}
\noindent The Caputo's fractional derivative $\partial_t^{\alpha,\eta}$ of order $\alpha\in (0,1)$ with respect to time variable $t$ defined by 
\begin{equation}\label{Caputo}
[D^{\alpha,\eta}\omega](t)=\partial_t^{\alpha,\eta}\omega(t)=\frac{1}{\Gamma(1-\alpha)}\int_0^t(t-s)^{-\alpha}e^{-\eta(t-s)}\frac{d\omega}{ds}(s)ds,
\end{equation}
where $\Gamma$ denotes the Gamma function.\\
\begin{figure}[H]\label{fig1}
	\begin{center}
		\begin{tikzpicture}
		\draw[blue,-](0.7,1)--(3.7,1);
		\draw[blue,-] (0.7,0)--(0.7,1);
		\draw[blue,-] (3.7,0)--(3.7,1);
		\draw[blue,-](0.7,0)--(3.7,0);
		\draw[blue,dashed] (0.7,0.5)--(6.7,0.5);
		
		\draw[blue,-] (0.7,-0.5)--(0.7,1.5);
		\draw[blue,-] (6.7,-0.5)--(6.7,1.5);

		\draw[blue,-] (3.7,0.5)--(4,1);
		\draw[blue,-](4,1)--(4.5,0.1);
		\draw[red,-](4.5,0.1)--(5,1);
		\draw[red,-](5,1)--(5.5,0.1);
		\draw[red,-](5.5,0.1)--(6,1);
		\draw[blue,-](6,1)--(6.5,0.1);
		\draw[blue,-](6.5,0.1)--(6.7,0.5);

		\node[blue,above] at (3.59,0.4){\scalebox{0.75}{$0$}};
		\node at (3.7,0.5) [circle, scale=0.3, draw=blue!80,fill=blue!80] {};

		\node[red,above] at (6.05,0.5){\scalebox{0.75}{$l_1$}};
		\node at (6,0.5) [circle, scale=0.3, draw=red!80,fill=red!80] {};
		
		\node[blue,above] at (6.6,0.5){\scalebox{0.75}{$L$}};
		\node at (6.7,0.5) [circle, scale=0.3, draw=blue!80,fill=blue!80] {};
		\node[red,above] at (4.5,0.5){\scalebox{0.75}{$l_0$}};
		\node at (4.5,0.5) [circle, scale=0.3, draw=red!80,fill=red!80] {};
		
		\node[blue,above] at (1,0.4){\scalebox{0.75}{$-L$}};
		\node at (0.7,0.5) [circle, scale=0.3, draw=blue!80,fill=blue!80] {};
		
		\tikzstyle{ground}=[pattern=north east lines,draw=none,minimum width=0.75cm,minimum height=0.3cm]
		
		\node (wall) at (3.55,0.5) [ground, rotate=-90, minimum width=2cm,yshift=-3cm] {};
		
		\node (wall) at (9.85,0.5) [ground, rotate=-90, minimum width=2cm,yshift=-3cm] {};
		
		\draw [red,decorate,decoration={brace,amplitude=10pt},xshift=4pt,yshift=0pt]
		(4.3,1.25) -- (5.9,1.25)  ;
		
		\node[red,above] at (5.3,1.5){\scalebox{0.75}{FKV-damping}}; 
		
		\draw [blue,decorate,decoration={brace,amplitude=10pt,mirror,raise=4pt},yshift=0pt]
		(0.8,0) -- (3.7,0)  ;		
		\node[blue,below] at (2.25,-0.75){\scalebox{0.75}{ Beam Part}}; 
		
		\draw [blue,decorate,decoration={brace,amplitude=10pt,mirror,raise=4pt},yshift=0pt]
		(3.75,0) -- (6.6,0)  ;	
		\node[blue,below] at (5.3,-0.75){\scalebox{0.75}{ Wave Part}};

		\end{tikzpicture}
	\end{center}
		\caption{(EBB)-W$_{FKV}$ Model}\label{fig1}
\end{figure}
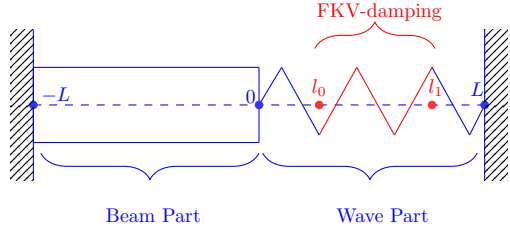

\noindent In the second model W-W$_{FKV}$, we investigate the stability of coupled wave equations coupled through boundary connections with localized non-regular fractional Kelvin-Voigt damping acting through one wave equation only (see Figure \ref{fig2}). The system that describes this model is as follows
\begin{equation*}\label{Sys1-W}\tag{W-W$_{FKV}$}
\left\lbrace\begin{array}{ll}
u_{tt}-\left(au_x+d(x)\partial_t^{\alpha,\eta}u_x\right)_x=0,&(x,t)\in (0,L)\times (0,\infty),\\[0.1in]
y_{tt}-by_{xx}=0,&(x,t)\in (-L,0)\times (0,\infty),\\[0.1in]
u(L,t)=y(-L,t)=0,&t\in (0,\infty),\\[0.1in]
au_{x}(0,t)=by_{x}(0,t),&t\in (0,\infty),\\[0.1in]
u(0,t)=y(0,t),&t\in (0,\infty),\\[0.1in]
u(x,0)=u_0(x),\, u_t(x,0)=u_1(x),& x\in (0,L),\\[0.1in]
y(x,0)=y_0(x),\, y_t(x,0)=y_1(x),& x\in (-L,0).
\end{array}
\right.
\end{equation*}

\begin{figure}[H]\label{fig2}
	\begin{center}
		\begin{tikzpicture}
			
		\draw[green,-] (0.7,0.5)--(1,1);
		\draw[green,-](1,1)--(1.5,0.1);
		\draw[green,-](1.5,0.1)--(2,1);
		\draw[green,-](2,1)--(2.5,0.1);
		\draw[green,-](2.5,0.1)--(3,1);
		\draw[green,-](3,1)--(3.5,0.1);
		\draw[green,-](3.5,0.1)--(3.7,0.5);	
		
		\draw[blue,dashed] (0.7,0.5)--(6.7,0.5);

		\draw[blue,-] (3.7,0.5)--(4,1);
		\draw[blue,-](4,1)--(4.5,0.1);
		\draw[red,-](4.5,0.1)--(5,1);
		\draw[red,-](5,1)--(5.5,0.1);
			\draw[red,-](5.5,0.1)--(6,1);
				\draw[blue,-](6,1)--(6.5,0.1);
					\draw[blue,-](6.5,0.1)--(6.7,0.5);

		\node[blue,above] at (3.59,0.4){\scalebox{0.75}{$0$}};
		\node at (3.7,0.5) [circle, scale=0.3, draw=blue!80,fill=blue!80] {};

			\node[red,above] at (6.05,0.5){\scalebox{0.75}{$l_1$}};
		\node at (6,0.5) [circle, scale=0.3, draw=red!80,fill=red!80] {};
		
			\node[blue,above] at (6.6,0.5){\scalebox{0.75}{$L$}};
		\node at (6.7,0.5) [circle, scale=0.3, draw=blue!80,fill=blue!80] {};
			\node[red,above] at (4.5,0.5){\scalebox{0.75}{$l_0$}};
		\node at (4.5,0.5) [circle, scale=0.3, draw=red!80,fill=red!80] {};
		
			\node[blue,above] at (1,0.4){\scalebox{0.75}{$-L$}};
		\node at (0.7,0.5) [circle, scale=0.3, draw=blue!80,fill=blue!80] {};
		
		\tikzstyle{ground}=[pattern=north east lines,draw=none,minimum width=0.75cm,minimum height=0.3cm]
		
	\node (wall) at (3.55,0.5) [ground, rotate=-90, minimum width=2cm,yshift=-3cm] {};
		
			\node (wall) at (9.85,0.5) [ground, rotate=-90, minimum width=2cm,yshift=-3cm] {};
		
		\draw [red,decorate,decoration={brace,amplitude=10pt},xshift=4pt,yshift=0pt]
	(4.3,1.25) -- (5.9,1.25)  ;
		
			\node[red,above] at (5.3,1.5){\scalebox{0.75}{FKV-damping}}; 
		
			\draw [blue,decorate,decoration={brace,amplitude=10pt,mirror,raise=4pt},yshift=0pt]
	(0.8,0) -- (3.7,0)  ;		
		\node[blue,below] at (2.25,-0.75){\scalebox{0.75}{ Wave Part}}; 
	
		\draw [blue,decorate,decoration={brace,amplitude=10pt,mirror,raise=4pt},yshift=0pt]
	(3.75,0) -- (6.6,0)  ;	
		\node[blue,below] at (5.3,-0.75){\scalebox{0.75}{ Wave Part}};

		\end{tikzpicture}
	\end{center}
	\caption{W-W$_{FKV}$ Model}\label{fig2}
\end{figure}
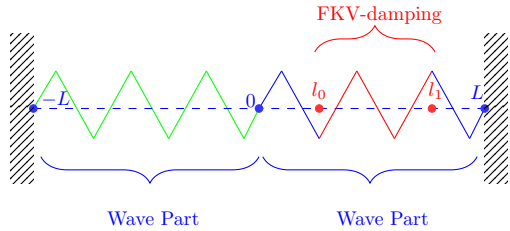

\noindent In the third model W-(EBB)$_{FKV}$, we consider a system of coupled Euler-Bernouli beam and wave equations. These two equations are coupled through boundary connections. In this case the localized non-smooth fractional Kelvin-Voigt damping acts only on  the Euler-Bernoulli beam (see Figure \ref{fig3}). The system that represents this model is as follows
\begin{equation*}\label{beam1}\tag{W-(EBB)$_{FKV}$}
\left\lbrace\begin{array}{ll}
u_{tt}-au_{xx}=0,&(x,t)\in (-L,0)\times (0,\infty),\\[0.1in]
y_{tt}+\left(by_{xx}+d(x)\partial_t^{\alpha,\eta}y_{xx}\right)_{xx}=0,&(x,t)\in (0,L)\times (0,\infty),\\[0.1in]
u(-L,t)=y(L,t)=y_{x}(L,t)=0,&t\in (0,\infty),\\[0.1in]
au_{x}(0,t)+by_{xxx}(0,t)=0,y_{xx}(0,t)=0,&t\in (0,\infty),\\[0.1in]
u(0,t)=y(0,t),&t\in (0,\infty),\\[0.1in]
u(x,0)=u_0(x),\, u_t(x,0)=u_1(x),& x\in (-L,0),\\[0.1in]
y(x,0)=y_0(x),\, y_t(x,0)=y_1(x),& x\in (0,L).
\end{array}
\right.
\end{equation*}

\begin{figure}[H]\label{fig3}
	\begin{center}
		\begin{tikzpicture}
		\draw[blue,-](3.7,1)--(4.5,1);
		\draw[blue,-](6,1)--(6.7,1);
		\draw[blue,-] (3.7,0)--(3.7,1);
		\draw[blue,-] (6.7,0)--(6.7,1);
		\draw[blue,-](3.7,0)--(4.5,0);
		\draw[blue,-](6,0)--(6.7,0);
		\draw[blue,dashed] (0.7,0.5)--(4.5,0.5);
		\draw[blue,dashed] (6,0.5)--(6.7,0.5);
		
		\draw[red,-](4.5,0)--(4.5,1);
		\draw[red,-](6,0)--(6,1);

		\draw[red,-](4.5,0)--(6,0);
		\draw[red,-](4.5,0.1)--(6,0.1);
		\draw[red,-](4.5,0.2)--(6,0.2);
		\draw[red,-](4.5,0.3)--(6,0.3);
		\draw[red,-](4.5,0.4)--(6,0.4);
		\draw[red,-](4.5,0.5)--(6,0.5);
		\draw[red,-](4.5,0.6)--(6,0.6);
		\draw[red,-](4.5,0.7)--(6,0.7);
		\draw[red,-](4.5,0.8)--(6,0.8);
		\draw[red,-](4.5,0.9)--(6,0.9);
		\draw[red,-](4.5,1)--(6,1);
		
		\draw[blue,-] (0.7,-0.5)--(0.7,1.5);
		\draw[blue,-] (6.7,-0.5)--(6.7,1.5);

		\draw[blue,-] (0.7,0.5)--(1,1);
		\draw[blue,-](1,1)--(1.5,0.1);
		\draw[blue,-](1.5,0.1)--(2,1);
		\draw[blue,-](2,1)--(2.5,0.1);
		\draw[blue,-](2.5,0.1)--(3,1);
		\draw[blue,-](3,1)--(3.5,0.1);
		\draw[blue,-](3.5,0.1)--(3.7,0.5);

		\node[blue,above] at (3.6,0.5){\scalebox{0.75}{$0$}};
		\node at (3.7,0.5) [circle, scale=0.3, draw=blue!80,fill=blue!80] {};

		\node[red,above] at (6.2,0.5){\scalebox{0.75}{$l_1$}};
		\node at (6,0.5) [circle, scale=0.3, draw=red!80,fill=red!80] {};
		
		\node[blue,above] at (6.6,0.5){\scalebox{0.75}{$L$}};
		\node at (6.7,0.5) [circle, scale=0.3, draw=blue!80,fill=blue!80] {};
		\node[red,above] at (4.3,0.5){\scalebox{0.75}{$l_0$}};
		\node at (4.5,0.5) [circle, scale=0.3, draw=red!80,fill=red!80] {};
		
		\node[blue,above] at (0.98,0.4){\scalebox{0.75}{$-L$}};
		\node at (0.7,0.5) [circle, scale=0.3, draw=blue!80,fill=blue!80] {};
		
		\tikzstyle{ground}=[pattern=north east lines,draw=none,minimum width=0.75cm,minimum height=0.3cm]
		
		\node (wall) at (3.55,0.5) [ground, rotate=-90, minimum width=2cm,yshift=-3cm] {};
		
		\node (wall) at (9.85,0.5) [ground, rotate=-90, minimum width=2cm,yshift=-3cm] {};
		
		\draw [red,decorate,decoration={brace,amplitude=10pt},xshift=4pt,yshift=0pt]
		(4.3,1.25) -- (5.9,1.25)  ;
		
		\node[red,above] at (5.3,1.5){\scalebox{0.75}{FKV-damping}}; 
		
		\draw [blue,decorate,decoration={brace,amplitude=10pt,mirror,raise=4pt},yshift=0pt]
		(0.8,0) -- (3.7,0)  ;		
		\node[blue,below] at (2.25,-0.75){\scalebox{0.75}{ Wave Part}}; 
		
		\draw [blue,decorate,decoration={brace,amplitude=10pt,mirror,raise=4pt},yshift=0pt]
		(3.75,0) -- (6.6,0)  ;	
		\node[blue,below] at (5.3,-0.75){\scalebox{0.75}{ Beam Part}};

		\end{tikzpicture}
	\end{center}
		\caption{W-(EBB)$_{FKV}$ Model}\label{fig3}
\end{figure}
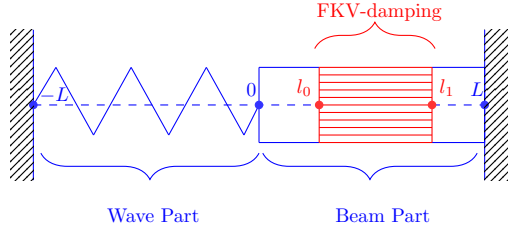

\noindent In the fourth model ((EBB)$_{FKV}$), we study a system of Euler-Bernoulli beam with a non-regular localized fractional Kelvin-Voigt damping (see Figure \ref{fig4}). The system is as follows
\begin{equation*}\label{beamAlone}\tag{(EBB)$_{FKV}$}
\left\lbrace\begin{array}{ll}
y_{tt}+\left(by_{xx}+d(x)\partial_t^{\alpha,\eta}y_{xx}\right)_{xx}=0,&(x,t)\in (0,L)\times (0,\infty),\\[0.1in]
y(0,t)=y_x(0,t)=y_{xx}(L,t)=y_{xxx}(L,t)=0,&t\in (0,\infty),\\[0.1in]
y(x,0)=y_0(x),\, y_t(x,0)=y_1(x),& x\in (0,L).
\end{array}
\right.
\end{equation*}
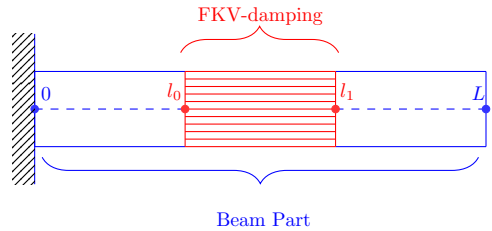
\begin{figure}[H]\label{fig4}
	\begin{center}
		\begin{tikzpicture}
		\draw[blue,-](0.7,1)--(2.7,1);
		\draw[blue,-](4.7,1)--(6.7,1);
		\draw[blue,-] (0.7,0)--(0.7,1);
		\draw[blue,-] (6.7,0)--(6.7,1);
		\draw[blue,-](0.7,0)--(2.7,0);
		\draw[blue,-](4.7,0)--(6.7,0);
		\draw[blue,dashed] (0.7,0.5)--(2.7,0.5);
		\draw[blue,dashed] (4.7,0.5)--(6.7,0.5);
		
		\draw[red,-](2.7,0)--(2.7,1);
		\draw[red,-](4.7,0)--(4.7,1);

		\draw[red,-](2.7,0)--(4.7,0);
		\draw[red,-](2.7,0.1)--(4.7,0.1);
		\draw[red,-](2.7,0.2)--(4.7,0.2);
		\draw[red,-](2.7,0.3)--(4.7,0.3);
		\draw[red,-](2.7,0.4)--(4.7,0.4);
		\draw[red,-](2.7,0.5)--(4.7,0.5);
		\draw[red,-](2.7,0.6)--(4.7,0.6);
		\draw[red,-](2.7,0.7)--(4.7,0.7);
		\draw[red,-](2.7,0.8)--(4.7,0.8);
		\draw[red,-](2.7,0.9)--(4.7,0.9);
		\draw[red,-](2.7,1)--(4.7,1);
		
		\draw[blue,-] (0.7,-0.5)--(0.7,1.5);



		\node[red,above] at (4.86,0.5){\scalebox{0.75}{$l_1$}};
		\node at (4.7,0.5) [circle, scale=0.3, draw=red!80,fill=red!80] {};
		
		\node[blue,above] at (6.6,0.5){\scalebox{0.75}{$L$}};
		\node at (6.7,0.5) [circle, scale=0.3, draw=blue!80,fill=blue!80]  {};

		\node[red,above] at (2.56,0.5){\scalebox{0.75}{$l_0$}};
		\node at (2.7,0.5) [circle, scale=0.3, draw=red!80,fill=red!80] {};
		
		\node[blue,above] at (0.85,0.5){\scalebox{0.75}{$0$}};
		\node at (0.7,0.5) [circle, scale=0.3, draw=blue!80,fill=blue!80] {};
		
		\tikzstyle{ground}=[pattern=north east lines,draw=none,minimum width=0.75cm,minimum height=0.3cm]
		
		\node (wall) at (3.55,0.5) [ground, rotate=-90, minimum width=2cm,yshift=-3cm] {};
		
		
		\draw [red,decorate,decoration={brace,amplitude=10pt},xshift=4pt,yshift=0pt]
		(2.5,1.25) -- (4.6,1.25)  ;
		
		\node[red,above] at (3.7,1.5){\scalebox{0.75}{FKV-damping}}; 
		
		
		\draw [blue,decorate,decoration={brace,amplitude=10pt,mirror,raise=4pt},yshift=0pt]
		(0.8,0) -- (6.6,0)  ;	
		\node[blue,below] at (3.7,-0.75){\scalebox{0.75}{ Beam Part}};

		\end{tikzpicture}
	\end{center}
	\caption{(EBB)$_{FKV}$ Model}\label{fig4}
\end{figure}
\noindent In the last model (EBB)-(EBB)$_{FKV}$, we consider a system of two Euler-Bernouli beam equations coupled through boundary connections. The localized non-smooth fractional Kelvin-Voigt damping acts only on  one of the two equations (see Figure \ref{fig5}). The system that represents this model is as follows
\begin{equation*}\label{coupledBeam}\tag{(EBB)-(EBB)$_{FKV}$}
\left\lbrace\begin{array}{ll}
u_{tt}+au_{xxxx}=0,&(x,t)\in (-L,0)\times (0,\infty),\\[0.1in]
y_{tt}+\left(by_{xx}+d(x)\partial_t^{\alpha,\eta}y_{xx}\right)_{xx}=0,&(x,t)\in (0,L)\times (0,\infty),\\[0.1in]
u(-L,t)=u_{x}(-L,t)=y(L,t)=y_{x}(L,t)=0,&t\in (0,\infty),\\[0.1in]
au_{xxx}(0,t)-by_{xxx}(0,t)=0,u_{xx}(0)=y_{xx}(0,t)=0,&t\in (0,\infty),\\[0.1in]
u(0,t)=y(0,t),&t\in (0,\infty),\\[0.1in]
u(x,0)=u_0(x),\, u_t(x,0)=u_1(x),& x\in (-L,0),\\[0.1in]
y(x,0)=y_0(x),\, y_t(x,0)=y_1(x),& x\in (0,L).
\end{array}
\right.
\end{equation*}
\begin{figure}[H]
	\begin{center}
		\begin{tikzpicture}
		\draw[blue,-](0.7,1)--(4.5,1);
		\draw[blue,-](5.9,1)--(6.7,1);
		\draw[blue,-] (0.7,0)--(0.7,1);
		\draw[blue,-] (6.7,0)--(6.7,1);
		\draw[blue,-](0.7,0)--(4.5,0);
		\draw[blue,-](5.9,0)--(6.7,0);
		\draw[blue,dashed] (0.7,0.5)--(4.5,0.5);
		\draw[blue,dashed] (5.9,0.5)--(6.7,0.5);
		
		\draw[red,-](5.9,0)--(5.9,1);
		\draw[red,-](4.5,0)--(4.5,1);

		\draw[red,-](4.5,0)--(5.9,0);
		\draw[red,-](4.5,0.1)--(5.9,0.1);
		\draw[red,-](4.5,0.2)--(5.9,0.2);
		\draw[red,-](4.5,0.3)--(5.9,0.3);
		\draw[red,-](4.5,0.4)--(5.9,0.4);
		\draw[red,-](4.5,0.5)--(5.9,0.5);
		\draw[red,-](4.5,0.6)--(5.9,0.6);
		\draw[red,-](4.5,0.7)--(5.9,0.7);
		\draw[red,-](4.5,0.8)--(5.9,0.8);
		\draw[red,-](4.5,0.9)--(5.9,0.9);
		\draw[red,-](4.5,1)--(5.9,1);
		
		\draw[blue,-] (0.7,-0.5)--(0.7,1.5);
			\draw[blue,-] (6.7,-0.5)--(6.7,1.5);
				\draw[blue,-] (3.7,0)--(3.7,1);



		\node[red,above] at (6.09,0.5){\scalebox{0.75}{$l_1$}};
		\node at (5.9,0.5) [circle, scale=0.3, draw=red!80,fill=red!80] {};
		
		\node[blue,above] at (6.6,0.5){\scalebox{0.75}{$L$}};
		\node at (6.7,0.5) [circle, scale=0.3, draw=blue!80,fill=blue!80]  {};

		\node[blue,above] at (3.6,0.5){\scalebox{0.75}{$0$}};
		\node at (3.7,0.5) [circle, scale=0.3, draw=blue!80,fill=blue!80]  {};

		\node[red,above] at (4.35,0.5){\scalebox{0.75}{$l_0$}};
		\node at (4.5,0.5) [circle, scale=0.3, draw=red!80,fill=red!80] {};
		
		\node[blue,above] at (0.95,0.5){\scalebox{0.75}{$-L$}};
		\node at (0.7,0.5) [circle, scale=0.3, draw=blue!80,fill=blue!80] {};
		
		\tikzstyle{ground}=[pattern=north east lines,draw=none,minimum width=0.75cm,minimum height=0.3cm]
		
		\node (wall) at (3.55,0.5) [ground, rotate=-90, minimum width=2cm,yshift=-3cm] {};
		
		\node (wall) at (9.85,0.5) [ground, rotate=-90, minimum width=2cm,yshift=-3cm] {};
		
		\draw [red,decorate,decoration={brace,amplitude=10pt},xshift=4pt,yshift=0pt]
		(4.3,1.25) -- (5.8,1.25)  ;
		
		\node[red,above] at (5.2,1.5){\scalebox{0.65}{FKV-damping}}; 
		
		
		\draw [blue,decorate,decoration={brace,amplitude=10pt,mirror,raise=4pt},yshift=0pt]
		(0.8,0) -- (3.6,0)  ;	
		\node[blue,below] at (2.19,-0.65){\scalebox{0.65}{ Beam Part}}; 
		
		\draw [blue,decorate,decoration={brace,amplitude=10pt,mirror,raise=4pt},yshift=0pt]
		(3.7,0) -- (6.6,0)  ;	
		\node[blue,below] at (5.1,-0.65){\scalebox{0.65}{ Beam Part}};

		\end{tikzpicture}
	\end{center}
		\caption{(EBB)-(EBB)$_{FKV}$}\label{fig5}
\end{figure}

\noindent 
We give the physical meaning of the following variables. 
\begin{equation}
\begin{array}{lll}
y=\text{vertical displacement}\, ,\hspace{0.5cm}u_x=\text{the stress of the wave}.\\
y_x=\text{rotation},\\
y_{xx}=\text{Bending moment},\\
y_{xxx}=\text{Shear Force},\\
y_{xxxx}=\text{Loading}.
\end{array}
\end{equation}
The way the beam supported is translated into conditions on the function $y$ and its derivatives. These conditions are collectively referred to as boundary conditions. They are meaningful in physics and engineering. The boundary conditions in the model \eqref{beamAlone} signifies the following\\
$\bullet \quad y(0,t)=0$: This signifies that the beam is pinned to its support, which means that the beam cannot experience any deflection at $x=0$.\\
$\bullet \quad y_x(0,t)=0$: It signifies that the rotation at the pinned support is zero.\\ 
$\bullet \quad y_{xx}(L,t)=0$: It means that there is no bending moment at the free end of the beam.\\ 
$\bullet \quad y_{xxx}(L,t)=0$: This boundary conditions gives the assumption that there is no shearing force acting at the free end of the beam.\\
This kind of models, supported at one end with the other end free, described in the above four conditions can be referred to as a cantilever beam. A good example on the cantilever beam is a balcony, it is supported at one end only, the rest of the beam extends over the open space. Other examples are a cantilever roof in a bus shelter, car park or railway station. We give some of the advantages and disadvantages of the cantilever beam.\\
Advantages:
\begin{itemize}
\item Cantilever beams do not require support on the opposite side.
\item The negative bending moment created in cantilever beams helps to counteract the positive bending moments created.
\item Cantilever beams can be easily constructed.
\item These beam enables erection with little disturbance in navigation.
\end{itemize}
Disadvantages:
\begin{itemize}
\item Cantilever beams are subjected to large deflections.
\item Cantilever beams are subjected to larger moments.
\item A strong fixed support or a backspan is necessary to keep the structure stable.
\end{itemize}
In a cantilever beam, the bending moment at the free end always vanishes. In fact, if we connect the beam with a wave (see \eqref{Sys1}, \eqref{beam1}) at the free end the bending moment will be zero and this induces a shear force on the end of the beam. Consequently, the fourth boundary condition above is no longer valid, and it is replaced by $au_x(0,t)+by_{xxx}(0,t)=0$. This condition signifies that the shear force of the beam and the stress force of the wave are such that one cancels the other.
\subsection{Description of the paper}
In this paper, we investigate the stability results of four models of systems with a non-smooth localized fractional Kelvin-Voigt damping where the coupling is made via boundary connections. In the first model \eqref{Sys1} we consider the coupled Euler-Bernoulli beam and wave equation with the damping acts on the wave equation only.  In  Subsection \ref{strongStabilityWave}, we reformulate \eqref{Sys1} into an augmented model and we prove the well-posedness of the system by semigroup approach. Moreover, using a general criteria of Arendt and Batty, we show the strong stability of our system in the absence of the compactness of the resolvent. In section \ref{Section-poly}, using the semigroup theory of linear operators and a result obtained by Borichev and Tomilov we show that the energy of the System \eqref{Sys1} has a polynomial decay rate of type $t^{\frac{-4}{2-\alpha}}$. In the second model \eqref{Sys1-W}, we consider two wave equations coupled through boundary connections with a non-smooth localized fractional Kelvin-Voigt damping acting only on one of the two equations. We establish a polynomial energy decay rate of type $t^{\frac{-4}{2-\alpha}}$.
In the third model \eqref{beam1}, we consider Euler-Bernoulli beam and wave equations coupled through boundary connections with the damping to act through the Euler-Bernoulli beam equation only. For this model, we show that the energy of the System \eqref{beam1} has a polynomial decay rate of type $t^{\frac{-2}{3-\alpha}}$. For the model \eqref{beamAlone} we consider the Euler-Bernoulli beam with a non-smooth localized fractional Kelvin-Voigt damping. We prove that the energy of the system \eqref{Beam-Alone} decays polynomially with a decay rate $t^{\frac{-2}{1-\alpha}}$. Finally, for the fourth model \eqref{coupledBeam}, we study the polynomial stability of two  Euler-Bernoulli beam equations coupled through boundary connections with damping acting only on one of the two equation. We establish a polynomial energy decay rate of type $t^{\frac{-2}{3-\alpha}}$.
The table below (Table \ref{tabel}) summarizes the decay rate of the energy for the five models. Also, it gives the decay rate of the same four models but with Kelvin-Voigt damping (as $\alpha\to 1$).

\begin{table}[!htb]
\centering
\scalebox{1.3}{
\begin{tabular}{|c|c|c|}
\hline
Model & Decay Rate & $\alpha\to 1$\\
\hline
(EBB)-W$_{FKV}$& $t^{\frac{-4}{2-\alpha}}$ & $t^{-4}$\\
\hline
W-W$_{FKV}$& $t^{\frac{-4}{2-\alpha}}$ & $t^{-4}$\\
\hline
W-(EBB)$_{FKV}$& $t^{\frac{-2}{3-\alpha}}$ & $t^{-1}$\\
\hline
(EBB)$_{FKV}$&$t^{\frac{-2}{1-\alpha}}$ & Exponential\\
\hline
(EBB)-(EBB)$_{FKV}$&$t^{\frac{-2}{3-\alpha}}$ & $t^{-1}$\\
\hline
\end{tabular}}
\caption{Decay Results}
\label{tabel}
\end{table}
\noindent Some significant deductions on the energy deacy rate are given below:\\
$\bullet$ From the decay rate of the models \eqref{Sys1} and \eqref{beam1}  we can deduce that if we want to choose the place where the fractional Kelvin-Voigt damping acts it is better to choose the damping on the wave. Since the energy of this model decays faster compared with that of \eqref{beam1} model.\\
$\bullet$ For the model \eqref{Sys1}, if we replace the condition of the null bending moment $(y_{xx}(0)=0)$ at the connecting boundary by taking the rotation to be zero $(y_x(0)=0)$, we get the same decay rate. So, this result improves the work in \cite{Fathi-2015} where they reached energy decay rate of type $t^{-2}$, however in our work we proved an energy decay rate of type $t^{-4}$( as $\alpha\to 1$)( see Section \ref{DAMPED-WAVE}).\\
$\bullet$ For the model \eqref{Sys1}, we established an energy decay rate of type $t^{-4}$ (as $\alpha\to 1$). By comparing this energy decay rate with that in \cite{akil2020stability}, where the authors considered two wave equations coupled through velocity with localized non smooth Kelvin-Voigt and they established an energy decay rate of type $t^{-1}$. We can see that, by comparing the energy decay rate of these two system that it is better, when considering Kelvin-Voigt damping, to consider the coupling through the boundary connection rather through the velocity.
\begin{rem}
We note that in the upcoming sections, the letters used to denote the variables are independent from each other in each section.
\end{rem}
\section{(EBB)-W$_{FKV}$ Model}\label{DAMPED-WAVE}
\noindent In this section, we consider the \eqref{Sys1} model, where we study the stability of the system a Euler-Bernoulli and wave equations coupled through boundary connection with a localized fractional Kelvin-Voigt damping acting on the wave equation only. 
\subsection{Well-Posedness and Strong Stability}\label{strongStabilityWave}
\noindent In this subsection, we study the strong stability of the system \eqref{Sys1} in the absence of the compactness of the resolvent. First, we will study the existence, uniqueness and regularity of the solution of the system.
\subsubsection{Augmented model and Well-Posedness.}
In this part, using a semigroup approach, we establish well-posedness for the system \eqref{Sys1}. First, we recall theorem 2 stated in \cite{Mbodje1,Akil-AA}.
\begin{theoreme}\label{theorem1}
Let $\alpha\in (0,1)$, $\eta\geq 0$ and $\mu(\xi)=\abs{\xi}^{\frac{2\alpha-1}{2}}$ be the function defined almost everywhere on $\mathbb{R}$. The relation between the 'input' V and the 'output' O of the following system
\begin{eqnarray}
\partial_t\omega(x,\xi,t)+(\xi^2+\eta)\omega(x,\xi,t)-V(x,t)|\xi|^{\frac{2\alpha-1}{2}}&=&0,(x,\xi,t)\in (0,L)\times \mathbb{R}\times (0,\infty),\label{aug1}\\ 
\omega(x,\xi,0)&=&0,(x,\xi)\in (0,L)\times \mathbb{R},\label{aug2}\\
\displaystyle{O(x,t)-\kappa(\alpha)\int_{\mathbb{R}}|\xi|^{\frac{2\alpha-1}{2}}\omega(x,\xi,t)d\xi }&=&0,(x,t)\in (0,L)\times (0,\infty),\label{aug3}
\end{eqnarray}
is given by 
\begin{equation}\label{relation}
O=I^{1-\alpha,\eta}V,
\end{equation}
where 
$$
[I^{\alpha,\eta}V](x,t)=\frac{1}{\Gamma(\alpha)}\int_0^t(t-s)^{\alpha-1}e^{-\eta(t-s)}V(s)ds\quad \text{and}\quad \kappa(\alpha)=\frac{\sin(\alpha \pi)}{\pi} .
$$
\end{theoreme}

\noindent In the above theorem, taking the input $V(x,t)=\sqrt{d(x)}u_{xt}(x,t)$, then  using Equation \eqref{Caputo}, we get that the output $O$ is given by 
$$
O(x,t)=\sqrt{d(x)}I^{1-\alpha,\eta}u_{xt}(x,t)=\frac{\sqrt{d(x)}}{\Gamma(1-\alpha)}\int_0^t(t-s)^{-\alpha}e^{-\eta(t-s)}\partial_su_x(x,s)ds=\sqrt{d(x)}\partial_t^{\alpha,\eta}u_{x}(x,t).
$$
Therefore, by taking the input $V(x,t)=\sqrt{d(x)}u_{xt}(x,t)$ in Theorem \ref{theorem1} and using the above equation, we get 
\begin{equation}\label{augnew}
\begin{array}{llll}
\partial_t\omega(x,\xi,t)+(\xi^2+\eta)\omega(x,\xi,t)-\sqrt{d(x)}u_{xt}(x,t)|\xi|^{\frac{2\alpha-1}{2}}&=&0,&(x,\xi,t)\in (0,L)\times \mathbb{R}\times (0,\infty),\\
\omega(x,\xi,0)&=&0,&(x,\xi)\in (0,L)\times \mathbb{R},\\
\displaystyle{\sqrt{d(x)}\partial_t^{\alpha,\eta}u_{x}(x,t)-\kappa(\alpha)\int_{\mathbb{R}}|\xi|^{\frac{2\alpha-1}{2}}\omega(x,\xi,t)d\xi }&=&0,&(x,t)\in (0,L)\times (0,\infty).
\end{array}
\end{equation}
From system \eqref{augnew}, we deduce that system \eqref{Sys1} can be recast into the following augmented model 
\begin{equation}\label{AUG1}
\left\{\begin{array}{ll}
 \displaystyle{u_{tt}-\left(au_x+\sqrt{d(x)}\kappa(\alpha)\int_{\mathbb{R}}\abs{\xi}^{\frac{2\alpha-1}{2}}\omega(x,\xi,t)d\xi\right)_{x}=0},&(x,t)\in (0,L)\times \R_{\ast}^+,\\[0.1in]
\displaystyle{y_{tt}+by_{xxxx}=0},&(x,t)\in (-L,0)\times \times \R_{\ast}^+,\\[0.1in]
\omega_t(x,\xi,t)+\left(|\xi|^2+\eta\right)\omega(x,\xi,t)-\sqrt{d(x)}u_{xt}(x,t)\abs{\xi}^{\frac{2\alpha-1}{2}}=0,&(x,\xi,t)\in (0,L)\times \mathbb{R}\times \R_{\ast}^+,
\end{array}
\right.
\end{equation}
with the following transmission and boundary conditions 
\begin{equation}\label{AUG2}
\left\{\begin{array}{ll}
u(L,t)=y(-L,t)=y_{x}(-L,t)=0,&t\in (0,\infty),\\[0.1in]
au_{x}(0,t)+by_{xxx}(0,t)=0,y_{xx}(0,t)=0,&t\in (0,\infty),\\[0.1in]
u(0,t)=y(0,t),&t\in (0,\infty),
\end{array}\right.
\end{equation}
and with the following initial conditions 
\begin{equation}\label{AUG3}
\begin{array}{lll}
u(x,0)=u_0(x),& u_t(x,0)=u_1(x),\hspace{0.2cm}\omega(x,\xi,0)=0& x\in (0,L),\xi\in\mathbb{R},\\[0.1in]
y(x,0)=y_0(x),& y_t(x,0)=y_1(x),& x\in (-L,0).
\end{array}
\end{equation}
The energy of the system \eqref{AUG1}-\eqref{AUG3} is given by  
\begin{equation*}
E_1(t)=\frac{1}{2}\int_0^L\left(\abs{u_t}^2+a\abs{u_x}^2\right)dx+\frac{1}{2}\int_{-L}^0\left(\abs{y_t}^2+b\abs{y_{xx}}^2\right)dx+\frac{\kappa(\alpha)}{2}\int_0^L\int_{\R}\abs{\omega(x,\xi,t)}^2d\xi dx.
\end{equation*}
\begin{lemma}\label{Denergy}
Let $U=(u,u_t,y,y_t,\omega)$ be a regular solution of the System \eqref{AUG1}-\eqref{AUG3}. Then, the energy $E_1(t)$ satisfies the following estimation 
\begin{equation}\label{denergy}
\frac{d}{dt}E_1(t)=-\kappa(\alpha)\int_{0}^{L}\int_{\mathbb{R}}(\xi^2+\eta)\abs{\omega(x,\xi,t)}^2d\xi dx.
\end{equation}
\end{lemma}
\begin{proof}
First, multiplying the first and the second equations  of \eqref{AUG1} by $\overline{u}_t$ and $\overline{y}_t$ respectively, integrating over $(0,L)$ and $(-L,0)$ respectively, using integration by parts with \eqref{AUG2} and taking the real part $\Re$, we get
\begin{equation}\label{denergy1}
\begin{array}{cc}
\displaystyle{\frac{1}{2}\frac{d}{dt}\left(\int_0^L\abs{u_t}^2+a\abs{u_x}^2\right)dx+
\frac{1}{2}\frac{d}{dt}\int_{-L}^0\left(\abs{y_t}^2+b\abs{y_{xx}}^2\right)dx}\\ [0.1in]
\displaystyle{+\Re\left(\kappa(\alpha)\int_0^L\sqrt{d(x)}\bar{u}_{tx}\left(\int_{\mathbb{R}}|\xi|^{\frac{2\alpha-1}{2}}\omega(x,\xi,t)d\xi\right)dx\right)=0}.
\end{array}
\end{equation}
Now, multiplying the third equation of \eqref{AUG1} by $\kappa(\alpha)\bar{\omega}$, integrating over $(0,L)\times \mathbb{R}$, then taking the real part, we get 
\begin{equation}\label{denergy2}
\begin{array}{cc}
\displaystyle{\frac{\kappa(\alpha)}{2}\frac{d}{dt}\int_{0}^{L}\int_{\mathbb{R}}\abs{\omega(x,\xi,t)}^2d\xi dx+\kappa(\alpha)\int_{0}^{L}\int_{\R}\left(\xi^2+\eta\right)\abs{\omega(x,\xi,t)}^2d\xi dx}\\[0.1in]
\displaystyle{=\Re\left(\kappa(\alpha)\int_0^L\sqrt{d(x)}u_{xt}\left(\int_{\mathbb{R}}|\xi|^{\frac{2\alpha-1}{2}}\overline{\omega}(x,\xi,t)d\xi\right)dx\right)}.
\end{array}
\end{equation}
Finally, by adding \eqref{denergy1} and \eqref{denergy2}, we obtain \eqref{denergy}. The proof is thus complete.
\end{proof}
$\newline$
\noindent Since $\alpha\in (0,1)$, then $\kappa(\alpha)>0$, and therefore $\displaystyle\frac{d}{dt}E_1(t)\leq 0$. Thus, system \eqref{AUG1}-\eqref{AUG3} is dissipative in the sense that its energy is a non-increasing function with respect to time variable $t$. Now, we define the following Hilbert energy space $\mathcal{H}_1$ by 
$$
\mathcal{H}_1=\left\{(u,v,y,z,\omega)\in H_R^1(0,L)\times L^2(0,L)\times H^2_L(-L,0)\times L^2(-L,0)\times W;\ \ u(0)=y(0)\right\},
$$
where $W=L^2\left((0,L)\times \R\right)$ and 
\begin{equation}
\left\{
\begin{array}{ll}
H_R^1(0,L)=\{u\in H^1(0,L); u(L)=0\},\\[0.1in]
H_L^2(-L,0)=\{y\in H^2(-L,0); y(-L)=y_x(-L)=0\}.
\end{array}\right.
\end{equation}
We note that the space $\mathcal{H}_2$ is a closed subspace of $H_R^1(0,L)\times L^2(0,L)\times H^2_L(-L,0)\times L^2(-L,0)\times W$.\\
The energy space $\mathcal{H}_1$ is equipped with the inner product defined by 
$$
\begin{array}{lll}
\displaystyle
\left<U,U_1\right>_{\mathcal{H}_1}&=&\displaystyle
\int_0^Lv\overline{v_1}dx+a\int_0^Lu_x(\overline{u_1})_xdx+\int_{-L}^0 z\overline{z_1}dx+b\int_{-L}^0 y_{xx}(\overline{y_1})_{xx}dx
+\kappa(\alpha)\int_{0}^{L}\int_{\mathbb{R}}\omega(x,\xi)\overline{\omega_1}(x,\xi)d\xi dx,
\end{array}
$$
for all $U=(u,v,y,z,\omega)$ and $U_1=(u_1,v_1,y_1,z_1,\omega_1)$ in $\mathcal{H}_1$. We use $\|U\|_{\mathcal{H}_1}$ to denote the corresponding norm. We define the unbounded linear operator $\AA_1:D(\AA_1)\subset\mathcal{H}_1\rightarrow\mathcal{H}_1$ by 
\begin{equation*}
D(\mathcal{A}_1)=\left\{\begin{array}{c}
\displaystyle{U=(u,v,y,z,\omega)\in \mathcal{H}_1;\ (v,z)\in H_R^1(0,L)\times H^2_L(-L,0),\ y\in  H^4(-L,0)},\\[0.1in]  
\displaystyle{\left(au_x+\sqrt{d(x)}\kappa(\alpha)\int_{\mathbb{R}}\abs{\xi}^{\frac{2\alpha-1}{2}}\omega(x,\xi)d\xi\right)_{x}}\in L^2(0,L),\vspace{0.2cm}\\ [0.1in]
-\left(\abs{\xi}^2+\eta\right)\omega(x,\xi)+\sqrt{d(x)}v_x|\xi|^{\frac{2\alpha-1}{2}},\quad|\xi|\omega(x,\xi)\in W, \vspace{0.2cm}\\[0.1in]
au_x(0)+by_{xxx}(0)=0,\,  y_{xx}(0)=0,\,\text{and}\, v(0)=z(0)
\end{array}
\right\},
\end{equation*}
and for all $U=(u,v,y,z,\omega)\in D(\mathcal{A}_1)$, 
$$
\mathcal{A}_1(u,v,y,z,\omega)^{\top}=\begin{pmatrix}
v\\ \vspace{0.2cm} \displaystyle{\left(au_x+\sqrt{d(x)}\kappa(\alpha)\int_{\mathbb{R}}\abs{\xi}^{\frac{2\alpha-1}{2}}\omega(x,\xi)d\xi\right)_{x}}\\ \vspace{0.2cm} z\\\vspace{0.2cm} -by_{xxxx}\\ \vspace{0.2cm}
-\left(\abs{\xi}^2+\eta\right)\omega(x,\xi)+\sqrt{d(x)}v_x|\xi|^{\frac{2\alpha-1}{2}}
\end{pmatrix}.
$$
\begin{rem}\label{Rem1}
The condition $\abs{\xi}\omega(x,\xi)\in W$ is imposed to insure the existence of $\displaystyle{\int_{0}^{L}\int_{\R}(\xi^2+\eta)\abs{\omega(x,\xi)}^2d\xi dx}$ in \eqref{denergy} and $\displaystyle{\sqrt{d(x)}\int_{\R}|\xi|^{\frac{2\alpha-1}{2}}\omega(x,\xi)d\xi}\in L^2(0,L)$.
\end{rem}
\noindent If $U=(u,u_t,y,y_t,\omega)$ is a regular solution of system \eqref{AUG1}-\eqref{AUG3}, then the system can be rewritten as evolution equation on the Hilbert space $\mathcal{H}_1$ given by
\begin{equation}\label{evolution-w}
U_t=\mathcal{A}_1U,\quad U(0)=U_0,
\end{equation}
where $U_0=(u_0,u_1,y_0,y_1,0)$.

\begin{Lemma}\label{lemI123}
Let $\alpha\in (0,1)$, $\eta\geq 0$, then the following  integrals 
\begin{equation}\label{I123}
\begin{array}{lll}
\displaystyle{\mathtt{I}_1(\eta,\alpha)=\kappa(\alpha)\int_{\R}\frac{\abs{\xi}^{2\alpha-1}}{1+\xi^2+\eta}d\xi\quad ,\quad\mathtt{I}_2(\eta,\alpha)=\int_{\R}\frac{\abs{\xi}^{2\alpha-1}}{(1+\xi^2+\eta)^2}d\xi} \vspace{0.6cm}\\ 
\hspace{0.8cm}\displaystyle{\text{and} \quad\mathtt{I}_3(\eta,\alpha)=\int_0^{+\infty}\frac{\xi^{2\alpha+1}}{(1+\xi^2+\eta)^2}d\xi}
\end{array}
\end{equation}
are well defined. 
\end{Lemma}
\begin{proof}
First, $\mathtt{I}_1(\eta,\alpha)$ can be written as 
\begin{equation}\label{eqI1}
\mathtt{I}_1(\eta,\alpha)=2\frac{\kappa(\alpha)}{1+\eta}\int_0^{+\infty}\frac{\xi^{2\alpha-1}}{1+\frac{\xi^2}{1+\eta}}d\xi.
\end{equation}
Thus equation \eqref{eqI1} can be simplified by defining a new variable $y=1+\frac{\xi^2}{1+\eta}$. Substituting $\xi$ by $(y-1)^{\frac{1}{2}}(1+\eta)^{\frac{1}{2}}$ in equation \eqref{eqI1} , we get 
$$
\mathtt{I}_1(\eta,\alpha)=\frac{\kappa(\alpha)}{(1+\eta)^{1-\alpha}}\int_1^{+\infty}\frac{1}{y(y-1)^{1-\alpha}}dy.
$$
Using the fact that $\alpha\in (0,1)$, its easy to see that $y^{-1}(y-1)^{\alpha-1}\in L^1(1,+\infty)$, therefore $\mathtt{I}_1(\eta,\alpha)$ is well defined. Now, for $\mathtt{I}_2(\eta,\alpha)$, using $\eta\geq 0$ and $\alpha \in (0,1)$, we get 
$$
\mathtt{I}_2(\eta,\alpha)< \int_{\R}\frac{\abs{\xi}^{2\alpha-1}}{1+\xi^2+\eta}d\xi=\frac{\mathtt{I}_1(\eta,\alpha)}{\kappa(\alpha)}<+\infty.
$$
Then, $\mathtt{I}_2(\eta,\alpha)$ is well-defined. \\
Now, for the integral $\mathtt{I}_3(\eta,\alpha)$, since 
\begin{equation*}
\frac{\xi^{2\alpha+1}}{(1+\xi^2+\eta)^2}\isEquivTo{0}\frac{\xi^{2\alpha+1}}{(1+\eta)^2}\quad \text{and}\quad \frac{\xi^{2\alpha+1}}{(1+\xi^2+\eta)^2}\isEquivTo{+\infty}\frac{1}{\xi^{3-2\alpha}},
\end{equation*}
and the fact that $\alpha\in (0,1)$, we get $\mathtt{I}_3(\eta,\alpha)$ is well-defined.\\
The proof is thus complete. 

\end{proof}

\begin{pro}\label{mdissipatif}
The unbounded linear operator $\mathcal{A}_1$ is m-dissipative in the energy space $\mathcal{H}_1$.
\end{pro}
\begin{proof}
For all $U=(u,v,y,z,\omega)\in D(\mathcal{A}_1)$, one has  
$$
\Re\left(\left<\mathcal{A}_1U,U\right>_{\mathcal{H}_1}\right)=-\kappa(\alpha)\int_{0}^{L}\int_{\mathbb{R}}(\xi^2+\eta)\abs{\omega(x,\xi)}^2d\xi dx\leq 0,
$$
which implies that $\mathcal{A}_1$ is dissipative. Now, let  $F=(f_1,f_2,f_3,f_4,f_5)\in \mathcal{H}_1$, we prove the existence of $U=(u,v,y,z,\omega)\in D(\mathcal{A}_1)$, solution of the equation 
$$
(I-\mathcal{A}_1)U=F.
$$
Equivalently, one must consider the system given by 
\begin{eqnarray}
u-v&=&f_1,\label{mdiss1}\\
v-\left(au_x+\sqrt{d(x)}\kappa(\alpha)\int_{\mathbb{R}}\abs{\xi}^{\frac{2\alpha-1}{2}}\omega(x,\xi)d\xi\right)_{x}&=&f_2,\label{mdiss2}\\
y-z&=&f_3,\label{mdiss3}\\
z+by_{xxxx}&=&f_4,\label{mdiss4}\\
(1+\xi^2+\eta)\omega(x,\xi)-\sqrt{d(x)}v_x|\xi|^{\frac{2\alpha-1}{2}}&=&f_5(x,\xi).\label{mdiss5}
\end{eqnarray}
Using Equations  \eqref{mdiss1}, \eqref{mdiss5} and the fact that $\eta\geq 0$, we get 
$$
\omega(x,\xi)=\frac{f_5(x,\xi)}{1+\xi^2+\eta}+\frac{\sqrt{d(x)}u_x\,\abs{\xi}^{\frac{2\alpha-1}{2}}}{1+\xi^2+\eta}-\frac{\sqrt{d(x)}(f_1)_x\,\abs{\xi}^{\frac{2\alpha-1}{2}}}{1+\xi^2+\eta}.
$$
Inserting the above equation and \eqref{mdiss1} in \eqref{mdiss2} and \eqref{mdiss3} in \eqref{mdiss4}, we get 
\begin{eqnarray}
\hspace{0.7cm}u-\left(au_{x}+d(x)\mathtt{I}_1(\eta,\alpha)u_x+d(x)\mathtt{I}_1(\eta,\alpha)(f_1)_x-\sqrt{d(x)}\kappa(\alpha)\int_{\mathbb{R}}\frac{\abs{\xi}^{\frac{2\alpha-1}{2}}f_5(x,\xi)}{1+\xi^2+\eta}d\xi\right)_x&=&F_1,\label{mdiss6}\\
y+by_{xxxx}&=&F_2\label{mdiss7''}
\end{eqnarray}
where $\mathtt{I}_1(\eta,\alpha)$ is defined in Equation \eqref{I123}, $F_1=f_1+f_2$ and $F_2=f_3+f_4$. And
with the following boundary conditions 
\begin{equation}\label{mdiss4'}
u(L)=y(-L)=y_{x}(-L)=y_{xx}(0)=0,\,au_{x}(0)+by_{xxx}(0)=0,\,\text{and}\, u(0)=y(0).
\end{equation}
Now, we define
$$V=\left\{(\varphi,\psi)\in H^1_R(0,L)\times H^2_L(-L,0); \ \ \varphi(0)=\psi(0)\right\}.$$
The space $V$ is equipped with the following inner product
$$
\begin{array}{lll}
\displaystyle
\left<(\varphi,\psi),(\varphi_1,\psi_1)\right>_{V}&=&\displaystyle
a\int_0^L \varphi_x\overline{\varphi_1}_xdx+b\int_{-L}^0 \psi_{xx}(\overline{\psi_1})_{xx}dx.
\end{array}
$$
Let $(\varphi,\psi)\in V$. Multiplying equations \eqref{mdiss6} and \eqref{mdiss7''}  by $\bar{\varphi}$ and $\bar{\psi}$, and integrating respectively on $(0,L) $ and $(-L,0)$, then using by parts integration, we get 
\begin{equation}\label{mdiss8}
a\left((u,y),(\varphi,\psi)\right)=L(\varphi,\psi)\quad \forall \left(\varphi,\psi\right)\in V,
\end{equation}
where 
$$
\begin{array}{l}
\displaystyle
a\left((u,y),(\varphi,\psi)\right)=\int_0^Lu\overline{\varphi}dx+a\int_0^Lu_x\overline{\varphi}_xdx+\int_{-L}^0y\overline{\psi}dx+b\int_{-L}^0y_{xx}\overline{\psi}_{xx}dx
+\mathtt{I}_1(\eta,\alpha)\int_0^Ld(x)u_{x}\overline{\varphi}_{x}dx
\end{array}
$$
and 
$$
\begin{array}{l}
\displaystyle 
L(\varphi,\psi)=\int_0^LF_1\overline{\varphi}dx+\mathtt{I}_1(\eta,\alpha)\int_0^L d(x)(f_1)_x\overline{\varphi}_xdx-\kappa(\alpha)\int_0^L\sqrt{d(x)}\overline{\varphi}_x\left(\int_{\R}\frac{\abs{\xi}^{\frac{2\alpha-1}{2}}f_5(x,\xi)}{1+\xi^2+\eta}d\xi\right)dx+\int_{-L}^0 F_2\overline{\psi}dx.
\end{array}
$$
Using the fact that $\mathtt{I}_1(\eta,\alpha)>0$, we get $a$ is a bilinear, continuous coercive form on $V\times V$. Next, by using Cauchy-Schwartz inequality and the definition of $d(x)$, we get 
\begin{equation}\label{mdiss9}
\left|\int_0^L\sqrt{d(x)}\bar{\varphi}_x\left(\int_{\R}\frac{\abs{\xi}^{\frac{2\alpha-1}{2}}f_5(x,\xi)}{1+\xi^2+\eta}d\xi\right)dx\right|\leq \sqrt{\frac{d_0}{\kappa(\alpha)}}\sqrt{\mathtt{I}_2(\eta,\alpha)}\|\varphi_x\|_{L^2(l_0,l_1)}\|f_5\|_W,
\end{equation}
where $\mathtt{I}_2(\eta,\alpha)$ is defined in Equation \eqref{I123}. Hence, L is a linear continuous form on $V$. Then, using Lax-Milgram theorem, we deduce that there exists unique $(u,y)\in V$ solution of the variational problem \eqref{mdiss8}. Applying the classical elliptic regularity, we deduce that $y\in H^4(-L,0)$, and 
$$\displaystyle{\left(au_x+\sqrt{d(x)}\kappa(\alpha)\int_{\mathbb{R}}\abs{\xi}^{\frac{2\alpha-1}{2}}\omega(x,\xi,t)d\xi\right)_{x}}\in L^2(0,L).$$
Defining 
\begin{equation}\label{mdiss7}
v:=u-f_1,\ z:=y-f_3\quad \text{and}\quad \omega(x,\xi)=\frac{f_5(x,\xi)}{1+\xi^2+\eta}+\frac{\sqrt{d(x)}u_x\,\abs{\xi}^{\frac{2\alpha-1}{2}}}{1+\xi^2+\eta}-\frac{\sqrt{d(x)}(f_1)_x\,\abs{\xi}^{\frac{2\alpha-1}{2}}}{1+\xi^2+\eta}.
\end{equation}
It is easy to see that $(v,z)\in H^1_R(0,L)\times H_L^2(-L,0)$ and $v(0)=z(0)$.\\
In order to complete the existence of $U\in D(\mathcal{A}_1)$, we need to prove $\omega(x,\xi)$ and $|\xi|\omega(x,\xi)\in W$. From equation \eqref{mdiss7}, we obtain 
\begin{equation*}
\int_{0}^{L}\int_{\R}\abs{\omega(x,\xi)}^2dx\leq 3\int_{0}^{L}\int_{\R}\frac{\abs{f_5(x,\xi)}^2}{(1+\xi^2+\eta)^2}d\xi dx+3d_0\mathtt{I}_2(\eta,\alpha)\int_{l_0}^{l_1}\left(\abs{u_x}^2+\abs{(f_1)_x}^2\right)dx.
\end{equation*}
Using Lemma \ref{lemI123}, the fact that $(u,f_1)\in H^1_R(0,L)\times H_R^1(0,L)$, we obtain 
\begin{equation*}
\mathtt{I}_2(\eta,\alpha)\int_{l_0}^{l_1}\left(\abs{u_x}^2+\abs{(f_1)_x}^2\right)dx<\infty. 
\end{equation*}
On the other hand, using the fact that $f_5\in W$, we get 
\begin{equation*}
\int_{0}^{L}\int_{\R}\frac{\abs{f_5(x,\xi)}^2}{(1+\xi^2+\eta)^2}d\xi dx\leq \frac{1}{(1+\eta)^2}\int_{0}^{L}\int_{\R}\abs{f_5(x,\xi)}^2d\xi dx<+\infty.
\end{equation*}
It follows that $\omega(x,\xi) \in W$. Next, using equation \eqref{mdiss7}, we get 
\begin{equation*}
\int_{0}^{L}\int_{\R}\abs{\xi\,\omega(x,\xi)}^2d\xi dx\leq 3\int_{0}^{L}\int_{\R}\frac{\xi^2\abs{f_5(x,\xi)}^2}{(1+\xi^2+\eta)^2}d\xi dx+6d_0\mathtt{I}_3(\eta,\alpha)\left(\int_{l_0}^{l_1}\left(\abs{u_x}^2+\abs{(f_1)_x}^2\right)dx\right),
\end{equation*}
where $\displaystyle{\mathtt{I}_3(\eta,\alpha)=\int_0^{+\infty}\frac{\xi^{2\alpha+1}}{(1+\xi^2+\eta)^2}d\xi}$. Using Lemma \ref{lemI123} we get that $\mathtt{I}_3(\eta,\alpha)$ is well-defined.\\
Now, using the fact that $f_5(x,\xi)\in W$ and 
\begin{equation*}
\max_{\xi\in\mathbb{R}}\frac{\xi^2}{(1+\xi^2+\eta)^2}=\frac{1}{4\left(1+\eta\right)}<\frac{1}{4},
\end{equation*}
we get
\begin{equation*} 
\int_{0}^{L}\int_{\R}\frac{\xi^2\abs{f_5(x,\xi)}^2}{(1+\xi^2+\eta)^2}d\xi dx\leq \max_{\xi\in\mathbb{R}}\frac{\xi^2}{(1+\xi^2+\eta)^2}\int_{0}^{L}\int_{\R}\abs{f_5(x,\xi)}^2d\xi dx<\frac{1}{4}\int_{0}^{L}\int_{\R}\abs{f_5(x,\xi)}^2d\xi dx<+\infty.
\end{equation*}
It follows that $\abs{\xi}\omega \in W$. Finally, since $\omega,f_5\in W$, we get 
\begin{equation*}
-\left(\abs{\xi}^2+\eta\right)\omega(x,\xi)+\sqrt{d(x)}v_x|\xi|^{\frac{2\alpha-1}{2}}=\omega(x,\xi)-f_5(x,\xi)\in W.
\end{equation*}
Therefore, there exists  $U:=(u,v,y,z,\omega)\in D(\mathcal{A}_1)$ solution $(I-\mathcal{A}_1)U=F$. The proof is thus complete. 
\end{proof}

\vspace{0.5cm}
\noindent From proposition \ref{mdissipatif}, the operator $\mathcal{A}_1$ is m-dissipative on $\mathcal{H}_1$, consequently it  generates a $C_0$-semigroup of contractions $(e^{t\mathcal{A}_1})_{t\geq 0}$ following Lummer-Phillips theorem (see in \cite{Pazy01} and \cite{LiuZheng01}). Then the solution of the evolution Equation \eqref{evolution-w} admits the following representation
$$
U(t)=e^{t\mathcal{A}_1}U_0,\quad t\geq 0,
$$
which leads to the well-posedness of \eqref{evolution-w}. Hence, we have the following result. 
\begin{theoreme}
Let $U_0\in \mathcal{H}_1$, then problem \eqref{evolution-w} admits a unique weak solution $U$ satisfies 
$$
U(t)\in C^0\left(\R^+,\mathcal{H}_1\right).
$$
Moreover, if $U_0\in D(\mathcal{A}_1)$, then problem \eqref{evolution-w} admits a unique strong solution $U$ satisfies 
$$
U(t)\in C^1\left(\R^+,\mathcal{H}_1\right)\cap C^0\left(\R^+,D(\mathcal{A}_1)\right).
$$
\end{theoreme}
\subsubsection{Strong Stability}\label{SS}
\noindent This part is devoted to study the strong stability
of the system. It is easy to see that the resolvent of A is not
compact. For this aim,  we use a general criteria of Arendt-Battay in \cite{Arendt01} (see Theorem \ref{arendtbatty}) to obtain the strong stability of the $C_0$-semigroup $(e^{t\mathcal{A}_1})_{t\geq 0}$. Our main result in this part is the following theorem.
\begin{theoreme}\label{Strong}
Assume that $\eta\geq 0$, then the $C_0-$semigroup of contractions $e^{t\mathcal{A}_1}$ is strongly stable on $\mathcal{H}_1$ in the sense that 
$$\lim_{t\to+\infty}\left\|e^{t\mathcal{A}_1}U_0\right\|_{\mathcal{H}_1}=0\quad\forall\ \ U_0\in\mathcal{H}_1.
$$
\end{theoreme}
\noindent In order to proof Theorem \ref{Strong} we need to prove that the operator $\mathcal{A}_1$ has no pure imaginary eigenvalues and $\sigma(\mathcal{A}_1)\cap i\R$ is countable, where $\sigma(\mathcal{A}_1)$ denotes the spectrum of $\mathcal{A}_1$. For clarity, we divide the proof into several lemmas. 
\begin{Lemma}\label{lemI2}
Let $\alpha\in (0,1)$, $\eta\geq 0$, $\la\in \R$ and $f_5\in W$. For {\rm(}$\eta>0$ and $\lambda\in \mathbb{R}${\rm)} or {\rm(}$\eta=0$ and $\lambda\in \mathbb{R}^{\ast}${\rm)}, we have
$$
\mathtt{I}_4(\la,\eta,\alpha)=i\la\kappa(\alpha)\int_{\R}\frac{\abs{\xi}^{2\alpha-1}}{i\la+\xi^2+\eta}d\xi< \infty,\quad
 \mathtt{I}_5(\la,\eta,\alpha)=\kappa(\alpha)\int_{\R}\frac{\abs{\xi}^{2\alpha-1}}{i\la+\xi^2+\eta}d\xi<\infty,$$
 and 
  $$\displaystyle\mathtt{I}_6(x,\la,\eta,\alpha):=\kappa(\alpha)\int_{\R}\frac{\abs{\xi}^{\frac{2\alpha-1}{2}}f_5(x,\xi)}{i\la+\xi^2+\eta}d\xi\in L^2(0,L)$$
\end{Lemma}
\begin{proof}
The integrals $\mathtt{I}_4$ and $\mathtt{I}_5$ can be written in the following form
$$\mathtt{I}_4(\la,\eta,\alpha)=\la^2\mathtt{I}_7(\la,\eta,\alpha)+i\la\mathtt{I}_8(\la,\eta,\alpha),\ \text{and} \ \mathtt{I}_5(\la,\eta,\alpha)=-i\la\mathtt{I}_7(\la,\eta,\alpha)+\mathtt{I}_8(\la,\eta,\alpha)$$
where
$$
\mathtt{I}_7(\la,\eta,\alpha)=\kappa(\alpha)\int_{\R}\frac{\abs{\xi}^{2\alpha-1}}{\la^2+(\xi^2+\eta)^2}d\xi,\ \text{and}\, \mathtt{I}_8(\la,\eta,\alpha)=\kappa(\alpha)\int_{\R}\frac{\abs{\xi}^{2\alpha-1}\left(\xi^2+\eta\right)}{\la^2+(\xi^2+\eta)^2}d\xi.
$$


\noindent So, we need to prove that $\mathtt{I}_7(\la,\eta,\alpha), \mathtt{I}_8(\la,\eta,\alpha)$ are well defined.\\
First, we have 
\begin{equation*}
\mathtt{I}_7(\lambda,\eta,\alpha)=2\kappa(\alpha)\int_{0}^{+\infty}\frac{\xi^{2\alpha-1}}{\lambda^2+(\xi^2+\eta)^2}d\xi=2\kappa(\alpha)\int_0^1\frac{\xi^{2\alpha-1}}{\lambda^2+(\xi^2+\eta)^2}d\xi+2\kappa(\alpha)\int_1^{+\infty}\frac{\xi^{2\alpha-1}}{\lambda^2+(\xi^2+\eta)^2}d\xi.
\end{equation*}
Hence in the both cases where {\rm(}$\eta>0$ and $\lambda\in \mathbb{R}${\rm)} or {\rm(}$\eta=0$ and $\lambda\in \mathbb{R}^{\ast}${\rm)}, we have 
\begin{equation*}
\frac{\xi^{2\alpha-1}}{\lambda^2+(\xi^2+\eta)^2}\isEquivTo{0}\frac{\xi^{2\alpha-1}}{\lambda^2+\eta^2}\quad \text{and}\quad \frac{\xi^{2\alpha-1}}{\lambda^2+(\xi^2+\eta)^2}\isEquivTo{+\infty}\frac{1}{\xi^{5-2\alpha}}.
\end{equation*}
Since $0<\alpha<1$ then $\mathtt{I}_7(\lambda,\eta,\alpha)$ is well-defined. 
Now, we have
\begin{equation*}
\mathtt{I}_8(\lambda,\eta,\alpha)=2\kappa(\alpha)\int_0^{+\infty}\frac{\xi^{2\alpha-1}(\xi^2+\eta)}{\lambda^2+(\xi^2+\eta)^2}d\xi=2\kappa(\alpha)\int_0^{1}\frac{\xi^{2\alpha-1}(\xi^2+\eta)}{\lambda^2+(\xi^2+\eta)^2}d\xi+2\kappa(\alpha)\int_1^{+\infty}\frac{\xi^{2\alpha-1}(\xi^2+\eta)}{\lambda^2+(\xi^2+\eta)^2}d\xi.
\end{equation*}
Similar to $\mathtt{I}_7$, in the both cases where {\rm(}$\eta>0$ and $\lambda\in \mathbb{R}${\rm)} or {\rm(}$\eta=0$ and $\lambda\in \mathbb{R}^{\ast}${\rm)}, we have 
\begin{equation*}
\frac{\xi^{2\alpha-1}(\xi^2+\eta)}{\lambda^2+(\xi^2+\eta)^2}\isEquivTo{0}\frac{\xi^{2\alpha-1}(\xi^2+\eta)}{\lambda^2+\eta^2}\quad \text{and}\quad \frac{\xi^{2\alpha-1}(\xi^2+\eta)}{\lambda^2+(\xi^2+\eta)^2}\isEquivTo{+\infty}\frac{1}{\xi^{3-2\alpha}}.
\end{equation*}
Since $0<\alpha<1$, then $\mathtt{I}_8(\lambda,\eta,\alpha)$ is well-defined. 
For $\mathtt{I}_6$, using Cauchy-Schwarz inequality and the fact that $f_5\in W$ and that $\mathtt{I_7}<\infty$ , we get
\begin{equation*}
\begin{array}{ll}
\displaystyle\int_0^L\left|\mathtt{I}_6(x,\la,\eta,\alpha)\right|^2dx & \displaystyle=\kappa(\alpha)^2\int_0^L\left|\int_{\R}\frac{\abs{\xi}^{\frac{2\alpha-1}{2}}f_5(x,\xi)}{i\la+\xi^2+\eta}d\xi\right|^2dx\\
&\displaystyle \leq
\kappa(\alpha)^2\left(\int_{\R}\frac{\abs{\xi}^{2\alpha-1}}{\la^2+(\xi^2+\eta)^2}d\xi\right)\int_0^L \int_{\R}\abs{f_5(x,\xi)}^2d\xi dx< +\infty.
\end{array}
\end{equation*}
The proof is thus complete. 
\end{proof}

\begin{Lemma}\label{lemI3}
Let $\alpha\in (0,1)$, $\eta\geq 0$, $\la\in\R$. For {\rm(}$\eta>0$ and $\lambda\in \mathbb{R}${\rm)} or {\rm(}$\eta=0$ and $\lambda\in \mathbb{R}^{\ast}${\rm)}, we have 
$$ \mathtt{I}_{11}(\la,\eta,\alpha)=\int_{\R}\frac{\abs{\xi}^{2\alpha-1}}{\sqrt{\la^2+(\xi^2+\eta)^2}}d\xi\quad\text{and}\quad \mathtt{I}_{12}(\la,\eta,\alpha)=\int_{\R}\frac{\abs{\xi}^{2\alpha+1}}{\la^2+(\xi^2+\eta)^2}d\xi
$$
are well-defined.
\end{Lemma}
\begin{proof}
We have
\begin{equation*} 
\mathtt{I}_{11}(\lambda,\eta,\alpha)=2\int_{0}^1\frac{\xi^{2\alpha-1}}{\sqrt{\la^2+(\xi^2+\eta)^2}}d\xi+2\int_{1}^{+\infty}\frac{\xi^{2\alpha-1}}{\sqrt{\la^2+(\xi^2+\eta)^2}}d\xi
\end{equation*}
Hence in the both cases where {\rm(}$\eta>0$ and $\lambda\in \mathbb{R}${\rm)} or {\rm(}$\eta=0$ and $\lambda\in \mathbb{R}^{\ast}${\rm)}, we have 
\begin{equation*}
\frac{\xi^{2\alpha-1}}{\sqrt{\la^2+(\xi^2+\eta)^2}}\isEquivTo{0} \frac{\xi^{2\alpha-1}}{\sqrt{\la^2+\eta^2}}\quad \text{and}\quad \frac{\xi^{2\alpha-1}}{\sqrt{\la^2+(\xi^2+\eta)^2}}\isEquivTo{+\infty}\frac{1}{\xi^{3-2\alpha}}.  
\end{equation*}
Since $0<\alpha<1$ then $\mathtt{I}_{11}(\lambda,\eta,\alpha)$ is well-defined. Now,
\begin{equation*}
\mathtt{I}_{12}(\lambda,\eta,\alpha)=2\int_0^{+\infty}\frac{\xi^{2\alpha+1}}{\lambda^2+(\xi^2+\eta)^2}d\xi=2\int_0^{1}\frac{\xi^{2\alpha-1}(\xi^2+\eta)}{\lambda^2+(\xi^2+\eta)^2}d\xi+2\int_1^{+\infty}\frac{\xi^{2\alpha-1}(\xi^2+\eta)}{\lambda^2+(\xi^2+\eta)^2}d\xi.
\end{equation*}
In a similar way, in the both cases where {\rm(}$\eta>0$ and $\lambda\in \mathbb{R}${\rm)} or {\rm(}$\eta=0$ and $\lambda\in \mathbb{R}^{\ast}${\rm)}, we have 
\begin{equation*}
\frac{\xi^{2\alpha+1}}{\lambda^2+(\xi^2+\eta)^2}\isEquivTo{0}\frac{\xi^{2\alpha+1}}{\lambda^2+\eta^2}\quad \text{and}\quad \frac{\xi^{2\alpha+1}}{\lambda^2+(\xi^2+\eta)^2}\isEquivTo{+\infty}\frac{1}{\xi^{3-2\alpha}}.
\end{equation*}
Since $0<\alpha<1$, then $\mathtt{I}_{12}(\lambda,\eta,\alpha)$ is well-defined. The proof is thus complete. 
\end{proof}

\begin{Lemma}\label{ker}
Assume that $\eta\geq 0$. Then, for all $\la\in \R$, we have $i\la I-\mathcal{A}_1$ is injective, i.e. 
$$
\ker\left(i\la I-\mathcal{A}_1\right)=\left\{0\right\}.
$$
\end{Lemma}
\begin{proof}
Let $\la\in \mathbb{R}$, such that $i\la$ be an eigenvalue of the operator $\mathcal{A}_1$ and $U=(u,v,y,z,\omega)\in D(\mathcal{A}_1)$ a corresponding eigenvector. Therefore, we have 
\begin{equation}\label{ker1}
\mathcal{A}_1U=i\la U.
\end{equation}
Equivalently, we have 
\begin{eqnarray}
v&=&i\la u,\label{ker2.1}\\
\displaystyle{\left(au_x+\sqrt{d(x)}\kappa(\alpha)\int_{\mathbb{R}}\abs{\xi}^{\frac{2\alpha-1}{2}}\omega(x,\xi)d\xi\right)_{x}}&=&i\la v, \label{ker2.2}\\
z&=&i\la y,\label{ker2.3}\\
-by_{xxxx}&=&i\la z, \label{ker2.4}\\
\left(i\la+|\xi|^2+\eta\right)\omega(x,\xi)&=&\sqrt{d(x)}v_x|\xi|^{\frac{2\alpha-1}{2}}. \label{ker2.5}
\end{eqnarray}
with the boundary conditions
\begin{equation}\label{boundary}
\left\{\begin{array}{ll}
u(L)=y(-L)=y_{x}(-L)=0,\\[0.1in]
y_{xx}(0)=0,
\end{array}\right.
\end{equation}
and with the continuity transmission conditions
\begin{equation}\label{transmission}
\left\{\begin{array}{ll}
u(0)=y(0),\\[0.1in]
au_{x}(0)=-by_{xxx}(0).
\end{array}\right.
\end{equation}
A straightforward calculation gives 
\begin{equation*}
0=\Re\left(\left<i\la U,U\right>_{\mathcal{H}_1}\right)=\Re\left(\left<\mathcal{A}_1U,U\right>_{\mathcal{H}_1}\right)=-\kappa(\alpha)\int_0^L\int_{\R}(\xi^2+\eta)|\omega(x,\xi)|^2d\xi dx.
\end{equation*}
Consequently, we deduce that
\begin{equation}\label{ker3.1}
\omega(x,\xi)=0\ \ \text{a.e. in}\ (0,L)\times \R.
\end{equation}
Inserting Equation \eqref{ker3.1} into \eqref{ker2.5} and using the definition of $d(x)$, we get
\begin{equation}\label{ker3.2}
v_x=0\quad\text{in}\quad (l_0,l_1).
\end{equation}
It follows, from Equation \eqref{ker2.1}, that
\begin{equation}\label{ker3.3}
\la u_x=0\quad\text{in}\quad (l_0,l_1).
\end{equation}
Here we will distinguish two cases.\\
\textbf{Case 1.} If $\la=0$:\\
From \eqref{ker2.1} and \eqref{ker2.3} we get
$$v=z=0\quad \text{on}\quad (0,L).$$
Using Equations \eqref{ker2.2}, \eqref{ker2.4} and \eqref{ker3.1} we get
$$u_{xx}=y_{xxxx}=0.$$
Using the boundary conditions in \eqref{boundary} we can write $u$ and $y$ as 
$$u=c_1(x-L)\quad \text{and}\quad y=c_3\left(\frac{x^3}{6}-\frac{L^2}{2}x-\frac{L^3}{3}\right)$$
where $c_1, c_3$ are constant numbers to be determined.
Now, using conditions in \eqref{transmission} we get
\begin{equation}
\left\{\begin{array}{ll}
\displaystyle c_1=\frac{L^2}{2}c_3,\\[0.1in]
\displaystyle ac_1=-bc_3.
\end{array}\right.
\end{equation}
Then, $\displaystyle c_3(a\frac{L^2}{2}+b)=0$. Since $a,b>0$, we deduce that $c_1=c_3=0$. Then we get $u=y=0$. Hence, $U=0$. In this case the proof is complete.\\
\textbf{Case 2.} If $\la\neq 0$:\\
From Equation \eqref{ker3.3}, we get
\begin{equation}\label{ker3.3'}
u_x=0\quad\text{in}\quad (l_0,l_1).
\end{equation}
Using Equations \eqref{ker3.1} and \eqref{ker3.3'} in \eqref{ker2.2}, and using Equation \eqref{ker2.1} we get
\begin{equation}\label{ker3.4}
u=0\quad\text{in}\quad (l_0,l_1).
\end{equation}
Substituting equations \eqref{ker2.1} and \eqref{ker2.3} into Equations \eqref{ker2.2} and \eqref{ker2.4} and using Equation \eqref{ker3.1}, we get
\begin{eqnarray}
\displaystyle{\la^2u+au_{xx}=0},&\text{over}&(0,L),\label{ker4.1}\\
\displaystyle{\la^2y-by_{xxxx}=0},&\text{over}&(-L,0),\label{ker4.2}
\end{eqnarray}
From Equation \eqref{ker4.1} and \eqref{ker3.4}, and using the unique continuation theorem (see \cite{Lebau}) we get
\begin{equation}\label{ker4.4}
u=0\quad\text{in}\quad (0,L).
\end{equation}
From Equation \eqref{ker4.2}, \eqref{boundary}-\eqref{transmission}, and using \eqref{ker4.4} we get the following system
\begin{equation}\label{eqys}
\left\{
\begin{array}{l}
\la^2y-by_{xxxx}=0,\  \text{over}\ (-L,0)\\ \\
y(0)=y_{xx}(0)=y_{(xxx)}(0)=0,\\ \\
y(-L)=y_x(-L)=0.
\end{array}\right.
\end{equation}
It's easy to see that $y=0$ is the unique solution of \eqref{eqys}.
Hence $U=0$. The proof is thus completed. 
\end{proof}
\begin{Lemma}\label{eta=0}
Assume that $\eta=0$. Then, the operator $-\mathcal{A}_1$ is not invertible and consequently $0\in \sigma(\mathcal{A}_1)$.
\end{Lemma}
\begin{proof}
Let $\displaystyle F=\left(\cos\left(\frac{\pi x}{2L}\right),0,0,0,0\right)\in \mathcal{H}_1$ and assume that there exists $U=(u,v,y,z,\omega)\in D(\mathcal{A}_1)$ such that $-\mathcal{A}_1U=F$. It follows that 
$$
v_x=-\dfrac{\pi}{2L}\sin\left(\frac{\pi x}{2L}\right)\ \ \text{in}\ \ (0,L)\quad \text{and}\quad \xi^2\omega(x,\xi)+\dfrac{\pi}{2L}\sqrt{d(x)}\sin\left(\frac{\pi x}{2L}\right)\abs{\xi}^{\frac{2\alpha-1}{2}}=0.$$
From the above equation, we deduce that $\displaystyle{\omega(x,\xi)=-\dfrac{\pi}{2L}\abs{\xi}^{\frac{2\alpha-5}{2}}\sqrt{d(x)}\cos\left(\frac{\pi x}{L}\right)\notin W}$, 
therefore  the assumption of the existence of $U$ is false and consequently the operator $-\mathcal{A}_1$ is not invertible. The proof is thus complete. 
\end{proof}
\begin{Lemma}\label{surjective}
If $(\eta>0\ \text{and}\ \la\in \R)$ or $(\eta=0\ \text{and}\ \la\in \R^{\ast})$, then $i\la I-\AA_1$ is surjective.
\end{Lemma}
\begin{proof}
Let $\mathrm{F} =(f_1,f_2,f_3,f_4,f_5)\in \HH$, we look for $U=(u,v,y,z,\omega)\in D(\AA_1)$ solution of 
\begin{equation}\label{surjective1}
(i\la I-\AA_1)U=\mathrm{F}. 
\end{equation}
Equivalently, we have 
\begin{eqnarray}
i\la u-v&=&f_1, \label{surj1.1}\\
\displaystyle{i\la v-\left(au_x+\sqrt{d(x)}\kappa(\alpha)\int_{\mathbb{R}}\abs{\xi}^{\frac{2\alpha-1}{2}}\omega(x,\xi)d\xi\right)_{x}}&=&f_2,\label{surj1.2}\\
i\la y-z&=&f_3,\label{surj1.3}\\
\displaystyle{i\la z+by_{xxxx} }&=&f_4,\label{surj1.4}\\
\left(i\la +\xi^2+\eta\right)\omega(x,\xi)-\sqrt{d(x)}v_x\abs{\xi}^{\frac{2\alpha-1}{2}}&=&f_5(x,\xi).\label{surj1.5}
\end{eqnarray}
Using Equations \eqref{surj1.1} and \eqref{surj1.5} and that fact that $\eta\geq 0$ we get
\begin{equation}\label{omega2}
\omega(x,\xi)=\frac{f_5(x,\xi)}{i\la+\xi^2+\eta}+\frac{\sqrt{d(x)}i\la u_x\,\abs{\xi}^{\frac{2\alpha-1}{2}}}{i\la+\xi^2+\eta}-\frac{\sqrt{d(x)}(f_1)_x\,\abs{\xi}^{\frac{2\alpha-1}{2}}}{i\la+\xi^2+\eta}.
\end{equation}
Substituting $v$ and $z$ in \eqref{surj1.1} and \eqref{surj1.3} into Equations \eqref{surj1.2} and \eqref{surj1.4}, and using \eqref{omega2} we get
\begin{eqnarray}
\la^2u+(au_x+d(x)\mathtt{I}_4(\la,\eta,\alpha)u_x-g\left(x,\la,\eta,\alpha\right))_x&=& f,\label{surj2.1}\\
\la^2y-by_{xxxx}&=&F\label{surj2.2},
\end{eqnarray}
such that 
\begin{equation*}
\left\lbrace
\begin{array}{l}
  f=-\left(f_2+i\la f_1\right)\in L^2(0,L),\\ \\ 
  g\left(x,\la,\eta,\alpha\right)=\mathtt{I}_5(\la,\eta,\alpha)d(x)(f_1)_x-\sqrt{d(x)}\mathtt{I}_6(x,\la,\eta,\alpha),\\ \\
F=-\left(f_4+i\la f_3\right)\in L^2(-L,0),
\end{array}
\right.
\end{equation*}
and $\mathtt{I}_4(\la,\eta,\alpha)$, $\mathtt{I}_5(\la,\eta,\alpha)$ and $\mathtt{I}_6(\la,\eta,\alpha)$ are defined in Lemma \ref{lemI2}.

\noindent Now, we distinguish two cases:\\[0.1in]
\textbf{Case 1:} $\eta>0$ and $\la=0$, then System \eqref{surj2.1}-\eqref{surj2.2} becomes 
\begin{eqnarray*}
(au_x-g\left(x,0,\eta,\alpha\right))_x&=&-f_2,\\
by_{xxxx}&=&f_4.\\
\end{eqnarray*}
By applying Lax-Milgram theorem, and using Lemma \ref{lemI2} it is easy to see that the above system has a unique strong solution $(u,y)\in V$. \\[0.1in]
\textbf{Case 2:} $\eta\geq 0$ and $\la\in \R^{\ast}$. The system  \eqref{surj2.1}-\eqref{surj2.2} becomes
\begin{eqnarray}
\la^2u+(au_x+d(x)\mathtt{I}_4(\la,\eta,\alpha)u_x)_x&=&G,\label{sj1}\\
\la^2y-by_{xxxx}&=&F\label{sj2},
\end{eqnarray}
such that 
$$G=f+g_x(x,\la,\eta,\alpha).$$
We first define the linear unbounded operator $\mathcal{L}:\mathbb{H}:=H^1_R(0,L)\times H_L^2(-L,0)\longmapsto \mathbb{H}^{\prime}$ where $\mathbb{H}^{\prime}$ is the dual space of $\mathbb{H}$ by 
$$
\mathcal{L}\mathrm{U}=\begin{pmatrix}
-\left(au_x+d(x)\mathtt{I}_4(\la,\eta,\alpha)u_x\right)_x\\[0.1in]
by_{xxxx}
\end{pmatrix},\quad \forall\ \mathrm{U}\in \mathbb{H}. 
$$
Thanks to Lax-Milgram theorem, it is easy to see that $\mathcal{L}$ is isomorphism. The system \eqref{sj1}-\eqref{sj2} is equivalent to
\begin{equation}\label{surjective7}
\left(\la^2\mathcal{L}^{-1}-I\right)\mathrm{U}=\mathcal{L}^{-1}\mathcal{F},\quad \text{where}\  \mathrm{U}=(u,y)^{\top}\ \text{and}\ \mathcal{F}=(G,F)^{\top}.
\end{equation}
Since the operator $\mathcal{L}^{-1}$ is isomorphism and $I$ is a compact operator from $\mathbb{H}$ to $\mathbb{H}^{\prime}$. Then, $\mathcal{L}^{-1}$ is compact operator from $\mathbb{H}$ to $\mathbb{H}$.
Consequently, by Fredholm's alternative, proving the existence of $\mathrm{U}$ solution of \eqref{surjective7} reduces to proving $\ker\left(\la^2\mathcal{L}^{-1}-I\right)=\left\{0\right\}$. Indeed, if $(\tilde{u},\tilde{y})\in \ker\left(\la^2\mathcal{L}^{-1}-I\right)$, then $\la^2(\tilde{u},\tilde{y})-\mathcal{L}\left(\tilde{u},\tilde{y}\right)=0$. It follows that, 
\begin{eqnarray}
\la^2\tilde{u}+\left(a\tilde{u}_x+d(x)\mathtt{I}_4(\la,\eta,\alpha)\tilde{u}_x\right)_x=0,\label{surjective8}\\
\la^2\tilde{y}-b\tilde{y}_{xxxx}=0\label{surjective9},\\
\tilde{u}(L)=\tilde{y}(-L)=\tilde{y}_x(-L)=\tilde{y}_{xx}(0)=0, \label{surjective10} \\ 
a\tilde{u}_x(0)+b\tilde{y}_{xxx}(0)=0,\tilde{u}(0)=\tilde{y}(0).\label{surjective11}
\end{eqnarray}
Multiplying \eqref{surjective8} and \eqref{surjective9} by $\overline{\tilde{u}}$ and $\overline{\tilde{y}}$ respectively, integrating over $(0,L)$ and $(-L,0)$ respectively and taking the sum, then using by parts integration and the boundary conditions \eqref{surjective10}-\eqref{surjective11}, and take the imaginary part we get 
\begin{equation*}
d_0\Im\left(\mathtt{I}_4(\la,\alpha,\eta)\right)\int_{l_0}^{l_1}\abs{\tilde{u}_x}^2dx=0. 
\end{equation*}
From Lemma \ref{lemI2} we have $\Im\left(\mathtt{I}_4(\la,\alpha,\eta)\right)=\la \mathtt{I}_8(\la,\eta,\alpha)\neq 0$,
we get $\tilde{u}_x=0$ in $\left(l_0,l_1\right)$.\\
Then, system \eqref{surjective8}-\eqref{surjective10} becomes 
\begin{eqnarray}
\la^2\tilde{u}+a\tilde{u}_{xx} &=&0 \quad \text{over}\quad(0,L)\label{surjective12},\\
\la^2\tilde{y}-b\tilde{y}_{xxxx} &=& 0\quad \text{over}\quad(-L,0)\label{surjective13},\\
\tilde{u}_x &=0& \quad \text{over}\quad (l_0,l_1).\label{surjective14}
\end{eqnarray}
It is now easy to see that if $(\tilde{u},\tilde{y})$ is a solution of system \eqref{surjective12}-\eqref{surjective14}, then the vector $\widetilde{U}$ defined by $\widetilde{U}:=(\tilde{u},i\la \tilde{u},\tilde{y},i\la \tilde{y},0)$ belongs to $D\left(\mathcal{A}_1\right)$, and $i\la \widetilde{U}-\mathcal{A}\widetilde{U}=0$.\\
Therefore, $\widetilde{U}\in \ker\left(i\la I-\mathcal{A}_1\right)$, then by using Lemma \ref{ker}, we get $\widetilde{U}=0$. This implies that system \eqref{surjective7} admits a unique solution due to Fredholm's alternative, hence \eqref{surjective7} admits a unique solution in $V$. Thus, we define $v:=i\la u-f_1$, $z:=i\la y-f_3$ and 
\begin{equation}\label{surjective16}
\omega(x,\xi)=\frac{f_5(x,\xi)}{i\la+\xi^2+\eta}+\frac{\sqrt{d(x)}i\la u_x\,\abs{\xi}^{\frac{2\alpha-1}{2}}}{i\la+\xi^2+\eta}-\frac{\sqrt{d(x)}(f_1)_x\,\abs{\xi}^{\frac{2\alpha-1}{2}}}{i\la+\xi^2+\eta}.
\end{equation}
Since $\mathrm{F}\in \mathcal{H}_1$, it is easy to see that $v\in H^1_R(0,L)$, $z\in H_L^2(-L,0)$, $v(0)=z(0)$ and 
$$\displaystyle{\left(au_x+\sqrt{d(x)}\kappa(\alpha)\int_{\mathbb{R}}\abs{\xi}^{\frac{2\alpha-1}{2}}\omega(x,\xi,t)d\xi\right)_{x}\in L^2(0,L)}.$$
 It is left to prove that $\omega$ and $\abs{\xi}\omega\in W$ (for the both cases). From equation \eqref{surjective16}, we get 
\begin{equation*}
\int_{0}^L\int_{\R}\abs{\omega(x,\xi)}^2d\xi dx\leq 3\int_0^L\int_{\R}\frac{\abs{f_5(x,\xi)}^2}{\la^2+(\xi^2+\eta)^2}d\xi dx+3a_0\left(\int_{l_0}^{l_1}(\abs{\la u_x}^2+\abs{(f_1)_x}^2)dx\right)\mathtt{I}_{7}(\la,\eta,\alpha).
\end{equation*}
Using the fact that $f_5\in W$ and $(\eta>0\ \text{and}\ \la\in \R)$ or $(\eta=0\ \text{and}\ \la\in \R^{\ast})$, we obtain 
$$
\int_{0}^L\int_{\R}\frac{\abs{f_5(x,\xi)}^2}{\la^2+(\xi^2+\eta)^2}d\xi\leq \frac{1}{\la^2+\eta^2}\int_{0}^{L}\int_{\R}\abs{f_5(x,\xi)}^2d\xi dx<+\infty. 
$$
Using Lemma \ref{lemI2}, it follows that $\omega \in W$. Next, using equation \eqref{surjective16}, we get 
$$
\int_0^L\int_{\R}\abs{\xi \omega}^2d\xi\leq 3\int_0^L\int_{\R}\frac{\xi^2\abs{f_5(x,\xi)}^2}{\la^2+(\xi^2+\eta)^2}d\xi dx+3a_0\int_0^L\left(|\la u_x|^2+\abs{(f_1)_x}^2\right)\mathtt{I}_{12}(\la,\eta,\alpha), 
$$ 
where $\displaystyle{\mathtt{I}_{12}(\la,\eta,\alpha)=\int_{\R}\frac{\abs{\xi}^{2\alpha+1}}{\la^2+(\xi^2+\eta)^2}d\xi}<+\infty$ by using Lemma  \ref{lemI3}. Now, using the fact that $f_5\in W$ and 
$$
\max_{\xi\in \R}\frac{\xi^2}{\lambda^2+(\xi^2+\eta)^2}=\frac{\sqrt{\eta^2+\la^2}}{\la^2+\left(\sqrt{\eta^2+\lambda^2}+\eta\right)^2}=C(\la,\eta),
$$
we get 
$$
\int_0^L\int_{\R}\frac{\xi^2\abs{f_5(x,\xi)}^2}{\la^2+(\xi^2+\eta)^2}d\xi dx\leq \int_0^L\max_{\xi\in \R}\frac{\xi^2}{\la^2+(\xi^2+\eta)^2}\int_{\R}\abs{f_5(x,\xi)}^2d\xi=C(\la,\eta)\int_0^L\int_{\R}\abs{f_5(x,\xi)}^2d\xi dx<+\infty. 
$$
It follows that $\abs{\xi}\omega\in W$. Finally, since $\omega\in W$, we get 
$$
-(\xi^2+\eta)\omega(x,\xi)+\sqrt{d(x)}v_x\abs{\xi}^{\frac{2\alpha-1}{2}}=i\la \omega(x,\xi)-f_5(x,\xi)\in W. 
$$
Thus, we. obtain $U=(u,v,y,z,\omega)\in D(\AA_1)$ solution of $(i\la I-\mathcal{A}_1)U=\mathrm{F}$. The proof is thus. complete.
\end{proof}
$\newline$ 
\textbf{Proof of Theorem \ref{Strong}.} First, using Lemma \ref{ker}, we directly deduce that $\AA_1$ has no pure imaginary eigenvalues. Next, using Lemmas \ref{eta=0}, \ref{surjective} and with the help of the closed graph theorem of Banach, we deduce that $\sigma(\AA_1)\cap i\R=\{\emptyset\}$ if $\eta>0$ and $\sigma(\AA_1)\cap i\R=\{0\}$ if $\eta=0$. Thus, we get the conclusion by Applying theorem \ref{arendtbatty} of Arendt Batty.

\subsection{Polynomial Stability in the case $\eta>0$}\label{Section-poly}
\noindent In this section, we study the polynomial stability of the system \eqref{AUG1}-\eqref{AUG3} in the case $\eta>0$. For this purpose, we will use a frequency domain approach method, namely we will use  Theorem \ref{bt}. Our main result in this section is the following theorem.


\begin{theoreme}\label{pol}
Assume that $\eta>0$. The $C_0-$semigroup $(e^{t\AA_1})_{t\geq 0}$ is polynomially stable; i.e. there exists constant $C_1>0$ such that for every $U_0\in D(\AA_1)$, we have 
\begin{equation}\label{Energypol}
E_1(t)\leq \frac{C_1}{t^{\frac{4}{2-\alpha}}}\|U_0\|^2_{D(\AA_1)},\quad t>0,\,\forall U_0\in D(\mathcal{A}_1).
\end{equation}
\end{theoreme}
According to Theorem \ref{bt}, by taking $\ell=1-\frac{\alpha}{2}$, the polynomial energy decay \eqref{Energypol} holds if the following conditions 
\begin{equation}\label{H1}\tag{${\rm{H_1}}$}
i\R\subset \rho(\mathcal{A}_1),
\end{equation}
and
\begin{equation}\label{H2}\tag{${\rm{H_2}}$}
\sup_{\la\in \R}\left\|(i\la I-\AA_1)^{-1}\right\|_{\mathcal{L}(\mathcal{H}_1)}=O\left(\abs{\la}^{1-\frac{\alpha}{2}}\right)
\end{equation}
are satisfied. Since Condition \eqref{H1} is already proved in Lemma \ref{ker}. We will prove condition \eqref{H2} by an argument of contradiction. For this purpose, suppose that \eqref{H2} is false,  then there exists $\left\{\left(\la_n,U_n:=(u_n,v_n,y_n,z_n,\omega_n(\cdot,\xi))^\top\right)\right\}\subset \R^{\ast}\times D(\AA_1)$ with 
\begin{equation}\label{pol1}
\abs{\la_n}\to +\infty \quad \text{and}\quad \|U_n\|_{\mathcal{H}_1}=\|(u_n,v_n,y_n,z_n,\omega_n(\cdot,\xi))\|_{\mathcal{H}_1}=1, 
\end{equation}
such that 
\begin{equation}\label{pol2-w}
\left(\la_n\right)^{1-\frac{\alpha}{2}}\left(i\la_nI-\AA_1\right)U_n=F_n:=(f_{1,n},f_{2,n},f_{3,n},f_{4,n},f_{5,n}(\cdot,\xi))^{\top}\to 0 \ \ \text{in}\ \ \mathcal{H}_1. 
\end{equation}
For simplicity, we drop the index $n$. Equivalently, from \eqref{pol2-w}, we have 
\begin{eqnarray}
i\la u-v&=&\dfrac{f_1}{\la^{1-\frac{\alpha}{2}}}  \quad\text{in}\ H_R^1(0,L),\label{pol3}\\
i\la v-\left(S_d\right)_x&=&\dfrac{f_2}{\la^{1-\frac{\alpha}{2}}} \quad \text{in}\ L^2(0,L),\label{pol4}\\
i\la y-z&=&\dfrac{f_3}{\la^{1-\frac{\alpha}{2}}} \quad\text{in}\ H_L^2(-L,0),\label{pol5}\\
i\la z+by_{xxxx}&=&\dfrac{f_4}{\la^{1-\frac{\alpha}{2}}} \quad \text{in}\ L^2(-L,0),\label{pol6}\\
(i\la+\xi^2+\eta)\omega(x,\xi)-\sqrt{d(x)}v_x|\xi|^{\frac{2\alpha-1}{2}}&=&\dfrac{f_5(x,\xi)}{\la^{1-\frac{\alpha}{2}}} \quad \text{in}\ W,\label{pol7}
\end{eqnarray}
where $\displaystyle S_d=au_x+\sqrt{d(x)}\kappa(\alpha)\int_{\mathbb{R}}|\xi|^{\frac{2\alpha-1}{2}}\omega(x,\xi)d\xi$.\\
Here we will check the condition \eqref{H2} by finding a contradiction with \eqref{pol1} by showing $\|U\|_{\mathcal{H}_1}=o(1)$. For clarity, we divide the proof into several Lemmas.
\begin{Lemma}\label{lemI4}
Let $\alpha\in (0,1)$, $\eta>0$ and $\la\in \R$, then
$$
\left\{\begin{array}{l}
\displaystyle
\mathtt{I}_{12}(\la,\eta,\alpha)=\int_{\R}\frac{\abs{\xi}^{\alpha+\frac{1}{2}}}{\left(\abs{\la}+\xi^2+\eta\right)^2}d\xi=c_1\left(\abs{\la}+\eta\right)^{\frac{\alpha}{2}-\frac{5}{4}},\\[0.1in]
\displaystyle
\mathtt{I}_{13}(\la,\eta)=\left(\int_{\R}\frac{1}{(\abs{\la}+\xi^2+\eta)^2}d\xi\right)^{\frac{1}{2}}=\sqrt{\frac{\pi}{2}}\frac{1}{(\abs{\la}+\eta)^{\frac{3}{4}}},\\[0.1in]
\displaystyle
\mathtt{I}_{14}(\la,\eta)=\left(\int_{\R}\frac{\xi^2}{\left(\abs{\la}+\xi^2+\eta\right)^4}d\xi\right)^{\frac{1}{2}}=\frac{\sqrt{\pi}}{4}\frac{1}{(\abs{\la}+\eta)^{\frac{5}{4}}}
\end{array}
\right.
$$
where $\displaystyle c_1=\int_{1}^{\infty}\frac{\left(y-1\right)^{\frac{\alpha}{2}-\frac{1}{4}}}{y^2}dy$.
\end{Lemma}
\begin{proof}
$\mathtt{I}_{12}$ can be written as
\begin{equation}\label{E==(B.2)}
\mathtt{I}_{12}(\lambda,\eta,\alpha)=\frac{2}{\left(\lambda+\eta\right)^{2}}\int_{0}^{\infty}\frac{ \xi^{\alpha+\frac{1}{2}}}{\left(1+\frac{\xi^2}{|\lambda|+\eta}\right)^{2}}d\xi.
\end{equation}
Thus, equation \eqref{E==(B.2)} may be simplified by defining a new variable $y=1+\frac{\xi^2}{\lambda+\eta}$.  Substituting $\xi$  by $\left(y-1\right)^{\frac{1}{2}}\left(\lambda+\eta\right)^{\frac{1}{2}}$ in equation \eqref{E==(B.2)}, we get
\begin{equation*}
\mathtt{I}_{12}(\lambda,\eta,\alpha)=\left(\lambda+\eta\right)^{\frac{\alpha}{2}-\frac{5}{4}}\int_{1}^{\infty}\frac{\left(y-1\right)^{\frac{\alpha}{2}-\frac{1}{4}}}{y^2}dy.
\end{equation*}
Using the fact that $\alpha\in ]0,1[$, it is easy to see that $y^{-2}\left(y-1\right)^{\frac{\alpha}{2}-\frac{1}{4}}\in L^1(1,+\infty)$. Hence, the last integral in the above equation is well defined.  
Now, $\mathtt{I}_{13}(\lambda,\eta)$ can be written as 
\begin{equation*}
\left(\mathtt{I}_{13}(\lambda,\eta)\right)^2=\frac{2}{(\lambda+\eta)^2}\int_0^{\infty}\frac{1}{\left(1+\left(\frac{\xi}{\sqrt{\lambda+\eta}}\right)^2\right)^2}d\xi=\frac{2}{(\lambda+\eta)^{\frac{3}{2}}}\int_0^{\infty}\frac{1}{\left(1+s^2\right)^2}ds=\frac{2}{(\lambda+\eta)^{\frac{3}{2}}}\times\frac{\pi}{4}=\frac{\pi}{2(\lambda+\eta)^{\frac{3}{2}}},
\end{equation*}
Therefore,\, $\mathtt{I}_{13}(\lambda,\eta)=\displaystyle{\sqrt{\frac{\pi}{2}}\frac{1}{(\lambda+\eta)^{\frac{3}{4}}}}$. Finally, $\mathtt{I}_{14}(\lambda,\eta)$ can be written as 
\begin{equation*}
\left(\mathtt{I}_{14}(\lambda,\eta)\right)^2=\frac{2}{(\lambda+\eta)^4}\int_0^{\infty}\frac{\xi^2}{\left(1+\left(\frac{\xi}{\sqrt{\lambda+\eta}}\right)^2\right)^4}d\xi=\frac{2}{(\lambda+\eta)^{\frac{5}{2}}}\int_0^{\infty}\frac{s^2}{\left(1+s^2\right)^4}ds=\frac{2}{(\lambda+\eta)^{\frac{5}{2}}}\times \frac{\pi}{32}.
\end{equation*}
Then $\mathtt{I}_{14}(\lambda,\eta)=\displaystyle{\frac{\sqrt{\pi}}{4}\frac{1}{(\lambda+\eta)^{\frac{5}{4}}}}$. The proof has been completed. 
\end{proof}
\begin{Lemma}\label{First-Estimation}
Assume that $\eta>0$. Then, the solution $(u,v,y,z,\omega)\in D(\AA_1)$ of system \eqref{pol3}-\eqref{pol7} satisfies the following asymptotic behavior estimations 
\begin{equation}\label{FE-1}
\int_0^L\int_{\R}\left(\abs{\xi}^2+\eta\right)\abs{\omega(x,\xi)}^2d\xi dx=o\left(\la^{-1+\frac{\alpha}{2}}\right)\, ,\,\int_{l_0}^{l_1}\abs{v_x}^2dx=o\left(\la^{-\frac{\alpha}{2}}\right)\, \text{and} \,\,
 \int_{l_0}^{l_1}\abs{u_x}^2dx=o\left(\la^{-2-\frac{\alpha}{2}}\right).
\end{equation}
\end{Lemma}
\begin{proof}
For clarity, we divide the proof into several steps.\\ 
\textbf{Step 1.} Taking the inner product of $F$ with $U$ in $\mathcal{H}_1$, then using \eqref{pol1} and the fact that $U$ is uniformly bounded in $\mathcal{H}_1$, we get 
$$
\kappa(\alpha)\int_0^L\int_{\R}\left(\xi^2+\eta\right)\abs{\omega(x,\xi)}^2d\xi dx=-\Re\left(\left<\AA_1 U,U\right>_{\mathcal{H}_1}\right)=\Re\left(\left<(i\la I-\AA)U,U\right>_{\mathcal{H}_1}\right)=o\left(\la^{-1+\frac{\alpha}{2}}\right).
$$
\textbf{Step 2.} Our aim here is to prove the second estimation in \eqref{FE-1}.\\
From \eqref{pol7}, we get 
$$
\sqrt{d(x)}\abs{\xi}^{\frac{2\alpha-1}{2}}\abs{v_x}\leq \left(\abs{\la}+\xi^2+\eta\right)\abs{\omega(x,\xi)}+\abs{\la}^{-1+\frac{\alpha}{2}}\abs{f_5(x,\xi)}.
$$
Multiplying the above inequality by $\left(\abs{\la}+\xi^2+\eta\right)^{-2}\abs{\xi}$, integrate over $\R$, we get 
\begin{equation}\label{FE-2}
\sqrt{d(x)}\mathtt{I}_{12}(\la,\eta,\alpha)\abs{v_x}\leq \mathtt{I}_{13}(\la,\eta)\left(\int_{\R}\abs{\xi\omega(x,\xi)}^2d\xi\right)^{\frac{1}{2}}+\abs{\la}^{-1+\frac{\alpha}{2}}\mathtt{I}_{14}(\la,\eta)\left(\int_{\R}\abs{f_5(x,\xi)}^2d\xi\right)^{\frac{1}{2}},
\end{equation}
where
$\mathtt{I}_{12}(\la,\eta,\alpha), \mathtt{I}_{13}(\la,\eta)\,\text{and}\,\mathtt{I}_{14}(\la,\eta)$ are defined in Lemma \ref{lemI4}.
Using Young's inequality and the definition of the function $d(x)$ in \eqref{FE-2}, we get
\begin{equation*}
\int_{l_0}^{l_1}\abs{v_x}^2dx\leq 2\frac{\mathtt{I}_{13}^2}{\mathtt{I}_{12}^2}\frac{o(1)}{\abs{\la}^{1-\frac{\alpha}{2}}}+2\frac{\mathtt{I}_{14}^2}{\mathtt{I}_{12}^2}\frac{o(1)}{\abs{\la}^{2-\alpha}}.
\end{equation*}
It follows from Lemma \ref{lemI4} that
\begin{equation}\label{FE-3}
\int_{l_0}^{l_1}\abs{v_x}^2dx\leq \frac{1}{c_1(\abs{\la}+\eta)^{\alpha-1}}\frac{o\left(1\right)}{\abs{\la}^{1-\frac{\alpha}{2}}}+\frac{\sqrt{\pi}}{4}\frac{1}{c_1\left(\abs{\la}+\eta\right)^{\alpha}}\frac{o(1)}{\abs{\la}^{2-\alpha}}. 
\end{equation}
Since $\alpha\in (0,1)$, we have $\min(\frac{\alpha}{2},2)=\frac{\alpha}{2}$, hence from the above equation, we get the second desired estimation in \eqref{FE-1}.\\
\textbf{Step 3.}  From Equation \eqref{pol3} we have
$$i\la u_x=v_x-\la^{-1+\frac{\alpha}{2}}(f_1)_x.$$
It follows that
$$
\displaystyle{\|\la u_x\|_{L^2(l_0,l_1)}\leq \|v_x\|_{L^2(l_0,l_1)}+|\la|^{-1+\frac{\alpha}{2}}\|(f_1)_x\|_{L^2(l_0,l_1)}\leq \dfrac{o(1)}{\la^{\frac{\alpha}{4}}}+\dfrac{o(1)}{\la^{1-\frac{\alpha}{2}}}}.
 $$
Since $\alpha\in (0,1)$, we have $\min\left(1+\frac{\alpha}{4},2-\frac{\alpha}{2}\right)=1+\frac{\alpha}{4}$, hence from the above equation, we get
$$  \int_{l_0}^{l_1}\abs{u_x}^2dx=\dfrac{o(1)}{\la^{2+\frac{\alpha}{2}}}.$$
The proof is thus completed.
\end{proof}
\begin{Lemma}\label{S-Estimation}
Let $0<\alpha<1$ and $\eta >0$. Then, the solution $(u,v,y,z,\omega)\in D(\AA_1)$ of system \eqref{pol3}-\eqref{pol7} satisfies the following asymptotic behavior 
\begin{equation}\label{eqS1-w}
\int_{l_0}^{l_1}\left|S_d\right|^2dx=\frac{o(1)}{\la^{1-\frac{\alpha}{2}}}.
\end{equation}
\end{Lemma}
\begin{proof}
Using the fact that $|P+Q|^2\leq 2P^2+2Q^2$, we obtain
\begin{equation*}
\begin{array}{lll}
\displaystyle
\int_{l_0}^{l_1}\left|S_d\right|^2dx&\displaystyle=\int_{l_0}^{l_1}\left|au_x+\sqrt{d(x)}\kappa(\alpha)\int_{\mathbb{R}}|\xi|^{\frac{2\alpha-1}{2}}\omega(x,\xi)d\xi\right|^2 dx\\
&\displaystyle \leq 2a^2 \int_{l_0}^{l_1}\abs{u_x}^2dx
\displaystyle
+2d_0\kappa(\alpha)^2\int_{l_0}^{l_1}\left(\int_{\R}\frac{\abs{\xi}^{\frac{2\alpha-1}{2}}\sqrt{\xi^2+\eta}}{\sqrt{\xi^2+\eta}}\omega(x,\xi)d\xi\right)^2 dx\\
&\displaystyle \leq 2a^2 \int_{l_0}^{l_1}\abs{u_x}^2dx+c_2 \int_{l_0}^{l_1}\int_{\R}(\xi^2+\eta)\abs{\omega(x,\xi)}^2d\xi dx
\end{array}
\end{equation*}
where $\displaystyle{c_2=d_0\kappa(\alpha)^2\,\mathtt{I}_{15}(\alpha,\eta)}$ and $\displaystyle{\mathtt{I}_{15}(\alpha,\eta)=\int_{\R}\frac{\abs{\xi}^{2\alpha-1}}{\abs{\xi}^2+\eta}d\xi}$. We have  
$$
\frac{\abs{\xi}^{2\alpha-1}}{\abs{\xi}^2+\eta}\isEquivTo{0} \frac{\abs{\xi}^{2\alpha-1}}{\eta}\quad \text{and}\quad \frac{\abs{\xi}^{2\alpha-1}}{\abs{\xi}^2+\eta}\isEquivTo{+\infty}\frac{1}{\abs{\xi}^{3-2\alpha}}.
$$
Since $0<\alpha<1$ and $\eta>0$, then  $\mathtt{I}_{15}(\alpha,\eta)$ is well defined. Using the first and the third estimations in \eqref{FE-1}, we get our desired result.
\end{proof}
\begin{Lemma}\label{vs}
Assume that $\eta>0$. Let $g\in C^1([l_0,l_1])$ such that
\begin{equation*}
g(l_1)=-g(l_0)=1,\quad \max_{x\in(l_0,l_1)}\abs{g(x)}=m_g\quad\text{and}\quad \max_{x\in(l_0,l_1)}\abs{g^{\prime}(x)}=m_{g^{\prime}}
\end{equation*}
where $m_g$ and $m_{g^{\prime}}$ are strictly positive constant numbers. Then, the solution $(u,v,y,z,\omega)\in D(\AA_1)$ of system \eqref{pol3}-\eqref{pol7} satisfies the following asymptotic behavior
\begin{equation}\label{BoundaryEst1'}
\abs{v(l_1)}^2+\abs{v(l_0)}^2\leq \left(\frac{\la^{1-\frac{\alpha}{2}}}{2}+2m_{g^{\prime}}\right)\int_{\ell_0}^{l_1}\abs{v}^2dx+\frac{o(1)}{\la}
\end{equation}
and 
\begin{equation}\label{BoundaryEst2'}
\abs{S_d(l_1)}^2+\abs{S_d(l_0)}^2\leq  \frac{\la^{1+\frac{\alpha}{2}}}{2}\int_{l_0}^{l_1}\abs{v}^2dx+o(1).
\end{equation}
\end{Lemma}
\begin{proof}
First we will prove Equation \eqref{BoundaryEst1'}.
From Equation \eqref{pol3}, we have
\begin{equation}\label{eq1-v}
v_x=i\la u_x-\la^{-1+\frac{\alpha}{2}}(f_1)_x.
\end{equation}
Multiply Equation \eqref{eq1-v} by $2g\bar{v}$ and integrate over $(l_0,l_1)$, we get
\begin{equation}
\abs{v(l_1)}^2+\abs{v(l_0)}^2=\int_{l_0}^{l_1}g^{\prime}\abs{v}^2dx+\Re\left(2i\la\int_{l_0}^{l_1}u_xg\bar{v}dx\right)-\Re\left(2\la^{-1+\frac{\alpha}{2}}\int_{l_0}^{l_1}(f_1)_xg\bar{v}dx\right).
\end{equation}
Then, 
\begin{equation}\label{eq2-v'}
\abs{v(l_1)}^2+\abs{v(l_0)}^2\leq m_{g^{\prime}}\int_{l_0}^{l_1}\abs{v}^2dx+2\la m_g\int_{l_0}^{l_1}\abs{u_x}\abs{\bar{v}}dx+2m_g\la^{-1+\frac{\alpha}{2}}\int_{l_0}^{l_1}\abs{(f_1)_x} \abs{\bar{v}}dx.
\end{equation}
Using Young's inequality we have
\begin{equation}\label{Young1}
2\la m_g\abs{u_x}\abs{\bar{v}}\leq \frac{\la^{1-\frac{\alpha}{2}}}{2}\abs{v}^2+2\la^{1+\frac{\alpha}{2}}m_g^2\abs{u_x}^2\quad\text{and}\quad 2m_g\la^{-1+\frac{\alpha}{2}}\abs{(f_1)_x} \abs{\bar{v}}\leq m_{g^{\prime}}\abs{v}^2+\frac{m_g^2}{m_{g^{\prime}}
}\la^{-2+\alpha}\abs{(f_1)_x}^2.
\end{equation}
Using Equation \eqref{Young1}, then Equation \eqref{eq2-v'} becomes
\begin{equation}
\abs{v(l_1)}^2+\abs{v(l_0)}^2\leq \left(\frac{\la^{1-\frac{\alpha}{2}}}{2}+2m_{g^{\prime}}\right)\int_{l_0}^{l_1}\abs{v}^2dx+2\la^{1+\frac{\alpha}{2}} m_g^2\int_{l_0}^{l_1}\abs{u_x}^2dx+\frac{m_g^2}{m_{g^{\prime}}}\la^{-2+\alpha}\int_{l_0}^{l_1}\abs{(f_1)_x}^2 .
\end{equation}
Using the third estimation in Equation \eqref{FE-1} and the fact that $\|(f_1)_x\|_{L^2(l_0,l_1)}=o(1)$, we obtain
\begin{equation}
\abs{v(l_1)}^2+\abs{v(l_0)}^2\leq \left(\frac{\la^{1-\frac{\alpha}{2}}}{2}+2m_{g^{\prime}}\right)\int_{l_0}^{l_1}\abs{v}^2dx+\frac{o(1)}{\la}+\frac{o(1)}{\la^{2-\alpha}}.
\end{equation}
Since $\alpha\in(0,1)$, hence
\begin{equation}\label{BoundaryEst1}
\abs{v(l_1)}^2+\abs{v(l_0)}^2\leq \left(\frac{\la^{1-\frac{\alpha}{2}}}{2}+2m_{g^{\prime}}\right)\int_{l_0}^{l_1}\abs{v}^2dx+\frac{o(1)}{\la}.
\end{equation}
Now, we will prove \eqref{BoundaryEst2'}. For this aim, multiply Equation \eqref{pol4} by $-2g\bar{S}_d$ and integrate over $(l_0,l_1)$, we get
\begin{equation}
\abs{S_d(l_1)}^2+\abs{S_d(l_0)}^2=\int_{l_0}^{l_1}g^{\prime}\abs{S_d}^2dx+\Re\left(2i\la\int_{l_0}^{l_1}vg\bar{S}_ddx\right)-\Re\left(2\la^{-1+\frac{\alpha}{2}}\int_{l_0}^{l_1}f_2g\bar{S}_ddx\right).
\end{equation}
Then, 
\begin{equation}\label{eq2-v}
\abs{S_d(l_1)}^2+\abs{S_d(l_0)}^2\leq m_{g^{\prime}}\int_{l_0}^{l_1}\abs{S_d}^2dx+2\la m_g\int_{l_0}^{l_1}\abs{v}\abs{S_d}dx+2m_g\la^{-1+\frac{\alpha}{2}}\int_{l_0}^{l_1}\abs{f_2} \abs{S_d}dx.
\end{equation}
Using Young's inequality and Equation \eqref{eqS1-w} we obtain
\begin{equation}\label{Young3}
2\la m_g\abs{v}\abs{\bar{S}_d}\leq \frac{\la^{1+\frac{\alpha}{2}}}{2}\abs{v}^2+2m_g^2\la^{1-\frac{\alpha}{2}}\abs{S_d}^2\leq \frac{\la^{1+\frac{\alpha}{2}}}{2}\abs{v}^2+o(1).
\end{equation}
Using Cauchy-Schwarz inequality, Equation \eqref{eqS1-w} and the fact that $\|f_2\|_{L^2(l_0,l_1)}=o(1)$, we obtain
\begin{equation}\label{Young4}
2m_g \la^{-1+\frac{\alpha}{2}}\int_{l_0}^{l_2}\abs{f_2}\abs{S_d}dx\leq \frac{1}{\la^{1-\frac{\alpha}{2}}}\|f_2\|_{L^2(l_0,l_1)}\|S\|_{L^2(l_0,l_1)}=\frac{o(1)}{\la^{\frac{3}{2}-\frac{3\alpha}{4}}}.
\end{equation}
Using Equations \eqref{eqS1-w}, \eqref{Young3} and \eqref{Young4} in \eqref{eq2-v}, and using the fact that $\alpha\in(0,1)$ we get
\begin{equation}\label{BoundaryEst2}
\abs{S_d(l_1)}^2+\abs{S_d(l_0)}^2\leq  \frac{\la^{1+\frac{\alpha}{2}}}{2}\int_{l_0}^{l_1}\abs{v}^2dx+o(1).
\end{equation}
\end{proof}
\begin{Lemma}\label{2nd-Estimation}
Let $0<\alpha<1$ and $\eta >0$. Then, the solution $(u,v,y,z,\omega)\in D(\AA)$ of system \eqref{pol3}-\eqref{pol7} satisfies the following asymptotic behavior
\begin{equation}\label{Estimation2}
\int_{l_0}^{l_1}\abs{v}^2dx=\dfrac{o(1)}{\la^{1+\frac{\alpha}{2}}}.
\end{equation}
\end{Lemma}
\begin{proof}
Multiply Equation \eqref{pol4} by $-i\la^{-1}\bar{v}$ and integrate over $(l_0,l_1)$, we get
\begin{equation}\label{eq3-v}
\displaystyle
\int_{l_0}^{l_1}\abs{v}^2dx=\Re\left(i\la^{-1}\int_{l_0}^{l_1}S_d\bar{v}_xdx\right)-\left[i\la^{-1}S_d\bar{v}\right]_{l_0}^{l_1}+\Re\left(i\la^{-2+\frac{\alpha}{2}}\int_{l_0}^{l_1}f_2\bar{v}dx\right).
\end{equation}
Estimation of the term $\displaystyle\Re\left(i\la^{-1}\int_{l_0}^{l_1}S_d\bar{v}_xdx\right)$. Using Cauchy-Schwarz inequality, the second estimation in \eqref{FE-1} and the estimation in \eqref{eqS1-w}, we get
\begin{equation}\label{eq4-v}
\left|\Re\left(i\la^{-1}\int_{l_0}^{l_1}S_d\bar{v}_xdx\right)\right|\leq \frac{1}{\la}\left(\int_{l_0}^{l_1}\abs{S_d}^2dx\right)^{\frac{1}{2}}\left(\int_{l_0}^{l_1}\abs{v_x}^2dx\right)^{\frac{1}{2}}= \frac{o(1)}{\la^{\frac{3}{2}}}.
\end{equation}
Estimation for the term $\displaystyle\Re\left(i\la^{-2+\frac{\alpha}{2}}\int_{l_0}^{l_1}f_2\bar{v}dx\right)$. Using Cauchy-Schwarz inequality, $v$ is uniformly bounded in $L^2(l_0,l_1)$ and $\|f_2\|_{L^2(l_0,l_1)}=o(1)$, we get
\begin{equation}\label{eq5-v}
\left|\Re\left(i\la^{-2+\frac{\alpha}{2}}\int_{l_0}^{l_1}f_2\bar{v}dx\right)\right|\leq  \la^{-2+\frac{\alpha}{2}} \left(\int_{l_0}^{l_1}\abs{f_2}^2dx\right)^{\frac{1}{2}}\left(\int_{l_0}^{l_1}\abs{v}^2dx\right)^{\frac{1}{2}}= \frac{o(1)}{\la^{2-\frac{\alpha}{2}}}.
\end{equation}
Inserting Equations \eqref{eq4-v} and \eqref{eq5-v} in \eqref{eq3-v}, using the fact that $\min(\frac{3}{2},2-\frac{\alpha}{2})=\frac{3}{2}$, and using Young's inequality on the second term of \eqref{eq3-v} we get
\begin{equation}\label{eq6-v}
\int_{l_0}^{l_1}\abs{v}^2dx\leq\frac{\la^{-1+\frac{\alpha}{2}}}{2}\left[\abs{v(l_1)}^2+\abs{v(l_0)}^2\right]+\frac{\la^{-1-\frac{\alpha}{2}}}{2}\left[\abs{S_d(l_1)}^2+\abs{S_d(l_0)}^2\right]+\frac{o(1)}{\la^{\frac{3}{2}}}.
\end{equation}
Now, inserting \eqref{BoundaryEst1'} and \eqref{BoundaryEst2'} in \eqref{eq6-v}, we get
\begin{equation}\label{eq7-v}
\int_{l_0}^{l_1}\abs{v}^2dx\leq \left(\frac{1}{2}+m_{g^{\prime}}\la^{-1+\frac{\alpha}{2}}\right)\int_{l_0}^{l_1}\abs{v}^2dx+\frac{o(1)}{\la^{1+\frac{\alpha}{2}}}+\frac{o(1)}{\la^{2-\frac{\alpha}{2}}}.
\end{equation}
Since $\alpha\in(0,1)$ then $\min(1+\frac{\alpha}{2},2-\frac{\alpha}{2})=1+\frac{\alpha}{2}$, then Equation \eqref{eq7-v} becomes
\begin{equation}\label{eq8-v}
\left(\frac{1}{2}-m_{g^{\prime}}\la^{-1+\frac{\alpha}{2}}\right)\int_{l_0}^{l_1}\abs{v}^2dx\leq \frac{o(1)}{\la^{1+\frac{\alpha}{2}}}.
\end{equation}
Using the fact that $\abs{\la}\to +\infty $, we can take $\la\geq 4^{\frac{1}{2-\alpha}}m_{g^{\prime}}^{\frac{1}{2-\alpha}}$, and we get our desired result. The proof is thus complete.
\end{proof}
\begin{Lemma}\label{estwave}
Assume that $\eta>0$. Let $h\in C^1([0,L])$ and $\varphi\in C^2([-L,0])$, then the solution $(u,v,y,z,\omega)\in D(\AA_1)$ of system \eqref{pol3}-\eqref{pol7} satisfies the following estimation
\begin{equation}\label{finalest1}
\begin{array}{ll}
\displaystyle
\int_0^Lh^{\prime}\left(\abs{v}^2+a^{-1}\abs{S_d}^2\right)dx+\int_{-L}^0 \varphi^{\prime}\left(\abs{z}^2+3b\abs{y_{xx}}^2\right)dx+2b\int_{-L}^0 y_{xx}\varphi^{\prime\prime}\bar{y}_xdx+b\varphi(-L)\abs{y_{xx}(-L)}^2\\[0.1in]
\displaystyle+h(0)\abs{v(0)}^2-ah(L)\abs{u_x(L)}^2+ah(0)\abs{u_x(0)}^2+\Re\left(2by_{xxx}(0)\varphi(0)\bar{y}_x(0)\right)-\varphi(0)\abs{z(0)}^2=o(1).
\end{array}
\end{equation}
\end{Lemma}
\begin{proof}
The proof is divided into several steps.\\
\textbf{Step 1.} Multiplying Equation \eqref{pol4} by $2a^{-1}h\bar{S}_d$ and integrating over $(0,L)$, we get
\begin{equation}\label{eq8-v}
\begin{array}{ll}
\displaystyle\Re\left(2a^{-1}i\la\int_0^Lvh\bar{S}_ddx\right)+a^{-1}\int_0^Lh^{\prime}\left|S\right|^2dx-a^{-1}h(L)\abs{S_d(L)}^2\\
\displaystyle+a^{-1}h(0)\abs{S_d(0)}^2=\Re\left(2a^{-1}\la^{-1+\frac{\alpha}{2}}\int_0^Lf_2h\bar{S}_ddx\right).
\end{array}
\end{equation}
From Equation \eqref{pol3} we have
$$i\la \bar{u}_x=-\bar{v}_x-\la^{-1+\frac{\alpha}{2}}(\bar{f_1})_x.$$
Then
\begin{equation}
i\la a^{-1}\bar{S}_d=-\bar{v}_x-\la^{-1+\frac{\alpha}{2}}(\bar{f_1})_x+i\la a^{-1}\sqrt{d(x)}\kappa(\alpha)\int_{\R}\abs{\xi}^{\frac{2\alpha-1}{2}}\bar{\omega}(x,\xi)d\xi.
\end{equation}
Then the first term of \eqref{eq8-v}, become
\begin{equation}\label{eqq6}
\begin{array}{ll}
\displaystyle
\Re\left(2a^{-1}i\la \int_0^Lvh\bar{S}_ddx\right)=\int_0^Lh^{\prime}\abs{v}^2dx-h(L)\abs{v(L)}^2+h(0)\abs{v(0)}^2-\Re\left(2\la^{-1+\frac{\alpha}{2}}\int_0^Lvh(\bar{f_1})_xdx\right)\\
\displaystyle\hspace{3cm}+\Re\left(2i\la a^{-1}\kappa(\alpha)\int_0^Lhv\sqrt{d(x)}\left(\int_{\R}\abs{\xi}^{\frac{2\alpha-1}{2}}\bar{\omega}(x,\xi)d\xi\right)dx\right).
\end{array}
\end{equation}
Inserting Equation \eqref{eqq6} into \eqref{eq8-v}, and using the fact that $v$ and $S_d$ are uniformly bounded in $L^2(0,L)$ and $\|f_2\|_{L^2(0,L)}=o(1)$ and $\|f_1\|_{H^1_L(0,L)}=o(1)$, we obtain
\begin{equation}\label{eq9-v}
\begin{array}{ll}
\displaystyle
\int_0^Lh^{\prime}\left(\abs{v}^2+a^{-1}\abs{S_d}^2\right)dx+h(0)\abs{v(0)}^2-h(L)\abs{u_x(L)}^2+h(0)a\abs{u_x(0)}^2\\
\displaystyle+\Re\left(2i\la a^{-1}\kappa(\alpha)\int_0^Lhv\sqrt{d(x)}\left(\int_{\R}\abs{\xi}^{\frac{2\alpha-1}{2}}\bar{\omega}(x,\xi)d\xi\right)dx\right)=\frac{o(1)}{\la^{1-\frac{\alpha}{2}}}.
\end{array}
\end{equation}
Estimation of the term $\displaystyle\Re\left(2i\la a^{-1}\kappa(\alpha)\int_0^Lhv\sqrt{d(x)}\left(\int_{\R}\abs{\xi}^{\frac{2\alpha-1}{2}}\bar{\omega}(x,\xi)d\xi\right)dx\right)$. Using the definition of $d(x) $ and Cauchy-Schwarz inequality, the fact that $0<\alpha<1$ and $\eta>0$, and using the first estimation in \eqref{FE-1} and  Equation \eqref{Estimation2}, we obtain
\begin{equation}\label{estTerm2}
\left|\Re\left(2i\la a^{-1}\kappa(\alpha)\int_0^Lhv\sqrt{d(x)}\left(\int_{\R}\abs{\xi}^{\frac{2\alpha-1}{2}}\bar{\omega}(x,\xi)d\xi\right)dx\right)\right|=o(1).
\end{equation}
Inserting Equation \eqref{estTerm2} in Equation \eqref{eq9-v}, and using the fact that $\alpha\in(0,1)$ we obtain
\begin{equation}\label{finalest1'}
\begin{array}{ll}
\displaystyle
\int_0^Lh^{\prime}\left(\abs{v}^2+a^{-1}\abs{S_d}^2\right)dx+h(0)\abs{v(0)}^2-ah(L)\abs{u_x(L)}^2+ah(0)\abs{u_x(0)}^2=o(1).
\end{array}
\end{equation}
\textbf{Step 2.} Multiplying Equation \eqref{pol7} by $2\varphi\bar{y}_x$ and integrating over $(-L,0)$, we get
\begin{equation}
\begin{array}{l}
\displaystyle\Re\left(2i\la\int_{-L}^0 z\varphi\bar{y}_xdx\right)-2b\int_{-L}^0 y_{xxx}\varphi^{\prime}\bar{y}_xdx-2b\int_{-L}^0 y_{xxx}\varphi\bar{y}_{xx}dx\\
\displaystyle+\Re\left(\left[2by_{xxx}\varphi\bar{y}_x\right]_{-L}^0\right)=\Re\left(2\la^{-1+\frac{\alpha}{2}}\int_{-L}^0f_4\varphi\bar{y}_xdx\right).
\end{array}
\end{equation}
Integrating by parts the second and third terms of the above equation we get
\begin{equation}\label{eq-4.28}
\begin{array}{ll}
\displaystyle
\Re\left(2i\la\int_{-L}^0 z\varphi\bar{y}_xdx\right)+2b\int_{-L}^0 \varphi^{\prime}\abs{y_{xx}}^2dx+2b\int_{-L}^0 y_{xx}\varphi^{\prime\prime}\bar{y}_xdx-\left[2by_{xx}\varphi^{\prime}\bar{y}_x\right]_{-L}^0\\ 
\displaystyle
+b\int_{-L}^0 \varphi^{\prime}\abs{y_{xx}}^2dx -\left[b\varphi\abs{y_{xx}}^2\right]_{-L}^0+\Re\left(2by_{xxx}(0)\varphi(0)\bar{y}_x(0)\right)=\Re\left(2\la^{-1+\frac{\alpha}{2}}\int_{-L}^0f_4\varphi\bar{y}_xdx\right).
\end{array}
\end{equation}
From Equation \eqref{pol5}, we have
\begin{equation}\label{neww}
i\la\bar{y}_x=-\bar{z}_x-\la^{-1+\frac{\alpha}{2}}(\bar{f_3})_x
.
\end{equation}
By inserting Equation \eqref{neww} into the first term of \eqref{eq-4.28}, we get
\begin{equation}\label{eq-4.29}
\begin{array}{ll}
\displaystyle
\Re\left(2i\la\int_{-L}^0 \varphi z\bar{y}_xdx\right)
&\displaystyle=\Re\left(-2\int_{-L}^0 \varphi z\bar{z}_xdx-2\la^{-1+\frac{\alpha}{2}}\int_{-L}^0 (\bar{f_3})_x \varphi zdx\right)\\
&\displaystyle=\int_{-L}^0 \varphi^{\prime}\abs{z}^2dx-\varphi(0)\abs{z(0)}^2-\Re\left(2\la^{-1+\frac{\alpha}{2}}\int_{-L}^0 (\bar{f_3})_x \varphi zdx\right).
\end{array}
\end{equation}
Inserting Equation \eqref{eq-4.29} into \eqref{eq-4.28}, and using the boundary conditions we get
\begin{equation}\label{eq11-v}
\begin{array}{ll}
\displaystyle\int_{-L}^0 \varphi^{\prime}\left(\abs{z}^2+3b\abs{y_{xx}}^2\right)dx+2b\int_{-L}^0 \varphi^{\prime\prime}y_{xx}\bar{y}_xdx+b\varphi(-L)\abs{y_{xx}(-L)}^2+\Re\left(2b\varphi(0)y_{xxx}(0)\bar{y}_x(0)\right)\\
\displaystyle-\varphi(0)\abs{z(0)}^2=\Re\left(2\la^{-1+\frac{\alpha}{2}}\int_{-L}^0\varphi f_4\bar{y}_xdx\right)+\Re\left(2\la^{-1+\frac{\alpha}{2}}\int_{-L}^0 \varphi (\bar{f_3})_x  zdx\right).
\end{array}
\end{equation}
Estimation of the term $\displaystyle\Re\left(2\la^{-1+\frac{\alpha}{2}}\int_{-L}^0\varphi f_4\bar{y}_xdx\right)$. Using Poincare inequality, Cauchy-Schwarz inequality, the definition of $\varphi$, and $y_{xx}$ is uniformly bounded in $L^2(-L,0)$, and that $\|f_4\|_{L^2(-L,0)}=o(1)$, we get
\begin{equation}\label{estTerm5}
\begin{array}{rcl}
\displaystyle
\left|\Re\left(2\la^{-1+\frac{\alpha}{2}}\int_{-L}^0\varphi f_4\bar{y}_xdx\right)\right|&\leq& \la^{-1+\frac{\alpha}{2}}\|y_x\|_{L^2(0,L)}\|f_4\|_{L^2(0,L)}\\[0.1in]
&\leq&\displaystyle
 \la^{-1+\frac{\alpha}{2}}c_p\|y_{xx}\|_{L^2(0,L)}\|f_4\|_{L^2(0,L)}=\frac{o(1)}{\la^{1-\frac{\alpha}{2}}}.
 \end{array}
\end{equation}
Estimation of the term $\displaystyle\Re\left(2\la^{-1+\frac{\alpha}{2}}\int_{-L}^0 \varphi(\bar{f_3})_x  zdx\right)$. Using Cauchy-Schwarz inequality, the definition of $\varphi$, and $z$ is bounded in $L^2(-L,0)$, and that $\|(f_3)_x\|_{L^2(-L,0)}=o(1)$ we get
\begin{equation}\label{estTerm6}
\left|\Re\left(2\la^{-1+\frac{\alpha}{2}}\int_{-L}^0 \varphi(\bar{f_3})_x  zdx\right)\right|\leq \la^{-1+\frac{\alpha}{2}}\left(\int_0^L\abs{z}^2dx\right)^{\frac{1}{2}}\left(\int_0^L\abs{(f_3)_x}^2dx\right)^{\frac{1}{2}}=\frac{o(1)}{\la^{1-\frac{\alpha}{2}}}.
\end{equation}
Inserting \eqref{estTerm5} and \eqref{estTerm6} into \eqref{eq11-v}, we get
\begin{equation}\label{eq16}
\begin{array}{l}
\displaystyle\int_{-L}^0 \varphi^{\prime}\left(\abs{z}^2+3b\abs{y_{xx}}^2\right)dx+2b\int_{-L}^0 \varphi^{\prime\prime}y_{xx}\bar{y}_xdx+b\varphi(-L)\abs{y_{xx}(-L)}^2\\[0.1in]
\displaystyle
+\Re\left(2b\varphi(0)y_{xxx}(0)\bar{y}_x(0)\right)-\varphi(0)\abs{z(0)}^2=\frac{o\left(1\right)}{\la^{1-\frac{\alpha}{2}}}.
\end{array}
\end{equation}
Now, summing Equations \eqref{finalest1'} and \eqref{eq16}, we get our desired result.
\end{proof}

\begin{Lemma}\label{lem7}
Assume that $\eta>0$. The solution $(u,v,y,z,\omega)\in D(\AA_1)$ of system \eqref{pol3}-\eqref{pol7} satisfies the following estimation
\begin{equation}\label{Estimation3}
\|U\|_{\mathcal{H}_1}=o(1).
\end{equation} 
\end{Lemma}
\begin{proof}
The proof of this Lemma is divided into several steps.\\
\textbf{Step 1.} In this step we will prove that 
$\|v\|_{L^(0,L)}=o(1)$ and $\|u_x\|_{L^(0,L)}=o(1).$\\
Taking $h(x)=x\theta_1(x)+(x-L)\theta_2(x)$ and $\varphi(x)=0$ in Equation \eqref{finalest1}, where $\theta_1,\, \theta_2\in C^1([0,L])$ are defined as follows
 \begin{equation}
\theta_{1}(x)=\left\{\begin{array}{ccc}
1&\text{if}&x\in [0,l_0],\\
0&\text{if}& x\in [l_1,L],\\
0\leq\theta_1\leq 1&& elsewhere,
\end{array}\right.
\end{equation}
and
\begin{equation}
\theta_{2}(x)=\left\{\begin{array}{ccc}
1&\text{if}&x\in [l_1,L],\\
0&\text{if}& x\in [0,l_0],\\
0\leq\theta_2\leq 1&& elsewhere.
\end{array}\right.
\end{equation}
we get
\begin{equation}
\int_0^L(\theta_1+\theta_2)\left(\abs{v}^2+a^{-1}\abs{S_d}^2\right)dx+\int_0^L(x\theta_1^{\prime}+(x-L)\theta_2^{\prime})\left(\abs{v}^2+a^{-1}\abs{S}^2\right)dx=o(1).
\end{equation}
Using Equations \eqref{eqS1-w}, \eqref{Estimation2} and the definition of $\theta_1$ and $\theta_2$, we get
\begin{equation}\label{vs1}
\int_0^L(\theta_1+\theta_2)\left(\abs{v}^2+a^{-1}\abs{S_d}^2\right)dx=o(1).
\end{equation}
Hence, we deduce that
\begin{equation}\label{vs1}
\int_{0}^{l_0}\abs{v}^2dx=o(1)\quad\text{and}\quad a\int_{0}^{l_0}\abs{u_x}^2dx=o(1)
\end{equation}
and
\begin{equation}\label{vs2}
\int_{l_1}^{L}\abs{v}^2dx=o(1)\quad\text{and}\quad a\int_{l_1}^{L}\abs{u_x}^2dx=o(1).
\end{equation}
Using \eqref{vs1}, \eqref{vs2}, \eqref{eqS1-w} and \eqref{Estimation2}, we get the desired result of Step 1.\\
\textbf{Step 2.} Taking $h(x)=x-L$ and $\varphi(x)=0$ in Equation \eqref{finalest1}, we get
\begin{equation}\label{vs3}
\int_0^L\left(\abs{v}^2+a^{-1}\abs{S}^2\right)dx-L\abs{v(0)}^2-aL\abs{u_x(0)}^2=o(1).
\end{equation}
By using Step 1, we get
\begin{equation}\label{limitu}
\abs{v(0)}^2=o(1)\quad\text{and}\quad\abs{u_x(0)}^2=o(1).
\end{equation}
\textbf{Step 3.} The aim of this step is to prove that $\|z\|_{L^2(-L,0)}=o(1)$ and $\|y_{xx}\|_{L^2(-L,0)}=o(1)$.\\
Taking $h(x)=0$ and $\varphi(x)=x+L$ in Equation \eqref{finalest1}, we get
\begin{equation}\label{yz1}
\int_{-L}^0 \left(\abs{z}^2+3b\abs{y_{xx}}^2\right)dx+\Re\left(2Lby_{xxx}(0)\bar{y}_x(0)\right)-L\abs{z(0)}^2=o(1).
\end{equation}
Using \eqref{limitu} and the transmission conditions we have
\begin{equation}\label{limity}
b\abs{y_{xxx}(0)}=a\abs{u_x(0)}=o\left(1\right)\quad\text{and}\quad \abs{z(0)}^2=\abs{v(0)}^2=o\left(1\right).
\end{equation}
Inserting Equation \eqref{limity} in Equation \eqref{yz1} and using the fact that $\displaystyle |y_x(0)|\leq \sqrt{L}\|y_{xx}\|_{L^2(-L,0)}=O(1)$ we get
\begin{equation}\label{yz2}
\int_{-L}^0\abs{z}^2dx=o(1)\quad\text{and}\quad b\int_{-L}^0\abs{y_{xx}}^2dx=o(1).
\end{equation}
Finally, using Equations  \eqref{2nd-Estimation}, \eqref{FE-1}, \eqref{yz2},  and Step 1., we get that $\|U\|_{\mathcal{H}_1}=o(1)$.
\end{proof}


\noindent \textbf{Proof of Theorem \ref{pol}.} 
From Lemma \ref{lem7} we get that $\|U\|_{\HH_1}=o(1)$, which contradicts \eqref{pol1}. Consequently, condition \eqref{H2} holds. This implies, from Theorem \ref{bt}, the energy decay estimation \eqref{Energypol}. The proof is thus complete.

\section{W-W$_{FKV}$ Model}\label{coupled-WAVE}
\noindent In this section, we consider the \eqref{Sys1-W} model, where we study the stability of a system of two wave equations coupled through boundary connections with a localized fractional Kelvin-Voigt damping acting on one equation only. \\
By taking the input $V(x,t)=\sqrt{d(x)}u_{xt}(x,t)$ in Theorem \ref{theorem1}, we get that the output $O$ is given by 
$$
O(x,t)=\sqrt{d(x)}I^{1-\alpha,\eta}u_{xt}(x,t)=\frac{\sqrt{d(x)}}{\Gamma(1-\alpha)}\int_0^t(t-s)^{-\alpha}e^{-\eta(t-s)}\partial_su_x(x,s)ds=\sqrt{d(x)}\partial_t^{\alpha,\eta}u_{x}(x,t).
$$
Then system \eqref{Sys1-W} can be recast into the following augmented model 
\begin{equation}\label{AUG1-CW}
\left\{\begin{array}{ll}
 \displaystyle{u_{tt}-\left(au_x+\sqrt{d(x)}\kappa(\alpha)\int_{\mathbb{R}}\abs{\xi}^{\frac{2\alpha-1}{2}}\omega(x,\xi,t)d\xi\right)_{x}=0},&(x,t)\in (0,L)\times \R_{\ast}^+,\\[0.1in]
\displaystyle{y_{tt}-by_{xx}=0},&(x,t)\in (-L,0)\times \times \R_{\ast}^+,\\[0.1in]
\omega_t(x,\xi,t)+\left(|\xi|^2+\eta\right)\omega(x,\xi,t)-\sqrt{d(x)}u_{xt}(x,t)\abs{\xi}^{\frac{2\alpha-1}{2}}=0,&(x,\xi,t)\in (0,L)\times \mathbb{R}\times \R_{\ast}^+,
\end{array}
\right.
\end{equation}
with the following transmission and boundary conditions 
\begin{equation}\label{AUG2-CW}
\left\{\begin{array}{ll}
u(L,t)=y(-L,t)=0,&t\in (0,\infty),\\[0.1in]
au_{x}(0,t)=by_{x}(0,t),&t\in (0,\infty),\\[0.1in]
u(0,t)=y(0,t),&t\in (0,\infty),
\end{array}\right.
\end{equation}
and with the following initial conditions 
\begin{equation}\label{AUG3-CW}
\begin{array}{lll}
u(x,0)=u_0(x),& u_t(x,0)=u_1(x),\hspace{0.2cm}\omega(x,\xi,0)=0& x\in (0,L),\xi\in\mathbb{R},\\[0.1in]
y(x,0)=y_0(x),& y_t(x,0)=y_1(x),& x\in (-L,0).
\end{array}
\end{equation}
The energy of the system \eqref{AUG1-CW}-\eqref{AUG3-CW} is given by  
\begin{equation*}
E_2(t)=\frac{1}{2}\int_0^L\left(\abs{u_t}^2+a\abs{u_x}^2\right)dx+\frac{1}{2}\int_{-L}^0\left(\abs{y_t}^2+b\abs{y_{x}}^2\right)dx+\frac{\kappa(\alpha)}{2}\int_0^L\int_{\R}\abs{\omega(x,\xi,t)}^2d\xi dx.
\end{equation*}
Using similar computations to Lemma \ref{Denergy}, we obtain
\begin{equation}\label{denergy-CW}
\frac{d}{dt}E_2(t)=-\kappa(\alpha)\int_{0}^{L}\int_{\mathbb{R}}(\xi^2+\eta)\abs{\omega(x,\xi,t)}^2d\xi dx.
\end{equation}
\noindent Since $\alpha\in (0,1)$, then $\kappa(\alpha)>0$, and therefore $\displaystyle\frac{d}{dt}E_2(t)\leq 0$. Thus, system \eqref{AUG1-CW}-\eqref{AUG3-CW} is dissipative in the sense that its energy is a non-increasing function with respect to time variable $t$. Now, we define the following Hilbert energy space $\mathcal{H}_2$ by 
$$
\mathcal{H}_2=\left\{(u,v,y,z,\omega)\in H_R^1(0,L)\times L^2(0,L)\times H^1_L(-L,0)\times L^2(-L,0)\times W;\ \ u(0)=y(0)\right\},
$$
where $W=L^2\left((0,L)\times \R\right)$ and 
\begin{equation}
\left\{
\begin{array}{ll}
H_R^1(0,L)=\{u\in H^1(0,L); u(L)=0\},\\[0.1in]
H_L^1(-L,0)=\{y\in H^1(-L,0); y(-L)=0\}.
\end{array}\right.
\end{equation}
The energy space $\mathcal{H}_2$ is equipped with the inner product defined by 
$$
\begin{array}{lll}
\displaystyle
\left<U,U_1\right>_{\mathcal{H}_2}&=&\displaystyle
\int_0^Lv\overline{v_1}dx+a\int_0^Lu_x(\overline{u_1})_xdx+\int_{-L}^0 z\overline{z_1}dx+b\int_{-L}^0 y_{x}(\overline{y_1})_{x}dx
+\kappa(\alpha)\int_{0}^{L}\int_{\mathbb{R}}\omega(x,\xi)\overline{\omega_1}(x,\xi)d\xi dx,
\end{array}
$$
for all $U=(u,v,y,z,\omega)$ and $U_1=(u_1,v_1,y_1,z_1,\omega_1)$ in $\mathcal{H}_2$. We use $\|U\|_{\mathcal{H}_2
}$ to denote the corresponding norm. We define the unbounded linear operator $\AA_2:D(\AA_2)\subset\mathcal{H}_2\rightarrow\mathcal{H}_2$ by 
\begin{equation*}
D(\mathcal{A}_2)=\left\{\begin{array}{c}
\displaystyle{
U=(u,v,y,z,\omega)\in \mathcal{H}_2;\ (v,z)\in H_R^1(0,L)\times H^1_L(-L,0),\ y\in  H^2(-L,0)},\\[0.1in]  
\displaystyle{\left(au_x+\sqrt{d(x)}\kappa(\alpha)\int_{\mathbb{R}}\abs{\xi}^{\frac{2\alpha-1}{2}}\omega(x,\xi)d\xi\right)_{x}}\in L^2(0,L),\vspace{0.2cm}\\ [0.1in]
-\left(\abs{\xi}^2+\eta\right)\omega(x,\xi)+\sqrt{d(x)}v_x|\xi|^{\frac{2\alpha-1}{2}},\quad|\xi|\omega(x,\xi)\in W, \vspace{0.2cm}\\[0.1in]
au_x(0)=by_{x}(0),\,\text{and}\, v(0)=z(0)
\end{array}
\right\},
\end{equation*}
and for all $U=(u,v,y,z,\omega)\in D(\mathcal{A}_2)$, 
$$
\mathcal{A}_2(u,v,y,z,\omega)^{\top}=\begin{pmatrix}
v\\ \vspace{0.2cm} \displaystyle{\left(au_x+\sqrt{d(x)}\kappa(\alpha)\int_{\mathbb{R}}\abs{\xi}^{\frac{2\alpha-1}{2}}\omega(x,\xi)d\xi\right)_{x}}\\ \vspace{0.2cm} z\\\vspace{0.2cm} by_{xx}\\ \vspace{0.2cm}
-\left(\abs{\xi}^2+\eta\right)\omega(x,\xi)+\sqrt{d(x)}v_x|\xi|^{\frac{2\alpha-1}{2}}
\end{pmatrix}.
$$
\noindent If $U=(u,u_t,y,y_t,\omega)$ is a regular solution of system \eqref{AUG1-CW}-\eqref{AUG3-CW}, then the system can be rewritten as evolution equation on the Hilbert space $\mathcal{H}_2$ given by
\begin{equation}\label{evolution-las}
U_t=\mathcal{A}_{2}U,\quad U(0)=U_0,
\end{equation}
where $U_0=(u_0,u_1,y_0,y_1,0)$.\\
In a similar way to Section \ref{strongStabilityWave}, we can see that the unbounded linear operator $\mathcal{A}_2$ is m-dissipative in the energy space $\mathcal{H}_2$. Also, the $C_0$-semigroup of contractions $e^{t\mathcal{A}_2}$ is strongly stable on $\mathcal{H}_2$ in the sense that $\displaystyle{\lim_{t\to+\infty}\left\|e^{t\mathcal{A}_{2}}U_0\right\|_{\mathcal{H}_{2}}}=0$.

\begin{theoreme}\label{pol-CW}
Assume that $\eta>0$. The $C_0-$semigroup $(e^{t\AA_2})_{t\geq 0}$ is polynomially stable; i.e. there exists constant $C_2>0$ such that for every $U_0\in D(\AA_2)$, we have 
\begin{equation}\label{Energypol-CW}
E_2(t)\leq \frac{C_1}{t^{\frac{4}{2-\alpha}}}\|U_0\|^2_{D(\AA_2)},\quad t>0,\,\forall U_0\in D(\mathcal{A}_2).
\end{equation}
\end{theoreme}
According to Theorem \ref{bt}, by taking $\ell=1-\frac{\alpha}{2}$, the polynomial energy decay \eqref{Energypol-CW} holds if the following conditions 
\begin{equation}\label{G1}\tag{${\rm{G_1}}$}
i\R\subset \rho(\mathcal{A}_2),
\end{equation}
and
\begin{equation}\label{G2}\tag{${\rm{G_2}}$}
\sup_{\la\in \R}\left\|(i\la I-\AA_2)^{-1}\right\|_{\mathcal{L}(\mathcal{H}_2)}=O\left(\abs{\la}^{1-\frac{\alpha}{2}}\right)
\end{equation}
are satisfied. Since Condition \eqref{G1} is already proved. We will prove condition \eqref{G2} by an argument of contradiction. For this purpose, suppose that \eqref{G2} is false, then there exists $\left\{\left(\la_n,U_n:=(u_n,v_n,y_n,z_n,\omega_n(\cdot,\xi))^\top\right)\right\}\subset \R^{\ast}\times D(\AA_2)$ with 
\begin{equation}\label{pol1-CW}
\abs{\la_n}\to +\infty \quad \text{and}\quad \|U_n\|_{\mathcal{H}_2}=\|(u_n,v_n,y_n,z_n,\omega_n(\cdot,\xi))\|_{\mathcal{H}_2}=1, 
\end{equation}
such that 
\begin{equation}\label{pol2-w-CW}
\left(\la_n\right)^{1-\frac{\alpha}{2}}\left(i\la_nI-\AA_2\right)U_n=F_n:=(f_{1,n},f_{2,n},f_{3,n},f_{4,n},f_{5,n}(\cdot,\xi))^{\top}\to 0 \ \ \text{in}\ \ \mathcal{H}_2. 
\end{equation}
For simplicity, we drop the index $n$. Equivalently, from \eqref{pol2-w-CW}, we have 
\begin{eqnarray}
i\la u-v&=&\dfrac{f_1}{\la^{1-\frac{\alpha}{2}}}  \quad\text{in}\ H_R^1(0,L),\label{pol3-CW}\\
i\la v-\left(S_d\right)_x&=&\dfrac{f_2}{\la^{1-\frac{\alpha}{2}}} \quad \text{in}\ L^2(0,L),\label{pol4-CW}\\
i\la y-z&=&\dfrac{f_3}{\la^{1-\frac{\alpha}{2}}} \quad\text{in}\ H_L^1(-L,0),\label{pol5-CW}\\
i\la z-by_{xx}&=&\dfrac{f_4}{\la^{1-\frac{\alpha}{2}}} \quad \text{in}\ L^2(-L,0),\label{pol6-CW}\\
(i\la+\xi^2+\eta)\omega(x,\xi)-\sqrt{d(x)}v_x|\xi|^{\frac{2\alpha-1}{2}}&=&\dfrac{f_5(x,\xi)}{\la^{1-\frac{\alpha}{2}}} \quad \text{in}\ W,\label{pol7-CW}
\end{eqnarray}
where $\displaystyle S_d=au_x+\sqrt{d(x)}\kappa(\alpha)\int_{\mathbb{R}}|\xi|^{\frac{2\alpha-1}{2}}\omega(x,\xi)d\xi$.\\
Here we will check the condition \eqref{G2} by finding a contradiction with \eqref{pol1-CW} by showing $\|U\|_{\mathcal{H}_2}=o(1)$.
In order to get this contradiction, we follow similar arguments as in Section \ref{Section-poly}. 
We get the same results as in Lemmas \ref{lemI4}, \ref{First-Estimation}, \ref{S-Estimation}, \ref{vs}, \ref{2nd-Estimation}.\\
Using the same computations as in Lemma \ref{estwave}, we get that the solution $(u,v,y,z,\omega)\in D(\AA_2)$ of system \eqref{pol3-CW}-\eqref{pol7-CW} satisfies the following estimation
\begin{equation}\label{finalest1-CW}
\begin{array}{c}
\displaystyle
\int_0^Lh^{\prime}\left(\abs{v}^2+a^{-1}\abs{S_d}^2\right)dx+\int_{-L}^0 \varphi^{\prime}\left(\abs{z}^2+b\abs{y_{x}}^2\right)dx+h(0)\abs{v(0)}^2-ah(L)\abs{u_x(L)}^2+ah(0)\abs{u_x(0)}^2\\[0.1in]
\displaystyle -\varphi(0)\abs{z(0)}^2- b\varphi(0)\abs{y_{x}(0)}^2 +b\varphi(-L)\abs{y_{x}(-L)}^2=o(1).
\end{array}
\end{equation}
We procced in a similar way to Lemma \ref{lem7}.\\
\textbf{Step 1.} Taking $h(x)=x\theta_1(x)+(x-L)\theta_2(x)$ and $\varphi(x)=0$ in Equation \eqref{finalest1-CW}, where $\theta_1,\, \theta_2\in C^1([0,L])$ are defined in Lemma \ref{lem7}, yields
$$\|v\|_{L^(0,L)}=o(1)\quad \text{and}\quad \|u_x\|_{L^(0,L)}=o(1).$$
\textbf{Step 2.} Taking $h(x)=x-L$ and $\varphi(x)=0$ in Equation \eqref{finalest1-CW}, we obtain
\begin{equation}\label{vs3-CW}
\int_0^L\left(\abs{v}^2+a^{-1}\abs{S_d}^2\right)dx-L\abs{v(0)}^2-aL\abs{u_x(0)}^2=o(1).
\end{equation}
By using Step 1, we get
\begin{equation}\label{limitu-CW}
\abs{v(0)}^2=\abs{u_x(0)}^2=o(1).
\end{equation}
\textbf{Step 3.} 
Taking $h(x)=0$ and $\varphi(x)=x+L$ in Equation \eqref{finalest1-CW}, we get
\begin{equation}\label{yz1-CW}
\int_{-L}^0 \left(\abs{z}^2+b\abs{y_{x}}^2\right)dx-L\abs{z(0)}^2-bL\abs{y_x(0)}^2=o(1).
\end{equation}
Using \eqref{limitu-CW} and the transmission conditions we have
\begin{equation}\label{limity-CW}
b\abs{y_{x}(0)}=a\abs{u_x(0)}=o\left(1\right)\quad\text{and}\quad \abs{z(0)}^2=\abs{v(0)}^2=o\left(1\right).
\end{equation}
Inserting Equation \eqref{limity-CW} in Equation \eqref{yz1-CW}, we get
\begin{equation}\label{yz2}
\int_{-L}^0\abs{z}^2dx=o(1)\quad\text{and}\quad b\int_{-L}^0\abs{y_{x}}^2dx=o(1).
\end{equation}
\noindent \textbf{Proof of Theorem \ref{pol-CW}.} 
From Lemma Step 1. and Staep 3. we deduce that $\|U\|_{\HH_2}=o(1)$, which contradicts \eqref{pol1-CW}. Consequently, condition \eqref{G2} holds. This implies, from Theorem \ref{bt}, the energy decay estimation \eqref{Energypol-CW}. The proof is thus complete.

\section{W-(EBB)$_{FKV}$ Model}\label{Section I}
\noindent This section is devoted to study the stability of the model \eqref{beam1}, where we consider the Euler-Bernoulli beam and  wave equations coupled through boundary connection. We take the fractional Kelvin-Voigt damping to be a localized internal damping acting on the Euler-Bernoulli beam only. 
\subsection{Well-Posedness and Strong Stability}\label{subsec}
\noindent In this section, we give the strong stability results of  the system \eqref{beam1}. First, using a semigroup approach, we establish well-posedness result for the system \eqref{beam1}.\\
In Theorem \ref{theorem1}, taking the input $V(x,t)=\sqrt{d(x)}y_{xxt}(x,t)$, then  using \eqref{Caputo}, we get the output $O$ is given by 
$$
O(x,t)=\sqrt{d(x)}I^{1-\alpha,\eta}y_{xxt}(x,t)=\frac{\sqrt{d(x)}}{\Gamma(1-\alpha)}\int_0^t(t-s)^{-\alpha}e^{-\eta(t-s)}\partial_sy_{xx}(x,s)ds=\sqrt{d(x)}\partial_t^{\alpha,\eta}y_{xx}(x,t).
$$
Therefore, by taking the input $V(x,t)=\sqrt{d(x)}y_{xxt}(x,t)$ in Theorem \ref{theorem1} and using the above equation, we get 
\begin{equation}\label{augnew1}
\begin{array}{llll}
\partial_t\omega(x,\xi,t)+(\xi^2+\eta)\omega(x,\xi,t)-\sqrt{d(x)}y_{xxt}(x,t)|\xi|^{\frac{2\alpha-1}{2}}&=&0,&(x,\xi,t)\in (0,L)\times \mathbb{R}\times (0,\infty),\\
\omega(\xi,0)&=&0,&(x,\xi)\in (0,L)\times \mathbb{R},\\
\displaystyle{\sqrt{d(x)}\partial_t^{\alpha,\eta}y_{xx}(x,t)-\kappa(\alpha)\int_{\mathbb{R}}|\xi|^{\frac{2\alpha-1}{2}}\omega(x,\xi,t)d\xi }&=&0,&(x,t)\in (0,L)\times (0,\infty).
\end{array}
\end{equation}
From system \eqref{augnew1}, we deduce that system \eqref{beam1} can be recast into the following augmented model 
\begin{equation}\label{AUG1B}
\left\{\begin{array}{ll}
 \displaystyle{u_{tt}-au_{xx}=0},&(x,t)\in (-L,0)\times \R_{\ast}^+,\\[0.1in]
\displaystyle{y_{tt}+\left(by_{xx}+\sqrt{d(x)}\kappa(\alpha)\int_{\mathbb{R}}\abs{\xi}^{\frac{2\alpha-1}{2}}\omega(x,\xi,t)d\xi\right)_{xx}=0},&(x,t)\in (0,L)\times \R_{\ast}^+,\\[0.1in]
\omega_t(x,\xi,t)+\left(|\xi|^2+\eta\right)\omega(x,\xi,t)-\sqrt{d(x)}y_{xxt}(x,t)\abs{\xi}^{\frac{2\alpha-1}{2}}=0,&(x,\xi,t)\in (0,L)\times \mathbb{R}\times \R_{\ast}^+,
\end{array}
\right.
\end{equation}
with the following transmission and boundary conditions 
\begin{equation}\label{AUG2B}
\left\{\begin{array}{ll}
u(-L,t)=y(L,t)=y_{x}(L,t)=0,&t\in (0,\infty),\\[0.1in]
au_{x}(0,t)+by_{xxx}(0,t)=0,y_{xx}(0,t)=0,&t\in (0,\infty),\\[0.1in]
u(0,t)=y(0,t),&t\in (0,\infty),
\end{array}\right.
\end{equation}
and with the following initial conditions 
\begin{equation}\label{AUG3B}
\begin{array}{lll}
u(x,0)=u_0(x),& u_t(x,0)=u_1(x)& x\in (-L,0)\\[0.1in]
y(x,0)=y_0(x),& y_t(x,0)=y_1(x),\hspace{0.2cm}\omega(x,\xi,0)=0& x\in (0,L),\xi\in\mathbb{R}.
\end{array}
\end{equation}
The energy of the system \eqref{AUG1B}-\eqref{AUG3B} is given by  
\begin{equation*}
E_3(t)=\frac{1}{2}\int_{-L}^0\left(\abs{u_t}^2+a\abs{u_x}^2\right)dx+\frac{1}{2}\int_{0}^L\left(\abs{y_t}^2+b\abs{y_{xx}}^2\right)dx+\frac{\kappa(\alpha)}{2}\int_{0}^L\int_{\R}\abs{\omega(x,\xi,t)}^2d\xi dx.
\end{equation*}
By similar computation to Lemma \ref{Denergy}, it is easy to see that the energy $E_3(t)$ satisfies the following estimation 
\begin{equation}\label{denergybeam}
\frac{d}{dt}E_3(t)=-\kappa(\alpha)\int_{0}^L\int_{\mathbb{R}}(\xi^2+\eta)\abs{\omega(x,\xi,t)}^2d\xi dx.
\end{equation}
$\newline$
\noindent Since $\alpha\in (0,1)$, then $\kappa(\alpha)>0$, and therefore $\displaystyle\frac{d}{dt}E_3(t)\leq 0$. Thus, system \eqref{AUG1B}-\eqref{AUG3B} is dissipative in the sense that its energy is a non-increasing function with respect to time variable $t$. Now, we define the following Hilbert energy space $\mathcal{H}_{3}$ by 
$$
\mathcal{H}_{3}=\left\{(u,v,y,z,\omega)\in H_L^1(-L,0)\times L^2(-L,0)\times H^2_R(0,L)\times L^2(0,L)\times W;\ \ u(0)=y(0)\right\},
$$
where $W=L^2\left((0,L)\times \R\right)$ and
\begin{equation}
\left\{
\begin{array}{ll}
H_L^1(-L,0)=\{u\in H^1(-L,0); u(-L)=0\},\\[0.1in]
H_R^2(0,L)=\{y\in H^2(0,L); y(L)=y_x(L)=0\}.
\end{array}\right.
\end{equation} 
The energy space $\mathcal{H}_3$ is equipped with the inner product defined by 
$$
\begin{array}{lll}
\displaystyle
\left<U,U_1\right>_{\mathcal{H}_3}&=&\displaystyle
\int_{-L}^0 v\bar{v_1}dx+a\int_{-L}^0u_x(\overline{u_1})_xdx+\int_{0}^L z\bar{z_1}dx+b\int_{0}^L y_{xx}(\overline{y_1})_{xx}dx
+\kappa(\alpha)\int_{0}^L\int_{\mathbb{R}}\omega(x,\xi)\overline{\omega_1}(x,\xi)d\xi dx,
\end{array}
$$
for all $U=(u,v,y,z,\omega)$ and $U_1=(u_1,v_1,y_1,z_1,\omega_1)$ in $\mathcal{H}_{3}$. We use $\|U\|_{\mathcal{H}_{3}}$ to denote the corresponding norm. We define the unbounded linear operator $\AA_{3}:D(\AA_{3})\subset\mathcal{H}_{3}\rightarrow\mathcal{H}_{3}$ by 
\begin{equation*}
D(\mathcal{A}_{3})=\left\{\begin{array}{c}
\displaystyle{U=(u,v,y,z,\omega)\in \mathcal{H}_3;\ (v,z)\in H_L^1(-L,0)\times H^2_R(0,L),\ u\in  H^2(-L,0)},\\[0.1in]  
\displaystyle{\left(by_{xx}+\sqrt{d(x)}\kappa(\alpha)\int_{\mathbb{R}}\abs{\xi}^{\frac{2\alpha-1}{2}}\omega(x,\xi)d\xi\right)_{xx}}\in L^2(0,L),\vspace{0.2cm}\\ [0.1in]
-\left(\abs{\xi}^2+\eta\right)\omega(x,\xi)+\sqrt{d(x)}z_{xx}|\xi|^{\frac{2\alpha-1}{2}},\quad|\xi|\omega(x,\xi)\in W, \vspace{0.2cm}\\[0.1in]
au_x(0)+by_{xxx}(0)=0,\,  y_{xx}(0)=0,\,\text{and}\, v(0)=z(0)
\end{array}
\right\},
\end{equation*}
and for all $U=(u,v,y,z,\omega)\in D(\mathcal{A}_{3})$, 
$$
\mathcal{A}_{3}(u,v,y,z,\omega)^{\top}=\begin{pmatrix}
v\\ \vspace{0.2cm} au_{xx}\\ \vspace{0.2cm} z\\\vspace{0.2cm}\displaystyle{-\left(by_{xx}+\sqrt{d(x)}\kappa(\alpha)\int_{\mathbb{R}}\abs{\xi}^{\frac{2\alpha-1}{2}}\omega(x,\xi)d\xi\right)_{xx}}\\ \vspace{0.2cm}
-\left(\abs{\xi}^2+\eta\right)\omega(x,\xi)+\sqrt{d(x)}z_{xx}|\xi|^{\frac{2\alpha-1}{2}}
\end{pmatrix}.
$$
\noindent If $U=(u,u_t,y,y_t,\omega)$ is a regular solution of system \eqref{AUG1B}-\eqref{AUG3B}, then the system can be rewritten as evolution equation on the Hilbert space $\mathcal{H}_3$ given by
\begin{equation}\label{evolution}
U_t=\mathcal{A}_{3}U,\quad U(0)=U_0,
\end{equation}
where $U_0=(u_0,u_1,y_0,y_1,0)$.\\
\noindent Similar to Proposition \ref{mdissipatif}, the operator $\mathcal{A}_2$ is m-dissipative on $\mathcal{H}_3$, consequently it  generates a $C_0$-semigroup of contractions $(e^{t\mathcal{A}_3})_{t\geq 0}$ following Lummer-Phillips theorem (see in \cite{LiuZheng01} and \cite{Pazy01}). Then the solution of the evolution Equation \eqref{evolution} admits the following representation
$$
U(t)=e^{t\mathcal{A}_3}U_0,\quad t\geq 0,
$$
which leads to the well-posedness of \eqref{evolution}. Hence, we have the following result. 
\begin{theoreme}
Let $U_0\in \mathcal{H}_3$, then problem \eqref{evolution} admits a unique weak solution $U$ satisfies 
$$
U(t)\in C^0\left(\R^+,\mathcal{H}_3\right).
$$
Moreover, if $U_0\in D(\mathcal{A}_3)$, then problem \eqref{evolution} admits a unique strong solution $U$ satisfies 
$$
U(t)\in C^1\left(\R^+,\mathcal{H}_3\right)\cap C^0\left(\R^+,D(\mathcal{A}_3)\right).
$$
\end{theoreme}
\begin{theoreme}\label{Strong-Beam}
Assume that $\eta\geq 0$, then the $C_0-$semigroup of contractions $e^{t\mathcal{A}_{3}}$ is strongly stable on $\mathcal{H}_{3}$ in the sense that $\displaystyle{\lim_{t\to+\infty}\left\|e^{t\mathcal{A}_{3}}U_0\right\|_{\mathcal{H}_{3}}}=0$.  
\end{theoreme}


\textbf{Proof of Theorem \ref{Strong-Beam}.}  The proof of this theorem follows by proceeding with similar arguments as in Subsection  \ref{strongStabilityWave}, and using the Arendt Batty Theorem (see Theorem \ref{arendtbatty} in Appendix).
\subsection{Polynomial Stability in the case $\eta>0$}\label{Section-poly-beam}
\noindent The aim of this part is to study the polynomial stability of system \eqref{AUG1B}-\eqref{AUG3B} in the case $\eta>0$. As the condition $ i\mathbb{R}\subset \rho(\mathcal{A}_2)$ is already checked in the subsection \ref{subsec}, it remains to prove that condition \eqref{h1} holds (see Theorem \ref{bt} in Appendix). This is established by using specific  multipliers, some interpolation inequalities and by solving differential equations of order 4.
\noindent Our main result in this part is the following theorem. 
\begin{theoreme}\label{polbeam}
Assume that $\eta>0$. The $C_0-$ semigroup $(e^{t\AA_3})_{t\geq 0}$ is polynomially stable; i.e. there exists constant $C_3>0$ such that for every $U_0\in D(\AA_3)$, we have 
\begin{equation}\label{Energypolbeam}
E_3(t)\leq \frac{C_2}{t^{\frac{2}{3-\alpha}}}\|U_0\|^2_{D(\AA_3)},\quad t>0,\,\forall\, U_0\in D(\mathcal{A}_3).
\end{equation}
\end{theoreme}
\noindent According to Theorem \ref{bt}, by taking $\ell=3-\alpha$, the polynomial energy decay \eqref{Energypolbeam} holds if the following conditions 
\begin{equation}\label{R1}\tag{${\rm{R_1}}$}
i\R\subset \rho(\mathcal{A}_3),
\end{equation}
and
\begin{equation}\label{R2}\tag{${\rm{R_2}}$}
\sup_{\la\in \R}\left\|(i\la I-\AA_3)^{-1}\right\|_{\mathcal{L}(\mathcal{H}_3)}=O\left(\abs{\la}^{3-\alpha}\right)
\end{equation}
are satisfied. 
\noindent Since condition \eqref{R1}is already proved (see Subsection \ref{subsec}), we still need to prove condition \eqref{R2}. For this purpose we will use an argument of contradiction. Suppose that \eqref{R2} is false, then there exists $\left\{\left(\la_n,U_n:=(u_n,v_n,y_n,z_n,\omega_n(\cdot,\xi))^\top\right)\right\}\subset \R^{\ast}\times D(\AA_3)$ with 
\begin{equation}\label{pol1beam}
\abs{\la_n}\to +\infty \quad \text{and}\quad \|U_n\|_{\mathcal{H}_3}=\|(u_n,v_n,y_n,z_n,\omega_n(\cdot,\xi))\|_{\mathcal{H}_3}=1, 
\end{equation}
such that 
\begin{equation}\label{pol2}
\left(\la_n^{3-\alpha}\right)\left(i\la_nI-\AA_3\right)U_n=F_n:=(f_{1,n},f_{2,n},f_{3,n},f_{4,n},f_{5,n}(\cdot,\xi))^{\top}\to 0 \ \ \text{in}\ \ \mathcal{H}_3. 
\end{equation}
For simplicity, we drop the index $n$. Equivalently, from \eqref{pol2}, we have 
\begin{eqnarray}
i\la u-v&=&\dfrac{f_1}{\la^{3-\alpha}} \quad\text{in}\ H_L^1(-L,0),\label{pol3beam}\\
i\la v-au_{xx}&=&\dfrac{f_2}{\la^{3-\alpha}} \quad \text{in}\ L^2(-L,0),\label{pol4beam}\\
i\la y-z&=&\dfrac{f_3}{\la^{3-\alpha}} \quad\text{in}\ H_R^2(0,L),\label{pol5beam}\\
i\la z+S_{xx}&=&\dfrac{f_4}{\la^{3-\alpha}}\quad \text{in}\ L^2(0,L),\label{pol6beam}\\
(i\la+\xi^2+\eta)\omega(x,\xi)-\sqrt{d(x)}z_{xx}|\xi|^{\frac{2\alpha-1}{2}}&=&\dfrac{f_5(x,\xi)}{\la^{3-\alpha}}\quad \text{in}\ W,\label{pol7beam}
\end{eqnarray}
where $\displaystyle S=by_{xx}+\sqrt{d(x)}\kappa(\alpha)\int_{\mathbb{R}}|\xi|^{\frac{2\alpha-1}{2}}\omega(x,\xi)d\xi$. Here we will check the condition \eqref{R2} by finding a contradiction with \eqref{pol1beam} by showing $\|U\|_{\mathcal{H}_3}=o(1)$. For clarity, we divide the proof into several Lemmas.
\begin{Lemma}\label{First-Estimation-beam}
Assume that $\eta>0$. Then, the solution $(u,v,y,z,\omega)\in D(\AA_3)$ of system \eqref{pol3beam}-\eqref{pol7beam} satisfies the following asymptotic behavior estimations 
\begin{equation}\label{FE-1-beam}
\begin{array}{l}
\displaystyle
\int_0^L\int_{\R}\left(\abs{\xi}^2+\eta\right)\abs{\omega(x,\xi)}^2d\xi dx=\frac{o\left(1\right)}{\la^{3-\alpha}},\quad \int_{l_0}^{l_1}\abs{z_{xx}}^2dx=\frac{o\left(1\right)}{\la^2},
 \\
 \displaystyle
\int_{l_0}^{l_1}\abs{y_{xx}}^2dx=\frac{o\left(1\right)}{\la^4}\quad
\text{and}\quad \int_{l_0}^{l_1}\left|S\right|^2dx=\frac{o(1)}{\la^{3-\alpha}}.
\end{array}
\end{equation}
\end{Lemma}
\begin{proof}
For the clarity of the proof, we divide the proof into several steps.\\ 
\textbf{Step 1.} Taking the inner product of $F$ with $U$ in $\mathcal{H}_3$, then using \eqref{pol1beam} and the fact that $U$ is uniformly bounded in $\mathcal{H}_3$, we get 
$$
\kappa(\alpha)\int_0^L\int_{\R}\left(\xi^2+\eta\right)\abs{\omega(x,\xi)}^2d\xi dx=-\Re\left(\left<\AA_3 U,U\right>_{\mathcal{H}_3}\right)=\Re\left(\left<(i\la I-\AA_3)U,U\right>_{\mathcal{H}_3}\right)=o\left(\la^{-3+\alpha}\right).
$$
\textbf{Step 2.} Our aim here is to prove the second estimation in \eqref{FE-1-beam}.\\
From \eqref{pol7beam}, we get
$$
\sqrt{d(x)}\abs{\xi}^{\frac{2\alpha-1}{2}}\abs{z_{xx}}\leq \left(\abs{\la}+\xi^2+\eta\right)\abs{\omega(x,\xi)}+\abs{\la}^{-3+\alpha}\abs{f_5(x,\xi)}.
$$
Multiplying the above inequality by $\left(\abs{\la}+\xi^2+\eta\right)^{-2}\abs{\xi}$, integrating over $\R$ and proceeding in a similar way as in Lemma \ref{First-Estimation} (Section \ref{Section-poly}), we get the second desired estimation in \eqref{FE-1-beam}.\\
\textbf{Step 3.}  From Equation \eqref{pol5beam} we have that
$$
\displaystyle{\|\la y_{xx}\|_{L^2(l_0,l_1)}\leq \|z_{xx}\|_{L^2(l_0,l_1)}+\abs{\la}^{-3+\alpha}}\|(f_3)_{xx}\|_{L^2(l_0,l_1)}.
 $$
Using Step 2, the fact that $\|f_3\|_{H^2_R(0,L)}=o(1)$, and that $\alpha\in (0,1)$, we get the third estimation in \eqref{FE-1-beam}.\\
\textbf{Step 4.}
Using the fact that $|P+Q|^2\leq 2P^2+2Q^2$, we obtain
\begin{equation*}
\begin{array}{lll}
\displaystyle
\int_{l_0}^{l_1}\left|S\right|^2dx&\displaystyle=\int_{l_0}^{l_1}\left|by_{xx}+\sqrt{d(x)}\kappa(\alpha)\int_{\mathbb{R}}|\xi|^{\frac{2\alpha-1}{2}}\omega(x,\xi)d\xi\right|^2 dx\\
&\displaystyle \leq 2b^2 \int_{l_0}^{l_1}\abs{y_{xx}}^2dx
\displaystyle
+2d_0\kappa(\alpha)^2\int_{l_0}^{l_1}\left(\int_{\R}\frac{\abs{\xi}^{\frac{2\alpha-1}{2}}\sqrt{\xi^2+\eta}}{\sqrt{\xi^2+\eta}}\omega(x,\xi)d\xi\right)^2 dx\\
&\displaystyle \leq 2b^2 \int_{l_0}^{l_1}\abs{y_{xx}}^2dx+c_2 \int_{l_0}^{l_1}\int_{\R}(\xi^2+\eta)\abs{\omega(x,\xi)}^2d\xi dx
\end{array}
\end{equation*}
where $\displaystyle{c_2=d_0\kappa(\alpha)^2\,\mathtt{I}_{15}(\alpha,\eta)}$ is defined in Lemma \ref{S-Estimation}. Thus, we get the last estimation in \eqref{FE-1-beam}. Hence, the proof is complete.
\end{proof}

\begin{Lemma}\label{lem-inter1}
Assume that $\eta>0$. The solution $(u,v,y,z,\omega)\in D(\AA_3)$ of system \eqref{pol3beam}-\eqref{pol7beam} satisfies the following asymptotic behavior 
\begin{equation}
\|z\|_{H^2(l_0,l_1)}=\frac{o(1)}{\la},\quad\|y\|_{H^2(l_0,l_1)}=\frac{o(1)}{\la^2},\quad \text{and}\quad \|S_x\|_{L^2(l_0,l_1)}=\frac{o(1)}{\la^{\frac{3-\alpha}{4}}}. 
\end{equation}
\end{Lemma}
\begin{proof}
The proof of this Lemma will be divided into two steps.\\
\textbf{Step 1.} Let $ \displaystyle0<\varepsilon< \frac{l_1-l_0}{2}$. We define $h\in C^{\infty}([0,L])$, $0\leq h\leq 1$ on $[0,L]$, $h=1$ on $(l_0+\varepsilon,l_1-\varepsilon)$, and $h=0$ on $(0,l_0)\cup(l_1,L)$. Also, we define $\displaystyle\max_{x\in[0,L]}\abs{ h^{\prime}(x)}=m^{\prime}_{h}$ and $\displaystyle\max_{x\in[0,L]} \abs{h^{\prime\prime}(x)}=m^{\prime\prime}_{h}$, where $m^{\prime}_{h}$ and $m^{\prime\prime}_{h}$ are strictly positive constant numbers. Multiply equation \eqref{pol6beam} by $-i\la^{-1}h\bar{z}$ and integrate over $(l_0,l_1)$, we get
\begin{equation}\label{eqStep1}
\int_{l_0}^{l_1}h|z|^2dx=i\la^{-1}\int_{l_0}^{l_1}S(h^{\prime\prime}\bar{z}+2h^{\prime}z_x+hz_{xx})dx-i\la^{-4+\alpha}\int_{l_0}^{l_1}hf_4\bar{z}dx.
\end{equation}
Using Nirenberg inequality Theorem (see \cite{NirenbergPaper}), Equations \eqref{pol1beam} and \eqref{FE-1-beam}, we have
\begin{equation}\label{nirenberg-ineq}
\|z_x\|_{L^2(l_0,l_1)}\leq \|z_{xx}\|_{L^2(l_0,l_1)}^{1/2} \|z\|_{L^2(l_0,l_1)}^{1/2}+\|z\|_{L^2(l_0,l_1)}\leq O(1).
\end{equation}
Estimation of the term $\displaystyle \left(i\la^{-1}\int_{l_0}^{l_1}h^{\prime\prime}S\bar{z}\right)$. Using Cauchy Schwarz  inequality, last estimation in \eqref{FE-1-beam} and the fact that $z$ is uniformly bounded in $L^2(0,L)$ and that $0<\alpha<1$, we get
\begin{equation}\label{eqStep1-1}
\left|i\la^{-1}\int_{l_0}^{l_1}h^{\prime\prime}S\bar{z}dx\right|\leq \frac{m^{\prime\prime}_h}{\la} \int_{l_0}^{l_1} \abs{S}\abs{z}dx\leq \frac{o(1)}{\la^{\frac{5-\alpha}{2}}}.
\end{equation}
Estimation of the term $\displaystyle \left(i\la^{-1}\int_{l_0}^{l_1}hSz_{xx}dx\right)$. Using Cauchy Schwarz inequality, the second and the last estimations in \eqref{FE-1-beam}, we get
\begin{equation}\label{eqStep1-2}
\left|i\la^{-1}\int_{l_0}^{l_1}hS\bar{z}_{xx}dx\right|\leq \frac{m^{\prime}_h}{\la}\left(\int_{l_0}^{l_1}\abs{S}^2dx\right)^{1/2}\left(\int_{l_0}^{l_1}\abs{z_{xx}}^2dx\right)^{1/2}=\frac{o(1)}{\la^{\frac{7}{2}-\frac{\alpha}{2}}}.
\end{equation}
Estimation of the term $\displaystyle \left(i\la^{-1}\int_{l_0}^{l_1}h^{\prime}Sz_{x}dx\right)$. Cauchy Schwarz inequality , the last estimation in  \eqref{FE-1-beam}, Equation \eqref{nirenberg-ineq} and the fact that $0<\alpha<1$, we get
\begin{equation}\label{eqStep1-3}
\left|i\la^{-1}\int_{l_0}^{l_1}h^{\prime}Sz_{x}dx\right|\leq \frac{m^{\prime}_h}{\la} \int_{l_0}^{l_1} \abs{S}\abs{z_x}dx\leq \frac{o(1)}{\la^{\frac{5-\alpha}{2}}}.
\end{equation}
Estimation of the term $\displaystyle \left(i\la^{-4+\alpha}\int_{l_0}^{l_1}hf_4\bar{z}dx\right)$. Using Cauchy-Schwarz inequality, the fact that $\|f_4\|_{L^2(0,L)}=o(1)$ and the fact that $z$ is uniformly bounded in $L^2(0,L)$, we get 
\begin{equation}\label{eqStep1-4}
\left|i\la^{-4+\alpha}\int_{l_0}^{l_1}hf_4\bar{z}dx\right|=\frac{o(1)}{\la^{4-\alpha}}.
\end{equation}
Thus, using Equations \eqref{eqStep1-1}-\eqref{eqStep1-4} in \eqref{eqStep1}, we get
\begin{equation}\label{Step1Eq}
\int_{l_0}^{l_1}h\abs{z}^2dx=\frac{o(1)}{\la^{\frac{5-\alpha}{2}}}\quad\text{and}\quad \int_{l_0+\varepsilon}^{l_1-\varepsilon}\abs{z}^2dx=\frac{o(1)}{\la^{\frac{5-\alpha}{2}}}.
\end{equation}
\textbf{Step 2.} Applying the interpolation theorem involving compact subdomain (\cite{Adams}, Theorem 4.23), we obtain
$$\|z_x\|_{L^2(l_0,l_1)}\leq \|z_{xx}\|_{L^2(l_0,l_1)}+\|z\|_{L^2(l_0+\varepsilon,l_1-\varepsilon)}.$$
Then, by using \eqref{Step1Eq} and that $0<\alpha<1$, we get
\begin{equation}\label{interpolation1}
\|z_x\|_{L^2(l_0,l_1)}=\frac{o(1)}{\la}.
\end{equation}
Also, Using Theorem  yields that 
\begin{equation}\label{interpolation2}
\|z\|_{L^2(l_0,l_1)}=\frac{o(1)}{\la}.
\end{equation}
Thus, using the first estimation in \eqref{FE-1-beam}, \eqref{interpolation1} and \eqref{interpolation2}, we obtain the first estimation in Lemma \ref{lem-inter1}.\\
Now, from Equation \eqref{pol5beam} we have $\displaystyle i\la y-z=\la^{-3+\alpha}f_3$, then
$$\|y\|_{H^2(l_0,l_1)}\leq \frac{1}{\la}\|z\|_{H^2(l_0,l_1)}+\frac{1}{\la^{3-\alpha}}\|f_3\|_{H^2(l_0,l_1)},$$
using the fact that $\alpha\in(0,1)$, $\|f_3\|_{H^2(l_0,l_1)}=o(1)$, and the first estimation in Lemma \ref{lem-inter1}, we obtain the second estimation of Lemma \ref{lem-inter1}. Using Nirenberg inequality Theorem (see \cite{NirenbergPaper}), Lemma \ref{FE-1-beam} and \eqref{interpolation2},   we get 
\begin{equation}\label{SX}
\|S_x\|_{L^2(l_0,l_1)}\leq \|S_{xx}\|^{\frac{1}{2}}_{L^2(l_0,l_1)}\|S\|^{\frac{1}{2}}_{L^2(l_0,l_1)}+\|S\|_{L^2(l_0,l_1)}=\frac{o(1)}{\la^{\frac{3-\alpha}{4}}}
\end{equation}
The proof has been completed. 
\end{proof}
\begin{rem}
It is easy to see the existence of $h(x)$. For example, we can take $\displaystyle h(x)=\cos\left(\frac{\pi(l_1-x)}{l_1-l_0}\right)$ and we get that $h(l_1)=-h(l_0)=1$, $h\in C^{\infty}([0,L])$, $\displaystyle\abs{g(x)}\leq 1$
and $\displaystyle\abs{g^{\prime}(x)}\leq \frac{\pi}{l_1-l_0}$.
\end{rem}
\begin{Lemma}\label{Lem5.5}
Assume that $\eta>0$. The solution $(u,v,y,z,\omega)\in D(\AA_3)$ of system \eqref{pol3beam}-\eqref{pol7beam} satisfies the following asymptotic behavior
\begin{equation}\label{eq1-5'}
\abs{y(l_0)}=\frac{o(1)}{\la^{2}}, \abs{y_x(l_0)}=\frac{o(1)}{\la^{2}}, \abs{y(l_1)}=\frac{o(1)}{\la^{2}},\,\text{and}\,\,\abs{y_x(l_1)}=\frac{o(1)}{\la^{2}}.
\end{equation}
Moreover,
\begin{equation}\label{eq2-5'}
\frac{1}{\la^{\frac{1}{2}}}\abs{ y_{xxx}(l_0^{-})}=\frac{o(1)}{\la^{\frac{5-\alpha}{4}}}\quad \text{and}\quad \frac{1}{\la^{\frac{1}{2}}}\abs{y_{xxx}(l_1^{+})}=\frac{o(1)}{\la^{\frac{5-\alpha}{4}}}.
\end{equation}
\end{Lemma}
\begin{proof}
For the proof of \eqref{eq1-5'}. 
Since $y,z\in H^2(0,L)$, Sobolev embedding theorem implies that $y,z\in C^{1}[0,L]$. Then, using the second estimation in Lemma \ref{lem-inter1} we get \eqref{eq1-5'}.\\
Define 
\begin{equation}\label{integral}
J(z)(x)=\frac{1}{\la^{3-\alpha}}\int_x^L\int_s^Lz(\tau)d\tau ds
\end{equation}
and
\begin{equation}\label{changeOfVar}
X=\frac{1}{i\la}[-S+J(f_4)].
\end{equation}
From \eqref{pol5beam} and \eqref{changeOfVar}, we get
\begin{equation}\label{Xxx}
X_{xx}=z.
\end{equation}
From \eqref{changeOfVar}, and using the fact that $\|f_4\|_{L^2(0,L)}=o(1)$, we have that 
\begin{equation}
\|\la X\|_{L^2(l_0,l_1)}\leq \|S\|_{L^2(l_0,l_1)}+\|J(f_4)\|_{L^2(l_0,l_1)}\leq \frac{o(1)}{\la^{\frac{3-\alpha}{2}}}.
\end{equation}
It follows that,
\begin{equation}\label{x}
\| X\|_{L^2(l_0,l_1)}=\frac{o(1)}{\la^{\frac{5-\alpha}{2}}}.
\end{equation}
Using Equation \eqref{Xxx} and Lemma \ref{lem-inter1} we get
\begin{equation}\label{x''}
\| X_{xx}\|_{L^2(l_0,l_1)}=\frac{o(1)}{\la},\,\| X_{xxx}\|_{L^2(l_0,l_1)}=\frac{o(1)}{\la},\, \text{and}\quad\| X_{xxxx}\|_{L^2(l_0,l_1)}=\frac{o(1)}{\la}.
\end{equation}
Now, using the interpolation inequality theorem \cite{NirenbergPaper}, and $\alpha\in (0,1)$, we get
\begin{equation}\label{x'}
\| X_{x}\|_{L^2(l_0,l_1)}\leq \|X_{xx}\|^{\frac{1}{2}}_{L^2(l_0,l_1)}\|X\|^{\frac{1}{2}}_{L^2(l_0,l_1)}+\|X\|_{L^2(l_0,l_1)}=\frac{o(1)}{\la^\frac{7-\alpha}{4}}.
\end{equation}
By using Equations \eqref{x}-\eqref{x'}, we get
\begin{equation}
\| X\|_{H^4(l_0,l_1)}=\frac{o(1)}{\la}.
\end{equation}
From the interpolation inequality (see Theorem 4.17 in \cite{Adams}), we have
\begin{equation}\label{laHalfX}
\|\la^{1/2}X\|_{H^2(l_0,l_1)}\lesssim\|\la^{1/2} X\|^{\frac{1}{2}}_{H^4(l_0,l_1)}\cdot\| \la^{1/2}X\|^{\frac{1}{2}}_{L^2(l_0,l_1)}=\frac{o(1)}{\la^{\frac{5-\alpha}{4}}}.
\end{equation} 
We note that $S=by_{xx}$ on $(0,l_0)\cup(l_1,L)$. Also, we have that $y\in H^4(0,l_0)$ and $y\in H^4(l_1,L)$.\\
From \eqref{changeOfVar}, we have
\begin{equation}\label{y}
by_{xxx}(l_0^-)=\left(J(f_4)-i\la X\right)_x(l_0)\quad\text{and}\quad by_{xxx}(l_1^+)=\left(J(f_4)-i\la X\right)_x(l_1)
\end{equation}
Dividing Equation \eqref{y} by $\la^{\frac{1}{2}}$, and using Equation \eqref{laHalfX} and the fact that $\| f_4\|_{L^2(0,L)}=o(1)$,  we get Equation \eqref{eq2-5'}. Thus, the proof of the Lemma is complete.
\end{proof}

\begin{Lemma}\label{lem-BW}
Assume that $\eta>0$. The solution $(u,v,y,z,\omega)\in D(\AA_3)$ of system \eqref{pol3beam}-\eqref{pol7beam} satisfies the following asymptotic behavior for every $h\in C^{2}([0,L])$ and $h(0)=h(L)=0$, and for every $g\in C^1[-L,0]$
\begin{equation}\label{Bbeam}
\begin{array}{ll}
\displaystyle\int_0^Lh^{\prime}\abs{z}^2dx+2b\int_0^L h^{\prime}\abs{y_{xx}}^2dx+b\int_0^{l_0} h^{\prime}\abs{y_{xx}}^2dx+b\int_{l_1}^{L} h^{\prime}\abs{y_{xx}}^2dx+\Re\left(2\int_0^L h^{\prime\prime}S\bar{y}_xdx\right)\\
\displaystyle-\Re\left(2\int_{l_0}^{l_1} hS_x\bar{y}_{xx}dx\right)-bh(l_0)\abs{y_{xx}(l_0^-)}^2+bh(l_1)\abs{y_{xx}(l_1^+)}^2=\frac{o(1)}{\la^{3-\alpha}}
\end{array}
\end{equation}
and 
\begin{equation}\label{Wave}
\int_{-L}^0 g^{\prime}\left(\abs{v}^2+a\abs{u_x}^2dx\right)-g(0)\abs{v(0)}^2-ag(0)\abs{u_x(0)}^2+ag(-L)\abs{u_x(-L)}^2=\frac{o(1)}{\la^{3-\alpha}}.
\end{equation}
\end{Lemma}
\begin{proof}
Multiply Equation \eqref{pol6beam} by $2h\bar{y_x}$ and integrate over $(0,L)$ we get
\begin{equation}\label{initial}
\Re\left(2i\la \int_0^Lzh\bar{y_x}dx\right)+\Re\left(2\int_0^LhS_{xx}\bar{y_x}dx\right)=\Re\left(2\la^{-3+\alpha}\int_0^L f_4h\bar{y_x}dx\right).
\end{equation}
From Equation \eqref{pol5beam} we have
$$i\la\bar{y}_x=-\bar{v}_x-\la^{-3+\alpha}(\bar{f_3})_x.$$
Then
\begin{equation}\label{1.1}
\Re\left(2i\la \int_0^Lzh\bar{y_x}dx\right)=\int_0^Lh^{\prime}\abs{z}^2dx-\Re\left(2\la^{-3+\alpha}\int_0^Lhz(\bar{f_3})_xdx\right).
\end{equation}
Estimation of the term $\displaystyle\Re\left(2\la^{-3+\alpha}\int_0^Lhz(\bar{f_3})_xdx\right)$. Using Cauchy-Schwarz inequality, the definition of $h$, the fact that $\|f_3\|_{H^2(0,L)}=o(1)$, and that $z$ is uniformly bounded in $L^2(0,L)$, we get
\begin{equation}\label{F3}
\left|\Re\left(2\la^{-3+\alpha}\int_0^Lhz(\bar{f_3})_xdx\right)\right|\lesssim \frac{1}{|\la|^{3-\alpha}}\left(\int_0^L\abs{z}^2dx\right)^{1/2}\left(\int_0^L\abs{(f_3)_x}^2dx\right)^{1/2}= \frac{o(1)}{\la^{3-\alpha}}.
\end{equation}
Then, using Equation \eqref{F3} in Equation \eqref{1.1} we get
\begin{equation}\label{I}
\Re\left(2i\la \int_0^Lzh\bar{y_x}dx\right)=\int_0^Lh^{\prime}\abs{z}^2dx+\frac{o(1)}{\la^{3-\alpha}}.
\end{equation}
Estimation of the term $\displaystyle\Re\left(2\la^{-3+\alpha}\int_0^L f_4h\bar{y_x}dx\right)$. Using Cauchy-Schwarz inequality, the definition of h, the fact that $\|f_4\|_{L^2(0,L)}=o(1)$, and that $\|y_x\|_{L^2(0,L)}\lesssim\|y_{xx}\|_{L^2(0,L)}=O(1)$, we get
\begin{equation}\label{II}
\Re\left(2\la^{-3+\alpha}\int_0^L f_4h\bar{y_x}dx\right)\lesssim \frac{1}{\la^{3-\alpha}}\left(\int_0^L\abs{f_4}^2dx\right)^{1/2}\left(\int_0^L\abs{y_{xx}}^2dx\right)^{1/2}\leq \frac{o(1)}{\la^{3-\alpha}}.
\end{equation}
For the second term of Equation \eqref{initial}, integrating by parts we get
\begin{equation}\label{2ndTerm}
\Re\left(2\int_0^LhS_{xx}\bar{y_x}dx\right)=\Re\left(2\int_0^Lh^{\prime\prime}S\bar{y}_xdx\right)+\Re\left(2\int_0^Lh^{\prime}S\bar{y}_{xx}dx\right)-\Re\left(2\int_0^LhS_{x}\bar{y}_{xx}dx\right).
\end{equation}
For the term $\displaystyle\Re\left(2\int_0^Lh^{\prime}S\bar{y}_{xx}dx\right)$, we have
\begin{equation}\label{2ndTerm1}
\Re\left(2\int_0^Lh^{\prime}S\bar{y}_{xx}dx\right)=2b\int_0^Lh^{\prime}\abs{y_{xx}}^2dx+\Re\left(2\int_0^Lh^{\prime}\bar{y}_{xx}\sqrt{d(x)}\kappa(\alpha)\int_{\R}\abs{\xi}^{\frac{2\alpha-1}{2}}\omega(x,\xi)d\xi dx\right).
\end{equation}
Estimation of the term $\displaystyle\Re\left(2\int_0^Lh^{\prime}\bar{y}_{xx}\sqrt{d(x)}\kappa(\alpha)\int_{\R}\abs{\xi}^{\frac{2\alpha-1}{2}}\omega(x,\xi)d\xi dx\right)$. Using Cauchy-Schwarz inequality, the definition of the functions $d(x)$ and $h$, and using Lemma \ref{First-Estimation-beam}, we get
\begin{equation}
\left|\Re\left(2\int_0^Lh^{\prime}\bar{y}_{xx}\sqrt{d(x)}\kappa(\alpha)\int_{\R}\abs{\xi}^{\frac{2\alpha-1}{2}}\omega(x,\xi)d\xi dx\right)\right|
=\frac{o(1)}{\la^{\frac{7-\alpha}{2}}}.
\end{equation}
This yields that,
\begin{equation}\label{III}
\Re\left(2\int_0^Lh^{\prime}S\bar{y}_{xx}dx\right)= 2b\int_0^Lh'\abs{y_{xx}}^2dx+\frac{o(1)}{\la^{\frac{7-\alpha}{2}}}.
\end{equation}
For the term $\displaystyle-\Re\left(2\int_0^LhS_{x}\bar{y}_{xx}dx\right)$. We have
\begin{equation}
\begin{array}{ll}
\displaystyle-\Re\left(2\int_0^LhS_{x}\bar{y}_{xx}dx\right)=-\Re\left(2b\int_0^{l_0}hy_{xxx}\bar{y}_{xx}dx+2b\int_{l_1}^Lhy_{xxx}\bar{y}_{xx}dx\right)-\Re\left(2\int_{l_0}^{l_1}hS_{x}\bar{y}_{xx}dx\right).
\end{array}
\end{equation}
Integrating the above Equation by parts, we get
\begin{equation}\label{2ndTerm2-2}
\begin{array}{ll}
\displaystyle-\Re\left(2\int_0^LhS_{x}\bar{y}_{xx}dx\right)= b\int_0^{l_0}h^{\prime}\abs{y_{xx}}^2dx+b\int_{l_1}^{L}h^{\prime}\abs{y_{xx}}^2dx-bh(l_0)\abs{y_{xx}(l_0)}^2\\ \displaystyle+bh(l_1)\abs{y_{xx}(l_1)}^2-\Re\left(2\int_{l_0}^{l_1}hS_{x}\bar{y}_{xx}dx\right).
\end{array}
\end{equation}
Then, substiting Equation \eqref{III} and \eqref{2ndTerm2-2} in Equation \eqref{2ndTerm}, we get
\begin{equation}\label{IIII}
\begin{array}{ll}
\displaystyle\Re\left(2\int_0^LhS_{xx}\bar{y_x}dx\right)= 2b\int_0^Lh^{\prime}\abs{y_{xx}}^2dx+b\int_0^{l_0}h^{\prime}\abs{y_{xx}}^2dx+b\int_{l_1}^{L}h^{\prime}\abs{y_{xx}}^2dx\\
\displaystyle-bh(l_0)\abs{y_{xx}(l_0)}^2+bh(l_1)\abs{y_{xx}(l_1)}^2 -\Re\left(2\int_{l_0}^{l_1}hS_{x}\bar{y}_{xx}dx\right)+\Re\left(2\int_0^Lh^{\prime\prime}S\bar{y}_xdx\right)+\frac{o(1)}{\la^{3-\alpha}}.
\end{array}
\end{equation}
Therefore, substituting Equations \eqref{I}, \eqref{II}, and \eqref{IIII} into \eqref{initial}, we get our desired result.\\
Now, we will prove Equation \eqref{Wave}. For this aim, multiply Equation \eqref{pol4beam} by $2g\bar{u}_x$ and integrate over $(-L,0)$, we get
\begin{equation}\label{bwave1}
\Re\left(2i\la\int_{-L}^0vg\bar{u}_xdx\right)-\Re\left(2a\int_{-L}^0gu_{xx}\bar{u}_xdx\right)=\Re\left(2\la^{-3+\alpha}\int_{-L}^0f_2g\bar{u}_xdx\right).
\end{equation}
We have $\displaystyle i\la\bar{u}_x=-\bar{v}_x-\la^{-3+\alpha}(\bar{f_1})_x$. Then,
\begin{equation}\label{bwave2}
\Re\left(2i\la\int_{-L}^0vg\bar{u}_xdx\right)=\int_{-L}^0g^{\prime}\abs{v}^2dx-g(0)\abs{v(0)}^2-\Re\left(2\la^{-3+\alpha}\int_{-L}^0vg(\bar{f_1})_xdx\right).
\end{equation}
Estimation of the term $\displaystyle \Re\left(2\la^{-3+\alpha}\int_{-L}^0vg(\bar{f_1})_xdx\right)$. Using Cauchy-schwarz inequality, $\|f_1\|_{H^2_L(-L,0)}=o(1)$, and the fact that $v$ is bounded in $L^2(0,L)$, we get
\begin{equation}
\left|\Re\left(2\la^{-3+\alpha}\int_{-L}^0vg(\bar{f_1})_xdx\right)\right|\leq \frac{o(1)}{\la^{3-\alpha}}.
\end{equation}
Then, using the above equation, \eqref{bwave2} becomes
\begin{equation}\label{bwave2'}
\Re\left(2i\la\int_{-L}^0vg\bar{u}_xdx\right)= \int_{-L}^0g^{\prime}\abs{v}^2dx-g(0)\abs{v(0)}^2+\frac{o(1)}{\la^{3-\alpha}}.
\end{equation}
Integrating the second term of Eqaution \eqref{bwave1}, we get
\begin{equation}\label{bwave3}
-\Re\left(2a\int_{-L}^0gu_{xx}\bar{u}_xdx\right)=a\int_{-L}^0g^{\prime}\abs{u_x}^2dx-g(0)\abs{u_x(0)}^2+ag(-L)\abs{u_x(-L)}^2.
\end{equation}
Estimation of the term $\displaystyle\Re\left(2\la^{-3+\alpha}\int_{-L}^0f_2g\bar{u}_xdx\right)$. Using Cauchy-Schwarz inequality, $\|f_2\|_{L^2(-L,0)}=o(1)$, and the fact that $u_x$ is bounded in $L^2(-L,0)$, we get
\begin{equation}\label{bwave4}
\left|\Re\left(2\la^{-3+\alpha}\int_{-L}^0f_2g\bar{u}_xdx\right)\right|= \frac{o(1)}{\la^{3-\alpha}}.
\end{equation}
Thus, summing Equations \eqref{bwave2'}, \eqref{bwave3} and  \eqref{bwave4} we obtain our desired result.
\end{proof}
\begin{Lemma}\label{Lem5.7}
Assume that $\eta>0$. The solution $(u,v,y,z,\omega)\in D(\AA_3)$ of system \eqref{pol3beam}-\eqref{pol7beam} satisfies the following asymptotic behavior 
\begin{equation}\label{O(1)}
\abs{\la y(0)}=O(1),\quad\abs{y_{xx}(l_0^-)}=O(1)\quad\text{and}\quad\abs{y_{xx}(l_{1}^+)}=O(1).
\end{equation}
\end{Lemma}
\begin{proof}
Take $g(x)=x+L$ in Equation \eqref{Wave}, we get
\begin{equation}
\int_{-L}^0 \left(\abs{v}^2+a\abs{u_x}^2dx\right)-L\abs{v(0)}^2-aL\abs{u_x(0)}^2=\frac{o(1)}{\la^{3-\alpha}}.
\end{equation}
From the above Equation and using the fact that $v$ and $u_x$ are uniformly bounded in $L^2(-L,0)$, we get
\begin{equation}\label{vv}
\abs{v(0)}=O(1)\quad\text{and}\quad \abs{u_x(0)}=O(1).
\end{equation}
From Equation \eqref{pol5beam} we have
$$i\la u=v+\la^{-3+\alpha}f_1.$$
Then, using \eqref{vv} and the fact that $\|f_1\|_{L^2(-L,0)}=o(1)$, we get
\begin{equation}
\abs{\la u(0)}\leq \abs{v(0)}+\la^{-3+\alpha}\abs{f_1(0)}\\
\leq \abs{v(0)}+\la^{-3+\alpha}\sqrt{L}\|f_1\|_{L^2(-L,0)}=O(1).
\end{equation}
From the continuity condition ($u(0)=y(0)$), we deduce that $\abs{\la y(0)}=O(1)$.
In order to prove the second term in the Equation \eqref{O(1)} we proceed as follows. From Equation \eqref{pol5beam} we have 
$$z=i\la y-\la^{-3+\alpha}f_3.$$
Substituting the above Equation into Equation \eqref{pol6beam}, we get
\begin{equation}\label{Diffeq}
-\la^2y+by_{xxxx}=\mathtt{F}\quad\text{on}\quad (0,l_0)\cup(l_1,L),
\end{equation}
where $\displaystyle \mathtt{F}=\la^{-3+\alpha}(f_4+i\la f_3)$.\\
Multiply Equation \eqref{Diffeq} by $2\zeta\bar{y}_x$, where $\zeta\in C^2[0,l_0]$, $\zeta(0)=\zeta^{\prime}(l_0)=0$ and $\zeta(l_0)=1$, and integrate over $(0,l_0)$, we get
\begin{equation}\label{zeta0}
-\Re\left(2\la^2\int_0^{l_0}\zeta y\bar{y}_xdx\right)+\Re\left(2b\int_0^{l_0}\zeta y_{xxxx}\bar{y}_xdx\right)=\Re\left(2\la^{-3+\alpha}\int_0^{l_0}(f_4+i\la f_3)\zeta\bar{y}_xdx\right).
\end{equation}
Estimation of the term $\displaystyle\Re\left(2\la^2\int_0^{l_0}\zeta y\bar{y}_xdx\right)$. Integrating by parts and using Equation \eqref{eq1-5'}, we get
\begin{equation}\label{zeta1}
-\Re\left(2\la^2\int_0^{l_0}\zeta y\bar{y}_xdx\right)=\int_0^{l_0}\zeta ^{\prime}\abs{\la y}^2dx+\frac{o(1)}{\la^2}.
\end{equation}
Integrating by parts the second term in Equation \eqref{zeta0}, and using Lemma \ref{Lem5.5}, we get
\begin{equation}\label{zeta2}
\begin{array}{ll}
\displaystyle\Re\left(2b\int_0^{l_0}\zeta y_{xxxx}\bar{y}_xdx\right)&\displaystyle=-\Re\left(2b\int_0^{l_0}\zeta^{\prime} y_{xxx}\bar{y}_xdx\right)-\Re\left(2b\int_0^{l_0}\zeta y_{xxx}\bar{y}_{xx}dx\right)+\Re\left(2b y_{xxx}(l_0^-)\bar{y}_x(l_0)\right)\\
&\displaystyle
=\Re\left(2b\int_0^{l_0}\zeta^{\prime\prime} y_{xx}\bar{y}_xdx\right)+3b\int_0^{l_0}\zeta^{\prime} \abs{y_{xx}}^2dx-b\zeta(l_0)\abs{y_{xx}(l_0)}^2+\frac{o(1)}{\la^\frac{11-\alpha}{4}}.
\end{array}
\end{equation}
Integrating by parts the last term of Equation \eqref{zeta0} and using \eqref{eq1-5'} and the fact that $\abs{f_3(l_0)}\lesssim \|f_3\|_{H^2_R(0,L)}=o(1) $, we get
\begin{equation}\label{zeta3}
\begin{array}{ll}
\displaystyle\Re\left(2\la^{-3+\alpha}\int_0^{l_0}(f_4+i\la f_3)\zeta\bar{y}_xdx\right)=\Re\left(2\la^{-3+\alpha}\int_0^{l_0}f_4\zeta\bar{y}_xdx\right)-\Re\left(2i\la^{-3+\alpha}\int_0^{l_0}f_3\zeta^{\prime}\la\bar{y}dx\right)\\
\displaystyle-\Re\left(2i\la^{-3+\alpha}\int_0^{l_0}(f_3)_x\zeta\la\bar{y}dx\right)+\frac{o(1)}{\la^{4-\alpha}}.
\end{array}
\end{equation}
Substituting Equations \eqref{zeta1}, \eqref{zeta2}, and \eqref{zeta3} in Equation \eqref{zeta0}, and using the fact that $\|U\|_{\mathcal{H}_3}=1$, $\|y_x\|_{L^2(0,L)}\leq c_p\|y_{xx}\|_{L^2(0,L)}=O(1)$ and $\|f_3\|_{H^2_R(0,L)}=o(1)$, we obtain our desired term.
For the last term in Equation \eqref{O(1)}, we proceed in a similar way as above and thus the proof of the Lemma is complete.
\end{proof}

\begin{Lemma}\label{Lem-yxx}
Assume that $\eta>0$. The solution $(u,v,y,z,\omega)\in D(\AA_3)$ of system \eqref{pol3beam}-\eqref{pol7beam} satisfies the following asymptotic behavior 
\begin{equation}
\abs{y_{xx}(l_0^{-})}=\frac{o(1)}{\la}\quad\text{and}\quad\abs{y_{xx}(l_1^+)}=\frac{o(1)}{\la}.
\end{equation}
\end{Lemma}
\begin{proof}
Equation \eqref{Diffeq} can be written as 
\begin{equation}\label{lineareq}
\left(\frac{\partial}{\partial x}-i\mu\right)\left(\frac{\partial}{\partial x}+i\mu\right)\left(\frac{\partial^2}{\partial^2 x}-\mu^2\right)y=\frac{1}{b}\mathtt{F},\, \text{on}\,\, (0,l_0)\cup(l_1,L).
\end{equation}
where $\displaystyle\mu=\sqrt{\frac{\la}{\sqrt{b}}}$. \\
\underline{On the interval $(0,l_0)$:}\\
Let 
$\displaystyle Y^1_{l_0}=\left(\frac{\partial}{\partial x}+i\mu\right)\left(\frac{\partial^2}{\partial x^2}-\mu^2\right)y$.
Solving on the interval $(0,l_0)$ the following Equation 
$$\left(\frac{\partial}{\partial x}-i\mu\right)Y^1_{l_0}=\frac{1}{b}\mathtt{F}$$
we get
\begin{equation}\label{lineareq2}
Y^1_{l_0}=K_1 e^{i\mu(x-l_0)}+\frac{1}{b}\int_{l_0}^x e^{i\mu(x-z)}\mathtt{F}(z)dz,
\end{equation}
where 
$$K_1=y_{xxx}(l_0^-)-\mu^2 y_x(l_0)+i\mu y_{xx}(l_0^-)-i\mu^3 y(l_0).$$
Let $\displaystyle Y^2_{l_0}=\left(\frac{\partial^2}{\partial x^2}-\mu^2\right)y$. We will solve the following differential equation
\begin{equation}\label{diffY}
\left(\frac{\partial}{\partial x}+i\mu\right)Y^2_{l_0}=Y^1_{l_0}.
\end{equation}
By using the solution $Y^1_{l_0}$, we obtain the solution of the differential equation \eqref{diffY}  
\begin{equation}\label{diffY2}
Y^2_{l_0}=K_2e^{-i\mu(x-l_0)}+\frac{K_1}{\mu}\sin((x-l_0)\mu)+\frac{1}{b\mu}\int_{l_0}^x\int_{l_0}^s e^{i\mu(2s-x-z)}\mathtt{F}(z)dzds.
\end{equation}
Integrating by parts the last term of the above Equation, we get
\begin{equation*}
\frac{1}{b\mu}\int_{l_0}^x\int_{l_0}^s e^{i\mu(2s-x-z)}\mathtt{F}(z)dzds=\frac{1}{b\mu}\int_{l_0}^x \sin((x-z)\mu)\mathtt{F}(z)dz.
\end{equation*}
Inserting the above Equation in \eqref{diffY2}, we get
\begin{equation}\label{Y2}
Y^2_{l_0}=K_2e^{-i\mu(x-l_0)}+\frac{K_1}{\mu}\sin((x-l_0)\mu)+\frac{1}{b\mu}\int_{l_0}^x \sin((x-z)\mu)\mathtt{F}(z)dz
\end{equation}
where
$$K_2=y_{xx}(l_0^-)-\mu^2 y(l_0).$$
Let $\displaystyle Y^3_{l_0}=\left(\frac{\partial}{\partial x}-\mu\right)y$. The solution of the following differential equation 
\begin{equation}
\left(\frac{\partial}{\partial x}+\mu\right)Y^3_{l_0}=Y^2_{l_0}
\end{equation}
is given by
\begin{equation}\label{Y3}
\begin{array}{ll}
\displaystyle Y^3_{l_0}=K_3e^{-\mu(x-l_0)}+\frac{K_2}{\mu(1-i)}\left[e^{-i\mu(x-l_0)}-e^{-\mu(x-l_0)}\right]
+\frac{1}{b\mu}\int_{l_0}^{x}\int_{l_0}^{s}e^{\mu(s-x)}\sin((s-z)\mu)\mathtt{F}(z)dzds\\
\displaystyle-\frac{K_1}{2\mu^2}\left[\cos((x-l_0)\mu)-\sin((x-l_0)\mu)-e^{-\mu(x-l_0)}\right]
\end{array}
\end{equation}
where
$$K_3=y_x(l_0)-\mu y(l_0).$$
Take $x=0$ in \eqref{Y3} and multiply the equation by $\mu e^{-\mu l_0}$, we get
\begin{equation}\label{alpha}
\begin{array}{lll}
\displaystyle\sqrt{\frac{5}{4}}\abs{y_{xx}(l_0^-)}\leq \abs{\mu y_x(0)}e^{-\mu l_0}+\abs{\mu^2 y(0)}e^{-\mu l_0}+\abs{y_x(l_0)}+\abs{\mu y(l_0)}+\frac{1}{2\mu}\abs{y_{xxx}(l_0)}\left[2e^{-\mu l_0}+1\right]\\
\displaystyle
+\frac{1}{2}\abs{\mu^2 y(l_0)}\left(2e^{-\mu l_0}+1\right)+\frac{1}{2}\abs{\mu y_x(l_0)}\left(2e^{-\mu l_0}+1\right)+ \frac{1}{\sqrt{2}}\abs{\mu^2 y(l_0)}\left(e^{-\mu l_0}+1\right)\\
\displaystyle
+\left(1+\frac{1}{\sqrt{2}}\right)\abs{y_{xx}(l_0^-)}e^{-\mu l_0}+\left|\frac{1}{b}\int_{0}^{l_0}\int_{l_0}^{s}e^{(s-l_0)\mu}\sin((s-z)\mu)\mathtt{F}(z)dzds\right|.
\end{array}
\end{equation}
For the integral in the above Equation, we have
\begin{equation}\label{integral}
\begin{array}{ll}
\displaystyle\frac{1}{b}\int_{0}^{l_0}\int_{l_0}^{s}e^{\mu(s-l_0)}\sin((s-z)\mu)\mathtt{F}(z)dzds=\frac{1}{b\la^{3-\alpha}}\int_{0}^{l_0}\int_{l_0}^{s}e^{\mu(s-l_0)}\sin((s-z)\mu)f_4(z)dzds\\
\displaystyle
+\frac{1}{b\la^{3-\alpha}}\int_{0}^{l_0}\int_{l_0}^{s}e^{\mu(s-l_0)}\sin((s-z)\mu) i\la f_3(z)dzds.
\end{array}
\end{equation}
Estimation of the first integral in the right side of Equation \eqref{integral}.
\begin{equation}\label{integ1}
\left|\frac{1}{b\la^{3-\alpha}}\int_{0}^{l_0}\int_{l_0}^{s}e^{\mu(s-l_0)}\sin((s-z)\mu)f_4(z)dzds\right|\leq \frac{l_0^{3/2}}{\la^{3-\alpha}}\int_0^{l_0}\abs{f_4(z)}^2dz =\frac{o(1)}{\la^{3-\alpha}}.
\end{equation}
Estimation of the second term in the second side of Equation \eqref{integral}. Integrating by parts, and using the fact that $\abs{f_3(0)}\lesssim\|(f_3)_x\|_{L^2(0,L)}=o(1)$ and $\abs{f_3(l_0)}\lesssim\|(f_3)_x\|_{L^2(0,L)}=o(1)$ and $\|f_3\|_{L^2(0,L)}=o(1)$, we get
\begin{equation}\label{integ2}
\small{\begin{array}{rl}
\displaystyle\left|\frac{1}{b\la^{\ell}}\int_0^{l_0}\int_{l_0}^{s}e^{\mu(s-l_0)}\sin((s-z)\mu) i\la f_3(z)dzds\right|&\displaystyle\vspace{0.2cm}=\left|\frac{i\la e^{-\mu l_0}}{b\la^{\ell}}\int_{0}^{l_0}\left(\int_{0}^{z}e^{\mu s}\sin((s-z)\mu)ds\right) f_3(z)dz\right|\\[0.1in] 
&\displaystyle
=\left|\frac{i\la e^{-\mu l_0}}{2b\mu \la^{\ell}}\int_0^{l_0}\left(\cos(\mu z)+\sin(\mu z)-e^{\mu z}\right)f_3(z)dz\right|\\[0.1in]
&\displaystyle \vspace{0.2cm}
\leq \frac{\la}{2b e^{\mu l_0}\la^{\ell}\mu}\left[2\int_0^{l_0}\abs{f_3(z)}dz+\left|\int_0^{l_0}e^{\mu z}f_3(z)dz\right|\right]\\[0.1in]
&\displaystyle
\leq \frac{\la}{2be^{\mu l_0}\la^{\ell}\mu}\left[2\int_0^{l_0}\abs{f_3(z)}dz+\frac{1}{\mu}(\abs{f_3(0)}+\abs{f_3(l_0)})e^{\mu l_0}\right.\\[0.1in]
&\displaystyle
\left.+\frac{e^{\mu l_0}}{\mu}\int_0^{l_0}\abs{f^{\prime}_3(z)}dz\right]=\frac{o(1)}{\la^{3-\alpha}}.
\end{array}}
\end{equation}
Hence,
\begin{equation}\label{Integ}
\left|\frac{1}{b}\int_{0}^{l_0}\int_{l_0}^{s}e^{\mu(s-l_0)}\sin((s-z)\mu)\mathtt{F}(z)dzds\right|=\frac{o(1)}{\la^{3-\alpha}}.
\end{equation}
Since $\displaystyle e^{\sqrt{\zeta}}\geq \frac{1}{5}\zeta^2$, and taking $\zeta=\frac{\la l_0^2}{\sqrt{b}}$ $\forall \la\in \R_{+}^{*}$, and using Lemmas \ref{Lem5.5} and \ref{Lem5.7}, and Equation \eqref{Integ} in Equation \eqref{alpha}, we get that 
$$\abs{y_{xx}(l_0^-)}=\frac{o(1)}{\la}.$$
\underline{On the interval $(l_1,L)$: } \\
Proceeding with a similar computation as on $(0,l_0)$, we get
\begin{equation}\label{Y2'}
Y^2_{l_1}=\mathbf{K}_2e^{-i\mu(x-l_1)}+\frac{\mathbf{K}_1}{\mu}\sin((x-l_1)\mu)+\frac{1}{b\mu}\int_{l_1}^x \sin((x-z)\mu)\mathtt{F}(z)dz
\end{equation}
where
$$\mathbf{K}_1=y_{xxx}(l_1^+)-\mu^2 y_x(l_1)+i\mu y_{xx}(l_1^+)-i\mu^3 y(l_1)$$
and
$$\mathbf{K}_2=y_{xx}(l_1^+)-\mu^2y(l_1).$$
We will solve the following differential Equation
\begin{equation}\label{Y-beta}
\left(\frac{\partial}{\partial x}-\mu\right)Y^3_{l_1}=Y^2_{l_1}
\end{equation}
where
$\displaystyle Y^3_{l_1}=\left(\frac{\partial}{\partial x}+\mu\right)y.$
The solution of \eqref{Y-beta} is
\begin{equation}\label{forBeta}
\begin{array}{ll}
\displaystyle Y^3_{l_1}=\mathbf{K}_3e^{\mu(x-l_1)}-\frac{\mathbf{K}_2}{1+i}\left[e^{-i\mu(x-l_1)}-e^{\mu(x-l_1)}\right]\\
\displaystyle -\frac{\mathbf{K}_1}{2\mu^2}\left[\cos((x-l_1)\mu)+\sin((x-l_1)\mu)-e^{\mu(x-l_1)}\right]+
\frac{1}{b\mu}\int_{l_1}^x\int_{l_1}^{s}e^{\mu(x-s)}\sin((s-z)\mu)\mathtt{F}(z)dzds
\end{array}
\end{equation}
where $$\mathbf{K}_3=y_{x}(l_1)+\mu y(l_1).$$
Taking $x=L$ in Equation \eqref{forBeta} and multiplying by $\mu e^{-\mu(L-l_1)}$, we get
\begin{equation}\label{modulus}
\begin{array}{lll}
\displaystyle
\frac{1}{\sqrt{2}}\,\abs{y_{xx}(l_1^+)}\leq \abs{\mu y_x(l_1)}+\abs{\mu^2 y(l_1)}+\frac{e^{-\mu(L-l_1)}}{\sqrt{2}}\abs{y_{xx}(l_1^+)}+\frac{1}{\sqrt{2}}\abs{\mu^2y_{x}(l_1)}\left(e^{-\mu(L-l_1)}+1\right)\\
\displaystyle
+\frac{1}{\mu}y_{xxx}(l_1^+)\left(2e^{-\mu(L-l_1)}+1\right)+\abs{\mu y_x(l_1)}\left(2e^{-\mu(L-l_1)}+1\right)
+2\abs{y_{xx}(l_1^+)}e^{-\mu(L-l_1)}
\\
 \displaystyle+\abs{\mu^2 y(l_1)}\left(2e^{-\mu(L-l_1)}+1\right)+
\left|\frac{1}{b}\int_{l_1}^L\int_{l_1}^{s}e^{-\mu(s-l_1)}\sin((s-z)\mu)\mathtt{F}(z)dzds\right|.
\end{array}
\end{equation}
Using the same computation in Equation \eqref{integral}, we get
\begin{equation}\label{Integ2}
\frac{1}{b}\int_{l_1}^x\int_{l_1}^{s}e^{-\mu(s-l_1)}\sin((s-z)\mu)\mathtt{F}(z)dzds=\frac{o(1)}{\la^{3-\alpha}}.
\end{equation}
Since $\displaystyle e^{\sqrt{\zeta}}\geq \frac{1}{5}\zeta^2$ and taking $\zeta=\la (L-l_1)^2$, and using Lemmas \ref{Lem5.5}, \ref{Lem5.7}, and Equation \eqref{Integ2} in \eqref{modulus}, we get
$$\abs{y_{xx}(l_1^+)}=\frac{o(1)}{\la}.$$
Thus the proof of this Lemma is complete.
\end{proof}

\begin{Lemma}\label{Stab-beam}
Assume that $\eta>0$. The solution $(u,v,y,z,\omega)\in D(\AA_3)$ of system \eqref{pol3beam}-\eqref{pol7beam} satisfies the following asymptotic behavior 
\begin{equation}
\int_0^L\abs{z}^2dx=\frac{o(1)}{\la^2}\quad\text{and}\quad b\int_0^L\abs{y_{xx}}^2dx=\frac{o(1)}{\la^2}.
\end{equation}
\end{Lemma}
\begin{proof}
Taking $h(x)=x\theta_1(x)$ in Equation \eqref{Bbeam}, where $\theta_1\in C^1([0,L])$ is defined as follows
 \begin{equation}
\theta_{1}(x)=\left\{\begin{array}{ccc}
1&\text{if}&x\in [0,l_0],\\
0&\text{if}& x\in [l_1,L],\\
0\leq\theta_1\leq 1&& elsewhere.
\end{array}\right.
\end{equation}
Estimation of the term $\displaystyle \Re\left(2\int_0^Lh^{\prime\prime}S\bar{y}_xdx\right)$. Using Cauchy-Schwarz inequality, the definition of $h$, Lemmas \ref{First-Estimation-beam} and \ref{lem-inter1}, we get
\begin{equation}\label{BeamEst1}
\left|\Re\left(2\int_0^Lh^{\prime\prime}S\bar{y}_xdx\right)\right|=\frac{o(1)}{\la^{\frac{7-\alpha}{2}}}.
\end{equation}
Estimation of the term $\displaystyle \Re\left(2\int_{l_0}^{l_1}hS_x\bar{y}_{xx}dx\right)$. Using Cauchy-Schwarz inequality, the definition of $h$, Lemmas \ref{First-Estimation-beam} and \ref{lem-inter1}, we get
\begin{equation}\label{BeamEst2}
\left|\Re\left(2\int_{l_0}^{l_1}hS_x\bar{y}_{xx}dx\right)\right|\leq \left(\int_{l_0}^{l_1}\abs{S_x}^2dx\right)^{1/2} \left(\int_{l_0}^{l_1}\abs{y_{xx}}^2dx\right)^{1/2}=\frac{o(1)}{\la^{\frac{11-\alpha}{4}}}.
\end{equation}
Using Lemma \ref{Lem-yxx}, Equations \eqref{BeamEst1} and \eqref{BeamEst2} in \eqref{Bbeam}, we get 
\begin{equation}
\int_0^Lh^{\prime}\abs{z}^2dx+2b\int_0^Lh^{\prime}\abs{y_{xx}}^2dx+b\int_0^{l_0}h^{\prime}\abs{y_{xx}}^2dx+b\int_{l_1}^Lh^{\prime}\abs{y_{xx}}^2dx=\frac{o(1)}{\la^2}.
\end{equation}
By using the definition of $h$, we obtain
\begin{equation}\label{BEAMfinal1}
\int_0^{l_0}\abs{z}^2dx=\frac{o(1)}{\la^2}\quad\text{and}\quad b\int_0^{l_0}\abs{y_{xx}}^2dx=\frac{o(1)}{\la^2}.
\end{equation}
Now, taking $h(x)=x\theta_2(x)$ in Equation \eqref{Bbeam}, where $\theta_2\in C^1([0,L])$ is defined as follows
\begin{equation}
\theta_{2}(x)=\left\{\begin{array}{ccc}
1&\text{if}&x\in [l_1,L],\\
0&\text{if}& x\in [ 0,l_0],\\
0\leq\theta_2\leq 1&& elsewhere.
\end{array}\right.
\end{equation}
Proceeding in a similar way as above we get
\begin{equation}\label{BEAMfinal2}
\int_{l_1}^L\abs{z}^2dx=\frac{o(1)}{\la^2}\quad\text{and}\quad b\int_{l_1}^{L}\abs{y_{xx}}^2dx=\frac{o(1)}{\la^2}.
\end{equation}
Therefore, using the third estimation of \eqref{FE-1-beam}, Lemma \ref{lem-inter1} and combining Equations \eqref{BEAMfinal1}, and \eqref{BEAMfinal2} we get our desired result.
\end{proof}

\begin{Lemma}\label{estOfyxxx}
Assume that $\eta>0$. The solution $(u,v,y,z,\omega)\in D(\AA_3)$ of system \eqref{pol3beam}-\eqref{pol7beam} satisfies the following asymptotic behavior 
\begin{equation}
\abs{y_{xxx}(0)}=o(1).
\end{equation}
\end{Lemma}

\begin{proof}
From the interpolation inequality Theorem (see \cite{NirenbergPaper}), and using the fact that $y\in H^4(0,l_0)$, $\|f_4\|_{L^2(0,l_0)}=o(1)$, Equation \eqref{pol6beam} and Lemma \ref{Stab-beam}, we get
\begin{equation}\label{1yxxx0}
\begin{array}{lll}
\displaystyle\|y_{xxx}\|_{L^2(0,l_0)}& \displaystyle\lesssim \|y_{xxxx}\|^{1/2}_{L^2(0,l_0)}\cdot\|y_{xx}\|^{1/2}_{L^2(0,l_0)}+\|y_{xx}\|_{L^2(0,l_0)}\\
&\displaystyle
\leq \left[\la^{\frac{1}{2}}\|z\|^{\frac{1}{2}}_{L^2(0,l_0)}+\frac{1}{\la^{\frac{3-\alpha}{2}}}\|f_4\|^{\frac{1}{2}}_{L^2(0,l_0)}\right]\|y_{xx}\|^{\frac{1}{2}}_{L^2(0,l_0)}+\|y_{xx}\|_{L^2(0,l_0)}=\frac{o(1)}{\la^{\frac{1}{2}}}.
\end{array}
\end{equation}
Using Equation \eqref{pol6beam} and the definition of $S$ on $(0,l_0)$ and \eqref{Stab-beam}, we get 
\begin{equation}\label{2yxxx0}
\|y_{xxxx}\|_{L^2(0,l_0)}=o(1).
\end{equation}
Then, from \eqref{1yxxx0} and  \eqref{2yxxx0}, we get 
$$
\|y_{xxx}\|_{H^1(0,l_0)}=o(1).
$$
Since $H^1(0,l_0)\subset C([0,l_0])$, then we get the desired result. Thus, the proof is complete.
\end{proof}
\begin{Lemma}\label{Wwave}
Assume that $\eta>0$. The solution $(u,v,y,z,\omega)\in D(\AA_3)$ of system \eqref{pol3beam}-\eqref{pol7beam} satisfies the following asymptotic behavior 
\begin{equation}
\int_{-L}^0\left(\abs{v}^2+a\abs{u_x}^2\right)dx=o(1).
\end{equation}
\end{Lemma}
\begin{proof}
Using the transmission condition and Lemma \ref{estOfyxxx}, we have
\begin{equation}\label{ux0}
\abs{u_x(0)}=\abs{y_{xxx}(0)}=o(1).
\end{equation}
Now, let $q\in C^2([0,l_0])$ such that $q(l_0)=q_x(l_0)=0$, $q(0)=1$. Multiply Equation \eqref{Diffeq} by $q\bar{y}_x$ and integrate over $(0,l_0)$, and using the fact that $\|f_3\|_{H^2_R(0,L)}=o(1)$, $ \|f_4\|_{L^2(0,L)}=o(1)$, and $\|y_{x}\|_{L^2(0,l_0)}\leq\|y_{x}\|_{L^2(0,L)}\lesssim \|y_{xx}\|_{L^2(0,L)}=O(1)$ , we get
\begin{equation}\label{q-eq}
\frac{1}{2}\int_0^{l_0}q^{\prime}\left(\abs{\la y}^2+3\abs{y_{xx}}^2\right)dx+b\int_0^{l_0}q^{\prime\prime}y_{xx}\bar{y}_xdx+\frac{1}{2}\abs{\la y(0)}^2-\Re\left(by_{xxx}(0)\bar{y}_x(0)\right)=\frac{o(1)}{\la^{2-\alpha}}.
\end{equation}
By using Lemmas \ref{Stab-beam}, \ref{estOfyxxx} and the fact that $\abs{y_x(0)}\lesssim \|y_{xx}\|_{L^2(0,L)}=O(1)$ and $\|y_{x}\|_{L^2(0,l_0)}\leq\|y_{x}\|_{L^2(0,L)}\lesssim \|y_{xx}\|_{L^2(0,L)}=O(1)$ in Equation \eqref{q-eq}, we obtain
\begin{equation}\label{q-eq1}
\abs{\la y(0)}^2=o(1).
\end{equation}
From the transmission conditions we have
$$\abs{\la u(0)}=\abs{\la y(0)}=o(1).$$
It follows from the above Equation and  Equation \eqref{pol3beam}  that
\begin{equation}\label{v0}
\abs{v(0)}=o(1).
\end{equation}
Thus, by taking $g(x)=x+L$ in Equation \eqref{Wave}, and using Equations \eqref{ux0} and \eqref{v0} we get our desired result.
\end{proof}

\noindent \textbf{Proof of Theorem \ref{pol}.} 
From Lemmas \ref{Stab-beam} and \ref{Wwave}, we get that $\|U\|_{\HH_3}=o(1)$, which contradicts \eqref{pol1beam}. Consequently, condition \eqref{R2} holds. This implies, from Theorem \ref{bt}, the energy decay estimation \eqref{Energypolbeam}. The proof is thus complete.
\begin{rem}
$\bullet$
The result in \cite{Fathi-2015} can be improved. Indeed, in \cite{Fathi-2015}, Fathi considered a Euler-Bernoulli beam and wave equations coupled through transmission conditions. The damping is locally distributed and acts on the wave equation, and the rotation vanishes at the connecting point $(y_x(\ell)=0)$. The system is given in the left side of Equation \eqref{Fathi}. 
He proved the polynomial stability with energy decay rate of type $t^{-2}$. By using similar computations as in Section \ref{DAMPED-WAVE} by taking $\alpha=1$, and by solving the ordinary differential equations in Section \ref{Section I} we can reach that $\abs{y_{xx}(\ell)}=o(1)$. Thus, with the same technique of the proof of Section \ref{DAMPED-WAVE}, we can reach that energy of the system  (the left system in Equation \eqref{Fathi}) of the mentioned paper satisfies the decay rate $t^{-4}$. \\
$\bullet$ In \cite{Fathi-2015}, when the damping acts on the beam equation the energy decay rate reached was $t^{-2}$ (the left system in Equation \eqref{Fathi}), when taking the condition $y_x(\ell)=0$ at the connecting point. However, in this paper, we proved in section \ref{Section I} the polynomial energy decay rate of type $t^{-1}$, when taking $y_{xx}(0)=0$ at the connecting point. From this comparison, we see that the boundary conditions play a critical role in the energy decay rate for the system \eqref{AUG1B}.
\end{rem}

\section{(EBB)$_{FKV}$ Model}\label{Beam-Alone}
\noindent In this section, we consider the Euler-Bernoulli beam with localized fractional Kelvin-Voigt damping. We study the polynomial stability of the system \eqref{beamAlone}.
In Theorem \ref{theorem1}, taking the input $V(x,t)=\sqrt{d(x)}y_{xxt}(x,t)$, then  using \eqref{Caputo}, we get the output $O$ is given by 
$$
O(x,t)=\sqrt{d(x)}I^{1-\alpha,\eta}y_{xxt}(x,t)=\frac{\sqrt{d(x)}}{\Gamma(1-\alpha)}\int_0^t(t-s)^{-\alpha}e^{-\eta(t-s)}\partial_sy_{xx}(x,s)ds=\sqrt{d(x)}\partial_t^{\alpha,\eta}y_{xx}(x,t).
$$
Then, we deduce that system \eqref{beamAlone} can be recast into the following augmented model 
\begin{equation}\label{AUG1Beam}
\left\{\begin{array}{ll}
\displaystyle{y_{tt}+\left(by_{xx}+\sqrt{d(x)}\kappa(\alpha)\int_{\mathbb{R}}\abs{\xi}^{\frac{2\alpha-1}{2}}\omega(x,\xi,t)d\xi\right)_{xx}=0},&(x,t)\in (0,L)\times \times \R_{\ast}^+,\\[0.1in]
\omega_t(x,\xi,t)+\left(|\xi|^2+\eta\right)\omega(x,\xi,t)-\sqrt{d(x)}y_{xxt}(x,t)\abs{\xi}^{\frac{2\alpha-1}{2}}=0,&(x,\xi,t)\in (0,L)\times \mathbb{R}\times \R_{\ast}^+,
\end{array}
\right.
\end{equation}
with the boundary conditions 
\begin{equation}\label{AUG2-las}
\begin{array}{ll}
y(0,t)=y_{x}(0,t)=0,\,y_{xx}(L,t)=0, \,y_{xxx}(L,t)=0,&t\in (0,\infty),
\end{array}
\end{equation}
and with the following initial conditions 
\begin{equation}\label{AUG3B-las}
\begin{array}{lll}
y(x,0)=y_0(x),\, y_t(x,0)=y_1(x),\,\omega(x,\xi,0)=0,\, x\in (0,L),\xi\in\mathbb{R}.
\end{array}
\end{equation}
The energy of the system \eqref{AUG1Beam}-\eqref{AUG3B-las} is given by  
\begin{equation*}
E_3(t)=\frac{1}{2}\int_0^L\left(\abs{y_t}^2+b\abs{y_{xx}}^2\right)dx+\frac{\kappa(\alpha)}{2}\int_{0}^L\int_{\R}\abs{\omega(x,\xi,t)}^2d\xi dx.
\end{equation*}
\begin{lemma}\label{Denergybeam-las}
Let $U=(y,y_t,\omega)$ be a regular solution of the System \eqref{AUG1Beam}-\eqref{AUG3B-las}. Then, the energy $E_4(t)$ satisfies the following estimation 
\begin{equation}\label{denergybeam-las}
\frac{d}{dt}E_4(t)=-\kappa(\alpha)\int_{0}^L\int_{\mathbb{R}}(\xi^2+\eta)\abs{\omega(x,\xi,t)}^2d\xi dx.
\end{equation}
\end{lemma}
$\newline$
\noindent Since $\alpha\in (0,1)$, then $\kappa(\alpha)>0$, and therefore $\displaystyle\frac{d}{dt}E_4(t)\leq 0$. Thus, system \eqref{AUG1Beam}-\eqref{AUG3B-las} is dissipative in the sense that its energy is a non-increasing function with respect to time variable $t$. Now, we define the following Hilbert energy space $\mathcal{H}_{4}$ by 
$$
\mathcal{H}_{4}= H^2_L(0,L)\times L^2(0,L)\times W,
$$
where $W=L^2\left((0,L)\times \R\right)$ and $H^2_L(0,L)=\left\{y\in H^2(0,L);\, y(0)=y_x(0)=0\right\}$.\\
The energy space $\mathcal{H}_4$ is equipped with the inner product defined by 
$$
\begin{array}{lll}
\displaystyle
\left<U,U_1\right>_{\mathcal{H}_4}&=&\displaystyle
\int_{0}^L  z\bar{z_1}dx+b\int_{0}^L y_{xx}(\overline{y_1})_{xx}dx
+\kappa(\alpha)\int_{0}^L \int_{\mathbb{R}}\omega(x,\xi)\overline{\omega_1}(x,\xi)d\xi dx,
\end{array}
$$
for all $U=(y,z,\omega)$ and $U_1=(y_1,z_1,\omega_1)$ in $\mathcal{H}_{4}$. We use $\|U\|_{\mathcal{H}_{4}}$ to denote the corresponding norm. We define the unbounded linear operator $\AA_{4}:D(\AA_{4})\subset\mathcal{H}_{4}\rightarrow\mathcal{H}_{4}$ by 
\begin{equation*}
D(\mathcal{A}_{4})=\left\{\begin{array}{c}
\displaystyle{U=(y,z,\omega)\in \mathcal{H}_4; z\in H^2_L(0,L)},\\[0.1in]  
\displaystyle{\left(by_{xx}+\sqrt{d(x)}\kappa(\alpha)\int_{\mathbb{R}}\abs{\xi}^{\frac{2\alpha-1}{2}}\omega(x,\xi,t)d\xi\right)_{xx}}\in L^2(0,L),\vspace{0.2cm}\\ [0.1in]
-\left(\abs{\xi}^2+\eta\right)\omega(x,\xi)+\sqrt{d(x)}z_{xx}|\xi|^{\frac{2\alpha-1}{2}},\quad|\xi|\omega(x,\xi)\in W,\\[0.1in]
y_{xx}(L)=0,\, \text{and}\, y_{xxx}(L)=0.
\end{array}
\right\},
\end{equation*}
and for all $U=(y,z,\omega)\in D(\mathcal{A}_{4})$, 
$$
\mathcal{A}_{4}(y,z,\omega)^{\top}=\begin{pmatrix}
z\\\vspace{0.2cm}\displaystyle{-\left(by_{xx}+\sqrt{d(x)}\kappa(\alpha)\int_{\mathbb{R}}\abs{\xi}^{\frac{2\alpha-1}{2}}\omega(x,\xi,t)d\xi\right)_{xx}}\\ \vspace{0.2cm}
-\left(\abs{\xi}^2+\eta\right)\omega(x,\xi)+\sqrt{d(x)}z_{xx}|\xi|^{\frac{2\alpha-1}{2}}
\end{pmatrix}.
$$
\noindent If $U=(y,y_t,\omega)$ is a regular solution of system \eqref{AUG1Beam}-\eqref{AUG3B-las}, then the system can be rewritten as evolution equation on the Hilbert space $\mathcal{H}_4$ given by
\begin{equation}\label{evolution-las}
U_t=\mathcal{A}_{4}U,\quad U(0)=U_0,
\end{equation}
where $U_0=(y_0,y_1,0)$.\\
We can see that (in a smiliar way as in Section \ref{DAMPED-WAVE}) the unbounded linear operator $\mathcal{A}_4$ is m-dissipative in the energy space $\mathcal{H}_4$. Also, the $C_0$-semigroup of contractions $e^{t\mathcal{A}_4}$ is strongly stable on $\mathcal{H}_4$ in the sense that $\displaystyle{\lim_{t\to+\infty}\left\|e^{t\mathcal{A}_{4}}U_0\right\|_{\mathcal{H}_{4}}}=0$. 
\begin{theoreme}\label{polbeam-las}
Assume that $\eta>0$. The $C_0-$semigroup $(e^{t\AA_4})_{t\geq 0}$ is polynomially stable; i.e. there exists constant $C_4>0$ such that for every $U_0\in D(\AA_4)$, we have 
\begin{equation}\label{Energypolbeam-las}
E_4(t)\leq \frac{C_4}{t^{\frac{2}{1-\alpha}}}\|U_0\|^2_{D(\AA_4)},\quad t>0,\,\forall\,U_0\in D(\AA_4).
\end{equation}
\end{theoreme}
\noindent According to Theorem \ref{bt}, by taking $\ell=1-\alpha$, the polynomial energy decay \eqref{Energypolbeam-las} holds if the following conditions 
\begin{equation}\label{P1}\tag{${\rm{P_1}}$}
i\R\subset \rho(\mathcal{A}_4),
\end{equation}
and
\begin{equation}\label{P2}\tag{${\rm{P_2}}$}
\sup_{\la\in \R}\left\|(i\la I-\AA_4)^{-1}\right\|_{\mathcal{L}(\mathcal{H}_4)}=O\left(\abs{\la}^{1-\alpha}\right)
\end{equation}
are satisfied. 
\noindent Since condition \eqref{P1}is already satisfied (similar way as in Subsection \ref{subsec}), we still need to prove condition \eqref{P2}. For this purpose we will use an argument of contradiction. Suppose that \eqref{P2} is false, then there exists $\left\{\left(\la_n,U_n:=(y_n,z_n,\omega_n(\cdot,\xi))^\top\right)\right\}\subset \R^{\ast}\times D(\AA_4)$ with 
\begin{equation}\label{pol1beam-las}
\abs{\la_n}\to +\infty \quad \text{and}\quad \|U_n\|_{\mathcal{H}_4}=1, 
\end{equation}
such that 
\begin{equation}\label{pol2-las}
\left(\la_n^{1-\alpha}\right)\left(i\la_nI-\AA_4\right)U_n=F_n:=(f_{1,n},f_{2,n},f_{3,n}(\cdot,\xi))^{\top}\to 0 \ \ \text{in}\ \ \mathcal{H}_4. 
\end{equation}
For simplicity, we drop the index $n$. Equivalently, from \eqref{pol2-las}, we have 
\begin{eqnarray}
i\la y-z&=&\dfrac{f_1}{\la^{1-\alpha}} \quad\text{in}\ H_L^2(0,L),\\
i\la z+S_{xx}&=&\dfrac{f_2}{\la^{1-\alpha}} \quad \text{in}\ L^2(0,L),\\
(i\la+\xi^2+\eta)\omega(x,\xi)-\sqrt{d(x)}z_{xx}|\xi|^{\frac{2\alpha-1}{2}}&=&\dfrac{f_3(x,\xi)}{\la^{1-\alpha}} \quad \text{in}\ W,
\end{eqnarray}
where $\displaystyle S=by_{xx}+\sqrt{d(x)}\kappa(\alpha)\int_{\mathbb{R}}|\xi|^{\frac{2\alpha-1}{2}}\omega(x,\xi)d\xi$. Here we will check the condition \eqref{P2} by finding a contradiction with \eqref{pol1beam-las} by showing $\|U\|_{\mathcal{H}_4}=o(1)$. 
In order to reach this contradiction, we proceed in a similar way as in Section \ref{Section-poly-beam}. We give the estimation results directly considering that the proof of these results can follow using similar computations as in Section \ref{Section-poly-beam}.\\
Assume that $\eta>0$. Then, the solution $(y,z,\omega)\in D(\AA_4)$ of system \eqref{AUG1Beam}-\eqref{AUG3B-las} satisfies the asymptotic behavior estimations mentioned below.\\
Similar to Lemma  \ref{First-Estimation-beam}, we obtain
\begin{equation}\label{only1}
\int_0^L\int_{\R}\left(\abs{\xi}^2+\eta\right)\abs{\omega(x,\xi)}^2d\xi dx=\frac{o\left(1\right)}{\la^{1-\alpha}}\, ,\,\int_{l_0}^{l_1}\abs{z_{xx}}^2dx=o(1)\quad\text{and}\quad
 \int_{l_0}^{l_1}\abs{y_{xx}}^2dx=\frac{o\left(1\right)}{\la^2}.
\end{equation}
Similar to Lemma \ref{First-Estimation-beam}, we obtain
\begin{equation}\label{eqS1only}
\int_{l_0}^{l_1}\left|S\right|^2dx=\frac{o(1)}{\la^{1-\alpha}}.
\end{equation}
Similar to Lemma \ref{lem-inter1}, we get
\begin{equation}
\|z\|_{H^2(l_0,l_1)}=o(1).
\end{equation}
Similar to Lemma \ref{Lem5.5}, we obtain that the solution of the system \eqref{AUG1Beam}-\eqref{AUG3B-las} satisfies the following estimations
\begin{equation}\label{eq1-5}
\abs{y(l_0)}=\frac{o(1)}{\la},\,\abs{y_x(l_0)}=\frac{o(1)}{\la},\,\abs{y(l_1)}=\frac{o(1)}{\la},\,\abs{y_x(l_1)}=\frac{o(1)}{\la}
\end{equation}
and
\begin{equation}\label{eq2-5}
\frac{1}{\la^{\frac{1}{2}}} \abs{y_{xxx}(l_0^{-})}=\frac{o(1)}{\la^{\frac{1}{4}-\frac{\alpha}{4}}} \quad\text{and}\quad \frac{1}{\la^{\frac{1}{2}}}\abs{y_{xxx}(l_1^{+})}=\frac{o(1)}{\la^{\frac{1}{4}-\frac{\alpha}{4}}}.
\end{equation}
Similar to Lemma \ref{Lem-yxx}, we obtain
\begin{equation}
\abs{y_{xx}(l_0^{-})}=o(1)\quad\text{and}\quad\abs{y_{xx}(l_1^+)}=o(1).
\end{equation}
Finally, by proceeding in a similar way as in Lemma \ref{Stab-beam}, we reach our desired result
\begin{equation}\label{only2}
\int_0^L\abs{z}^2dx=o(1)\quad\text{and}\quad b\int_0^L\abs{y_{xx}}^2dx=o(1)
\end{equation}
\noindent \textbf{Proof of Theorem \ref{pol1beam-las}.} 
From the first estimation in Equation \eqref{only1} and Equation \eqref{only2}, we get that $\|U\|_{\HH_4}=o(1)$, which contradicts \eqref{pol1beam-las}. Consequently, condition \eqref{P2} holds. This implies, from Theorem \ref{bt}, the energy decay estimation \eqref{Energypolbeam-las}. The proof is thus complete.

      
\section{(EBB)-(EBB)$_{FKV}$ Model}      
In this section, we consider a system of two Euler-Bernoulli beam equation coupled via boundary connections with a localized non-regular fractional Kelvin-Voigt damping acting on one of the two equations only. In this part, we study the polynomial stabilty of the system.\\
In Theorem \ref{theorem1}, taking the input $V(x,t)=\sqrt{d(x)}y_{xxt}(x,t)$, then  using \eqref{Caputo}, we get the output $O$ is given by 
$$
O(x,t)=\sqrt{d(x)}I^{1-\alpha,\eta}y_{xxt}(x,t)=\frac{\sqrt{d(x)}}{\Gamma(1-\alpha)}\int_0^t(t-s)^{-\alpha}e^{-\eta(t-s)}\partial_sy_{xx}(x,s)ds=\sqrt{d(x)}\partial_t^{\alpha,\eta}y_{xx}(x,t).
$$
Then, we deduce that system \eqref{coupledBeam} can be recast into the following augmented model 
\begin{equation}\label{AUG1-cp}
\left\{\begin{array}{ll}
u_{tt}+au_{xxxx}=0,&(x,t)\in (-L,0)\times \R_{\ast}^+,\\[0.1in]
\displaystyle{y_{tt}+\left(by_{xx}+\sqrt{d(x)}\kappa(\alpha)\int_{\mathbb{R}}\abs{\xi}^{\frac{2\alpha-1}{2}}\omega(x,\xi,t)d\xi\right)_{xx}=0},&(x,t)\in (0,L) \times \R_{\ast}^+,\\[0.1in]
\omega_t(x,\xi,t)+\left(|\xi|^2+\eta\right)\omega(x,\xi,t)-\sqrt{d(x)}y_{xxt}(x,t)\abs{\xi}^{\frac{2\alpha-1}{2}}=0,&(x,\xi,t)\in (0,L)\times \mathbb{R}\times \R_{\ast}^+,
\end{array}
\right.
\end{equation}
with the following transmission and boundary conditions 
\begin{equation}\label{AUG2-cp}
\left\{\begin{array}{ll}
u(-L,t)=u_x(-L,t)=y(L,t)=y_{x}(L,t)=0,&t\in (0,\infty),\\[0.1in]
au_{xxx}(0,t)-by_{xxx}(0,t)=0,u_{xx}(0,t)=y_{xx}(0,t)=0,&t\in (0,\infty),\\[0.1in]
u(0,t)=y(0,t),&t\in (0,\infty),
\end{array}\right.
\end{equation}
and with the following initial conditions 
\begin{equation}\label{AUG3-cp}
\begin{array}{lll}
u(x,0)=u_0(x),& u_t(x,0)=u_1(x)& x\in (-L,0)\\[0.1in]
y(x,0)=y_0(x),& y_t(x,0)=y_1(x),\hspace{0.2cm}\omega(x,\xi,0)=0& x\in (0,L),\xi\in\mathbb{R}.
\end{array}
\end{equation}
The energy of the system \eqref{AUG1-cp}-\eqref{AUG3-cp} is given by  
\begin{equation*}
E_5(t)=\frac{1}{2}\int_{-L}^0\left(\abs{u_t}^2+a\abs{u_{xx}}^2\right)dx+\frac{1}{2}\int_{0}^L\left(\abs{y_t}^2+b\abs{y_{xx}}^2\right)dx+\frac{\kappa(\alpha)}{2}\int_{0}^L\int_{\R}\abs{\omega(x,\xi,t)}^2d\xi dx.
\end{equation*}
By similar computation to Lemma \ref{Denergy}, it is easy to see that the energy $E_5(t)$ satisfies the following estimation 
\begin{equation}\label{denergybeam}
\frac{d}{dt}E_5(t)=-\kappa(\alpha)\int_{0}^L\int_{\mathbb{R}}(\xi^2+\eta)\abs{\omega(x,\xi,t)}^2d\xi dx.
\end{equation}
$\newline$
\noindent Since $\alpha\in (0,1)$, then $\kappa(\alpha)>0$, and therefore $\displaystyle\frac{d}{dt}E_5(t)\leq 0$. Thus, system \eqref{AUG1-cp}-\eqref{AUG3-cp} is dissipative in the sense that its energy is a non-increasing function with respect to time variable $t$. Now, we define the following Hilbert energy space $\mathcal{H}_{5}$ by 
$$
\mathcal{H}_{5}=\left\{(u,v,y,z,\omega)\in H_L^2(-L,0)\times L^2(-L,0)\times H^2_R(0,L)\times L^2(0,L)\times W;\ \ u(0)=y(0)\right\},
$$
where $W=L^2\left((0,L)\times \R\right)$ and
\begin{equation}
\left\{
\begin{array}{ll}
H_L^2(-L,0)=\{u\in H^2(-L,0); u(-L)=u_x(-L)=0\},\\[0.1in]
H_R^2(0,L)=\{y\in H^2(0,L); y(L)=y_x(L)=0\}.
\end{array}\right.
\end{equation} 
The energy space $\mathcal{H}_5$ is equipped with the inner product defined by 
$$
\begin{array}{lll}
\displaystyle
\left<U,U_1\right>_{\mathcal{H}_5}&=&\displaystyle
\int_{-L}^0 v\bar{v_1}dx+a\int_{-L}^0u_{xx}(\overline{u_1})_{xx}dx+\int_{0}^L z\bar{z_1}dx+b\int_{0}^L y_{xx}(\overline{y_1})_{xx}dx
+\kappa(\alpha)\int_{0}^L\int_{\mathbb{R}}\omega(x,\xi)\overline{\omega_1}(x,\xi)d\xi dx,
\end{array}
$$
for all $U=(u,v,y,z,\omega)$ and $U_1=(u_1,v_1,y_1,z_1,\omega_1)$ in $\mathcal{H}_{5}$. We use $\|U\|_{\mathcal{H}_{5}}$ to denote the corresponding norm. We define the unbounded linear operator $\AA_{5}:D(\AA_{5})\subset\mathcal{H}_{5}\rightarrow\mathcal{H}_{5}$ by 
\begin{equation*}
D(\mathcal{A}_{5})=\left\{\begin{array}{c}
\displaystyle{U=(u,v,y,z,\omega)\in \mathcal{H}_5;\ (v,z)\in H_L^2(-L,0)\times H^2_R(0,L),\ u\in  H^4(-L,0)},\\[0.1in]  
\displaystyle{\left(by_{xx}+\sqrt{d(x)}\kappa(\alpha)\int_{\mathbb{R}}\abs{\xi}^{\frac{2\alpha-1}{2}}\omega(x,\xi)d\xi\right)_{xx}}\in L^2(0,L),\vspace{0.2cm}\\ [0.1in]
-\left(\abs{\xi}^2+\eta\right)\omega(x,\xi)+\sqrt{d(x)}z_{xx}|\xi|^{\frac{2\alpha-1}{2}},\quad|\xi|\omega(x,\xi)\in W, \vspace{0.2cm}\\[0.1in]
au_{xxx}(0)-by_{xxx}(0)=0,\,  u_{xx}(0)=y_{xx}(0)=0,\,\text{and}\, v(0)=z(0)
\end{array}
\right\},
\end{equation*}
and for all $U=(u,v,y,z,\omega)\in D(\mathcal{A}_{5})$, 
$$
\mathcal{A}_{5}(u,v,y,z,\omega)^{\top}=\begin{pmatrix}
v\\ \vspace{0.2cm} -au_{xxxx}\\ \vspace{0.2cm} z\\\vspace{0.2cm}\displaystyle{-\left(by_{xx}+\sqrt{d(x)}\kappa(\alpha)\int_{\mathbb{R}}\abs{\xi}^{\frac{2\alpha-1}{2}}\omega(x,\xi)d\xi\right)_{xx}}\\ \vspace{0.2cm}
-\left(\abs{\xi}^2+\eta\right)\omega(x,\xi)+\sqrt{d(x)}z_{xx}|\xi|^{\frac{2\alpha-1}{2}}
\end{pmatrix}.
$$
\noindent If $U=(u,u_t,y,y_t,\omega)$ is a regular solution of system \eqref{AUG1-cp}-\eqref{AUG3-cp}, then the system can be rewritten as evolution equation on the Hilbert space $\mathcal{H}_5$ given by
\begin{equation}\label{evolution-cp}
U_t=\mathcal{A}_{5}U,\quad U(0)=U_0,
\end{equation}
where $U_0=(u_0,u_1,y_0,y_1,0)$.\\
\noindent Similar to Proposition \ref{mdissipatif}, the operator $\mathcal{A}_5$ is m-dissipative on $\mathcal{H}_5$, consequently it  generates a $C_0$-semigroup of contractions $(e^{t\mathcal{A}_5})_{t\geq 0}$ following Lummer-Phillips theorem (see in \cite{LiuZheng01} and \cite{Pazy01}). Then the solution of the evolution Equation \eqref{evolution-cp} admits the following representation
$$
U(t)=e^{t\mathcal{A}_5}U_0,\quad t\geq 0,
$$
which leads to the well-posedness of \eqref{evolution-cp}. Hence, we have the following result. 
\begin{theoreme}
Let $U_0\in \mathcal{H}_5$, then problem \eqref{evolution} admits a unique weak solution $U$ satisfies 
$$
U(t)\in C^0\left(\R^+,\mathcal{H}_5\right).
$$
Moreover, if $U_0\in D(\mathcal{A}_5)$, then problem \eqref{evolution} admits a unique strong solution $U$ satisfies 
$$
U(t)\in C^1\left(\R^+,\mathcal{H}_5\right)\cap C^0\left(\R^+,D(\mathcal{A}_5)\right).
$$
\end{theoreme}
\begin{theoreme}\label{polbeam-cp}
Assume that $\eta>0$. The $C_0-$ semigroup $(e^{t\AA_5})_{t\geq 0}$ is polynomially stable; i.e. there exists constant $C_5>0$ such that for every $U_0\in D(\AA_5)$, we have 
\begin{equation}\label{Energypolbeam-cp}
E_5(t)\leq \frac{C_5}{t^{\frac{2}{3-\alpha}}}\|U_0\|^2_{D(\AA_5)},\quad t>0,\,\forall\, U_0\in D(\mathcal{A}_5).
\end{equation}
\end{theoreme}
\noindent According to Theorem \ref{bt}, by taking $\ell=3-\alpha$, the polynomial energy decay \eqref{Energypolbeam-cp} holds if the following conditions 
\begin{equation}\label{Q1}\tag{${\rm{Q_1}}$}
i\R\subset \rho(\mathcal{A}_5),
\end{equation}
and
\begin{equation}\label{Q2}\tag{${\rm{Q_2}}$}
\sup_{\la\in \R}\left\|(i\la I-\AA_5)^{-1}\right\|_{\mathcal{L}(\mathcal{H}_5)}=O\left(\abs{\la}^{3-\alpha}\right)
\end{equation}
are satisfied. 
\noindent Since condition \eqref{Q1}is already proved (see Subsection \ref{subsec}), we still need to prove condition \eqref{Q2}. For this purpose we will use an argument of contradiction. Suppose that \eqref{Q2} is false, then there exists $\left\{\left(\la_n,U_n:=(u_n,v_n,y_n,z_n,\omega_n(\cdot,\xi))^\top\right)\right\}\subset \R^{\ast}\times D(\AA_5)$ with 
\begin{equation}\label{pol1beam-cp}
\abs{\la_n}\to +\infty \quad \text{and}\quad \|U_n\|_{\mathcal{H}_5}=\|(u_n,v_n,y_n,z_n,\omega_n(\cdot,\xi))\|_{\mathcal{H}_5}=1, 
\end{equation}
such that 
\begin{equation}\label{pol2-cp}
\left(\la_n^{2-\alpha}\right)\left(i\la_nI-\AA_5\right)U_n=F_n:=(f_{1,n},f_{2,n},f_{3,n},f_{4,n},f_{5,n}(\cdot,\xi))^{\top}\to 0 \ \ \text{in}\ \ \mathcal{H}_5. 
\end{equation}
For simplicity, we drop the index $n$. Equivalently, from \eqref{pol2-cp}, we have 
\begin{eqnarray}
i\la u-v&=&\dfrac{f_1}{\la^{3-\alpha}} \quad\text{in}\ H_L^2(-L,0),\label{pol3beam-cp}\\
i\la v+au_{xxxx}&=&\dfrac{f_2}{\la^{3-\alpha}} \quad \text{in}\ L^2(-L,0),\label{pol4beam-cp}\\
i\la y-z&=&\dfrac{f_3}{\la^{3-\alpha}} \quad\text{in}\ H_R^2(0,L),\label{pol5beam-cp}\\
i\la z+S_{xx}&=&\dfrac{f_4}{\la^{3-\alpha}}\quad \text{in}\ L^2(0,L),\label{pol6beam-cp}\\
(i\la+\xi^2+\eta)\omega(x,\xi)-\sqrt{d(x)}z_{xx}|\xi|^{\frac{2\alpha-1}{2}}&=&\dfrac{f_5(x,\xi)}{\la^{3-\alpha}} \quad \text{in}\ W,\label{pol7beam-cp}
\end{eqnarray}
where $\displaystyle S=by_{xx}+\sqrt{d(x)}\kappa(\alpha)\int_{\mathbb{R}}|\xi|^{\frac{2\alpha-1}{2}}\omega(x,\xi)d\xi$. Here we will check the condition \eqref{Q2} by finding a contradiction with \eqref{pol1beam-cp} by showing $\|U\|_{\mathcal{H}_5}=o(1)$. We need to prove several asymptotic behavior estimations for the solution to obtain this contradiction.  Here, we give the estimations directly since the proof can be done in a smiliar way as in Subsection \ref{Section-poly-beam}. Assume that $\eta>0$. 
Similar to Lemma \ref{First-Estimation-beam}, the solution $(u,v,y,z,\omega)\in D(\AA_5)$ of system \eqref{pol3beam-cp}-\eqref{pol7beam-cp} satisfies the following asymptotic behavior estimations 
\begin{equation}\label{only1-cp}
\begin{array}{l}
\displaystyle
\int_0^L\int_{\R}\left(\abs{\xi}^2+\eta\right)\abs{\omega(x,\xi)}^2d\xi dx=\frac{o\left(1\right)}{\la^{3-\alpha}},\quad \int_{l_0}^{l_1}\abs{z_{xx}}^2dx=\frac{o\left(1\right)}{\la^2},
 \\
 \displaystyle
\int_{l_0}^{l_1}\abs{y_{xx}}^2dx=\frac{o\left(1\right)}{\la^4}\quad
\text{and}\quad \int_{l_0}^{l_1}\left|S\right|^2dx=\frac{o(1)}{\la^{3-\alpha}}.
\end{array}
\end{equation}
Similar to Lemma  \ref{lem-inter1}, we have that the solution of the system \eqref{pol1beam-cp}-\eqref{pol7beam-cp} satisfies
\begin{equation}
\|z\|_{H^2(l_0,l_1)}=\frac{o(1)}{\la},\quad \|y\|_{H^2(l_0,l_1)}=\frac{o(1)}{\la^{2}}\quad \text{and}\quad \|S_x\|_{L^2(l_0,l_1)}=\frac{o(1)}{\la^{\frac{3-\alpha}{4}}}. 
\end{equation}
Similar to Lemma \ref{Lem5.5}, we get that the solution $(u,v,y,z,\omega)\in D(\AA_5)$ of system \eqref{pol3beam-cp}-\eqref{pol7beam-cp} satisfies the following asymptotic behavior
\begin{equation}\label{cp1}
\abs{y(l_0)}=\frac{o(1)}{\la^{2}}, \abs{y_x(l_0)}=\frac{o(1)}{\la^{2}}, \abs{y(l_1)}=\frac{o(1)}{\la^{2}},\,\text{and}\,\,\abs{y_x(l_1)}=\frac{o(1)}{\la^{2}}.
\end{equation}
Moreover,
\begin{equation}\label{cp2}
\frac{1}{\la^{\frac{1}{2}}}\abs{ y_{xxx}(l_0^{-})}=\frac{o(1)}{\la^{\frac{5-\alpha}{4}}}\quad \text{and}\quad \frac{1}{\la^{\frac{1}{2}}}\abs{y_{xxx}(l_1^{+})}=\frac{o(1)}{\la^{\frac{5-\alpha}{4}}}.
\end{equation}
Similar to Lemma \ref{lem-BW}, by multiplying Equation \eqref{pol6beam-cp} by $2h\bar{y}_x$, where $h\in C^{2}([0,L])$ and $h(0)=h(L)=0$ and proceeding as Lemma \ref{lem-BW}, we obtain
\begin{equation}\label{Bbeam-cp}
\begin{array}{ll}
\displaystyle\int_0^Lh^{\prime}\abs{z}^2dx+2b\int_0^L h^{\prime}\abs{y_{xx}}^2dx+b\int_0^{l_0} h^{\prime}\abs{y_{xx}}^2dx+b\int_{l_1}^{L} h^{\prime}\abs{y_{xx}}^2dx+\Re\left(2\int_0^L h^{\prime\prime}S\bar{y}_xdx\right)\\
\displaystyle-\Re\left(2\int_{l_0}^{l_1} hS_x\bar{y}_{xx}dx\right)-bh(l_0)\abs{y_{xx}(l_0^-)}^2+bh(l_1)\abs{y_{xx}(l_1^+)}^2=\frac{o(1)}{\la^{3-\alpha}}.
\end{array}
\end{equation}
Similar to Lemma \eqref{Lem5.7}, we get that
\begin{equation}\label{O(1)-cp}
\abs{y_{xx}(l_0^-)}=O(1)\quad\text{and}\quad\abs{y_{xx}(l_{1}^+)}=O(1).
\end{equation}
Similar to Lemma \ref{Lem-yxx}, we obtain that the solution $(u,v,y,z,\omega)\in D(\AA_5)$ of system \eqref{pol3beam-cp}-\eqref{pol7beam-cp} satisfies the following asymptotic behavior 
\begin{equation}
\abs{y_{xx}(l_0^{-})}=\frac{o(1)}{\la}\quad\text{and}\quad\abs{y_{xx}(l_1^+)}=\frac{o(1)}{\la}.
\end{equation}
Similar to Lemma \ref{Stab-beam}, we have that the solution $(u,v,y,z,\omega)\in D(\AA_5)$ of system \eqref{pol3beam-cp}-\eqref{pol7beam-cp} satisfies the following asymptotic behavior 
\begin{equation}\label{coupled1}
\int_0^L\abs{z}^2dx=\frac{o(1)}{\la^2}\quad\text{and}\quad b\int_0^L\abs{y_{xx}}^2dx=\frac{o(1)}{\la^2}.
\end{equation}
Similar to Lemma \ref{estOfyxxx}, we reach that 
\begin{equation}\label{cp3}
\abs{y_{xxx}(0)}=o(1).
\end{equation}
Using the transmission condition ($au_{xxx}(0)-by_{xxx}(0)=0$) and Equation \eqref{cp3}, we obtain
\begin{equation}\label{cp4}
\abs{u_{xxx}(0)}=o(1).
\end{equation}
Similar the computations in Equations \eqref{q-eq} and using the fact that $\abs{y_x(0)}\lesssim \|y_{xx}\|_{L^2(0,L)}=O(1)$, we obtain
\begin{equation}\label{cp5}
\abs{\la y(0)}^2=o(1).
\end{equation}
By using the transmission conditions and Equation \eqref{cp5}, we get
\begin{equation}\label{cp6}
\abs{\la u(0)}^2=\abs{\la y(0)}^2=o(1).
\end{equation}
Substituting $v=i\la u-\la^{-2+\alpha}$ into Equation \eqref{pol4beam-cp}, we get 
\begin{equation}\label{mult}
-\la^2 u+au_{xxxx}=\la^{-3+\alpha}f_2+i\la^{-2+\alpha}f_1.
\end{equation} 
Now, multiply Equation \eqref{mult} by $2(x+L)\bar{u_x}$ and integrate over $(-L,0)$, and by using boundary conditions and the fact that $\|f_1\|_{H^2_L(-L,0)}=o(1)$, $ \|f_2\|_{L^2(0,L)}=o(1)$, $\abs{u_x(0)}\leq \|u_{xx}\|_{L^2(-L,0)}=O(1)$,  $\|u_{x}\|_{L^2(-L,0)}\leq c_p\|u_{xx}\|_{L^2(-L,0)}=O(1)$ , we get
\begin{equation}\label{q-eq-cp}
\int_{-L}^{0}\left(\abs{\la u}^2+3\abs{u_{xx}}^2\right)dx-L\abs{\la y(0)}^2+\Re\left(2aLu_{xxx}(0)\bar{u}_x(0)\right)=\frac{o(1)}{\la^{2-\alpha}}.
\end{equation}
By using equations \eqref{cp4} and \eqref{cp6}, and the fact that 
 $\abs{u_x(0)}\lesssim \|u_{xx}\|_{L^2(-L,0)}=O(1)$ in Equation \eqref{q-eq-cp}, we obtain
 \begin{equation}\label{coupled2}
\int_{-L}^0\abs{v}^2dx=o(1)\quad\text{and}\quad a\int_{-L}^0\abs{u_{xx}}^2dx=o(1).
\end{equation}
\noindent \textbf{Proof of Theorem \ref{pol1beam-cp}.} 
From the first estimation in Equation \eqref{only1-cp} and Equations \eqref{coupled1} and \eqref{coupled2}, we get that $\|U\|_{\HH_5}=o(1)$, which contradicts \eqref{pol1beam-cp}. Consequently, condition \eqref{Q2} holds. This implies, from Theorem \ref{bt}, the energy decay estimation \eqref{Energypolbeam-cp}. The proof is thus complete.

\section{Appendix}
\noindent In this section, we introduce the notions of stability that we encounter in this work.


\begin{definition}\label{Defsta}
{Assume that $A$ is the generator of a C$_0$-semigroup of contractions $\left(e^{tA}\right)_{t\geq0}$  on a Hilbert space  $\mathcal{H}$. The  $C_0$-semigroup $\left(e^{tA}\right)_{t\geq0}$  is said to be
\begin{enumerate}
\item[1.]  strongly stable if 
$$\lim_{t\to +\infty} \|e^{tA}x_0\|_{H}=0, \quad\forall \ x_0\in H;$$
\item[2.]  exponentially (or uniformly) stable if there exist two positive constants $M$ and $\epsilon$ such that
\begin{equation*}
\|e^{tA}x_0\|_{H} \leq Me^{-\epsilon t}\|x_0\|_{H}, \quad
\forall\  t>0,  \ \forall \ x_0\in {H};
\end{equation*}
\item[3.] polynomially stable if there exists two positive constants $C$ and $\alpha$ such that
\begin{equation*}
 \|e^{tA}x_0\|_{H}\leq C t^{-\alpha}\|x_0\|_{H},  \quad\forall\ 
t>0,  \ \forall \ x_0\in D\left(\mathcal{A}\right).
\end{equation*}
In that case, one says that the semigroup $\left(e^{tA}\right)_{t\geq 0}$ decays  at a rate $t^{-\alpha}$.
\noindent The  $C_0$-semigroup $\left(e^{tA}\right)_{t\geq0}$  is said to be  polynomially stable with optimal decay rate $t^{-\alpha}$ (with $\alpha>0$) if it is polynomially stable with decay rate $t^{-\alpha}$ and, for any $\varepsilon>0$ small enough, the semigroup $\left(e^{tA}\right)_{t\geq0}$  does  not decay at a rate $t^{-(\alpha-\varepsilon)}$.
\end{enumerate}}
\end{definition}
\noindent To show the strong stability of a $C_0-$semigroup of contraction  $(e^{tA})_{t\geq 0}$ we rely on the following result due to Arendt-Batty \cite{Arendt01}.

\begin{theoreme}\label{arendtbatty}
Assume that $A$ is the generator of a C$_0-$semigroup of contractions $\left(e^{tA}\right)_{t\geq0}$  on a Hilbert space $\mathcal{H}$. If
 \begin{enumerate}
 \item[1.]  $A$ has no pure imaginary eigenvalues,
  \item[2.]  $\sigma\left(A\right)\cap i\mathbb{R}$ is countable,
 \end{enumerate}
where $\sigma\left(A\right)$ denotes the spectrum of $A$, then the $C_0-$semigroup $\left(e^{tA}\right)_{t\geq0}$  is strongly stable.
\end{theoreme}

\noindent  Concerning the characterization of exponential stability of a $C_0-$semigroup of contraction  $(e^{tA})_{t\geq 0}$ we rely on the following result due to Huang \cite{Huang01} and Pr\"uss \cite{pruss01}. 
\begin{theoreme}\label{hp}
Let $A:\ D(A)\subset H\rightarrow H$ generate a $C_0-$semigroup of contractions $\left(e^{tA}\right)_{t\geq 0}$ on $H$. Assume that $i\la \in \rho(A)$, $\forall \la \in \R$. Then, the $C_0-$semigroup $\left(e^{tA}\right)_{t\geq 0}$ is exponentially stable if and only if 
$$
\varlimsup_{\la\in\R,\ \abs{\la}\to +\infty}\|(i\la I-A)^{-1}\|_{\mathcal{L}(H)}<+\infty.
$$
\end{theoreme}
\noindent  Also, concerning the characterization of polynomial stability of a $C_0-$semigroup of contraction  $(e^{tA})_{t\geq 0}$ we rely on the following result due to Borichev and Tomilov \cite{Borichev01} (see also \cite{RaoLiu01} and \cite{Batty01}). 

\begin{theoreme}\label{bt}
Assume that $A$ is the generator of a strongly continuous semigroup of contractions $\left(e^{tA}\right)_{t\geq0}$  on $H$.   If   $ i\mathbb{R}\subset \rho(A)$, then for a fixed $\ell>0$ the following conditions are equivalent
\begin{equation}\label{h1}
\sup_{\lambda\in\mathbb{R}}\left\|\left(i\lambda I-A\right)^{-1}\right\|_{\mathcal{L}\left(H\right)}=O\left(|\lambda|^\ell\right),
\end{equation}
\begin{equation}\label{h2}
\|e^{tA}U_{0}\|^2_{H} \leq \frac{C}{t^{\frac{2}{\ell}}}\|U_0\|^2_{D(A)},\hspace{0.1cm}\forall t>0,\hspace{0.1cm} U_0\in D(A),\hspace{0.1cm} \text{for some}\hspace{0.1cm} C>0.
\end{equation}
\end{theoreme}

\section{Conclusion}
We have studied the stabilization of five models of systems. We considered a Euler-Bernoulli beam equation and a wave equation coupled through boundary connections with a localized non-regular fractional Kelvin-Voigt damping that acts through the wave equation only. We proved the strong stability of the system using Arendt-Batty criteria. In addition, we established a polynomial energy decay rate of type $t^{\frac{-4}{2-\alpha}}$. Also, we considered two wave equations coupled via boundary connections with localized non-smooth fractional Kelvin-Voigt damping. We showed a polynomial energy decay rate of type $t^{\frac{-4}{2-\alpha}}$. Moreover, we studied the system of Euler-Bernoulli beam and wave equations coupled through boundary connections where the dissipation acts through the beam equation. We proved a polynomial energy decay rate of type $t^{\frac{-2}{3-\alpha}}$. In addition, we considered the Euler-Bernoulli beam alone with the same localized non-smooth damping. We established a polynomial energy decay rate of type $t^{\frac{-2}{1-\alpha}}$. In the last model, we studied the polynomial stability of a system of two Euler-Bernoulli beam equations coupled through boundary conditions with a localized non-regular fractional Kelvin-Voigt damping acting only on one of the two equations. We reached a polynomial energy decay rate of type $t^{\frac{-2}{3-\alpha}}$.
\section*{Acknowledgements }
\noindent Mohammad Akil would like to thank LAMA laboratory of Mathematics of the Université
Savoie Mont Blanc for their supports.\\

\noindent Ibtissam Issa would like to thank the Lebanese University for its funding and LAMA laboratory of Mathematics of the Université Savoie Mont Blanc for their support.\\

\noindent Ali Wehbe would like to thank the CNRS and the LAMA laboratory of Mathematics of the Université Savoie Mont Blanc for their supports.


\begin{thebibliography}{10}

\bibitem{benaissa}
Z.~Achouri, N.~E. Amroun, and A.~Benaissa.
\newblock The euler–bernoulli beam equation with boundary dissipation of
  fractional derivative type.
\newblock {\em Mathematical Methods in the Applied Sciences},
  40(11):3837--3854.

\bibitem{Adams}
R.~A. Adams.
\newblock {\em Sobolev spaces / Robert A. Adams}.
\newblock Academic Press New York, 1975.

\bibitem{Akil-AA}
M.~Akil, Y.~Chitour, M.~Ghader, and A.~Wehbe.
\newblock Stability and exact controllability of a timoshenko system with only
  one fractional damping on the boundary.
\newblock {\em Asymptotic Analysis}, 119:1--60, 10 2019.

\bibitem{akil-mcrf}
M.~Akil and A.~Wehbe.
\newblock Stabilization of multidimensional wave equation with locally boundary
  fractional dissipation law under geometric conditions.
\newblock {\em Mathematical Control \& Related Fields}, 8:1--20, 01 2018.

\bibitem{Alabau1999}
F.~Alabau.
\newblock Stabilisation fronti{\`{e}}re indirecte de syst{\`{e}}mes faiblement
  coupl{\'{e}}s.
\newblock {\em Comptes Rendus de l{\textquotesingle}Acad{\'{e}}mie des Sciences
  - Series I - Mathematics}, 328(11):1015--1020, June 1999.

\bibitem{Alabau-cannarsa-komornik}
F.~Alabau, P.~Cannarsa, and V.~Komornik.
\newblock Indirect internal stabilization of weakly coupled evolution
  equations.
\newblock {\em Journal of Evolution Equations}, 2(2):127--150, May 2002.

\bibitem{Alabau2002}
F.~Alabau-Boussouira.
\newblock Indirect boundary stabilization of weakly coupled hyperbolic systems.
\newblock {\em SIAM Journal on Control and Optimization}, 41(2):511--541, 2002.

\bibitem{Alabau2007}
F.~Alabau-Boussouira.
\newblock Asymptotic behavior for timoshenko beams subject to a single
  nonlinear feedback control.
\newblock {\em Nonlinear Differential Equations and Applications {NoDEA}},
  14(5-6):643--669, Dec. 2007.

\bibitem{alabau}
F.~Alabau-Boussouira and M.~L\'eautaud.
\newblock Indirect stabilization of locally coupled wave-type systems.
\newblock {\em ESAIM: Control, Optimisation and Calculus of Variations},
  18(2):548--582, 2012.

\bibitem{Benissa-time}
H.~Allouni, M.~Kesri, and B.~Abbes.
\newblock On the asymptotic behaviour of two coupled strings through a
  fractional joint damper.
\newblock {\em Rendiconti del Circolo Matematico di Palermo Series 2},
  69:613--640, 08 2020.

\bibitem{Alves2}
M.~Alves, J.~Rivera, M.~Sepúlveda, and O.~Vera~Villagran.
\newblock The lack of exponential stability in certain transmission problems
  with localized kelvin--voigt dissipation.
\newblock {\em SIAM Journal on Applied Mathematics}, 74, 03 2014.

\bibitem{Ammari-fathi}
K.~Ammari, H.~Fathi, and L.~Robbiano.
\newblock Fractional-feedback stabilization for a class of evolution systems.
\newblock {\em Journal of Differential Equations}, 268(10):5751 -- 5791, 2020.

\bibitem{Ammari-J-Mehren2009}
K.~Ammari, M.~Jellouli, and M.~Mehrenberger.
\newblock Feedback stabilization of a coupled string-beam system.
\newblock {\em NHM}, 4:19--34, 03 2009.

\bibitem{Ammari03}
K.~Ammari, Z.~Liu, and S.~Farhat.
\newblock Stability of the wave equations on a tree with local kelvin-voigt
  damping.
\newblock {\em Semigroup Forum}, 03 2018.

\bibitem{Ammari-mehren2009}
K.~Ammari and M.~Mehrenberger.
\newblock Study of the nodal feedback stabilization of a string-beams network.
\newblock {\em JAMC J Appl Math Comput}, 36, 08 2011.

\bibitem{Ammari-serge}
K.~Ammari and S.~Nicaise.
\newblock Stabilization of a transmission wave/plate equation.
\newblock {\em Journal of Differential Equations}, 249(3):707 -- 727, 2010.

\bibitem{Ammari-2009-CUBO}
K.~Ammari and G.~Vodev.
\newblock Boundary stabilization of the transmission problem for the
  bernoulli-euler plate equation.
\newblock {\em Cubo}, 5, 01 2009.

\bibitem{Arendt01}
W.~Arendt and C.~J.~K. Batty.
\newblock Tauberian theorems and stability of one-parameter semigroups.
\newblock {\em Trans. Amer. Math. Soc.}, 306(2):837--852, 1988.

\bibitem{Ronald}
R.~L. Bagley and P.~J. Torvik.
\newblock Fractional calculus - a different approach to the analysis of
  viscoelastically damped structures.
\newblock {\em AIAA Journal}, 21(5):741--748, 1983.

\bibitem{RL}
R.~L. Bagley and P.~J. Torvik.
\newblock A theoretical basis for the application of fractional calculus to
  viscoelasticity.
\newblock {\em Journal of Rheology}, 27(3):201--210, 1983.

\bibitem{bardos}
C.~Bardos, G.~Lebeau, and J.~Rauch.
\newblock Sharp sufficient conditions for the observation, control, and
  stabilization of waves from the boundary.
\newblock {\em SIAM Journal on Control and Optimization}, 30(5):1024--1065,
  1992.

\bibitem{B-T1991}
J.~Bartolomeo and R.~Triggiani.
\newblock Uniform energy decay rates for euler–bernoulli equations with
  feedback operators in the dirichlet/neumann boundary conditions.
\newblock {\em SIAM Journal on Mathematical Analysis}, 22(1):46--71, 1991.

\bibitem{Batty01}
C.~J.~K. Batty and T.~Duyckaerts.
\newblock Non-uniform stability for bounded semi-groups on {B}anach spaces.
\newblock {\em J. Evol. Equ.}, 8(4):765--780, 2008.

\bibitem{saroj}
S.~K. Biswas and N.~U. Ahmed.
\newblock Optimal control of large space structures governed by a coupled
  system of ordinary and partial differential equations.
\newblock {\em Math. Control. Signals Syst.}, 2(1):1--18, 1989.

\bibitem{Borichev01}
A.~Borichev and Y.~Tomilov.
\newblock Optimal polynomial decay of functions and operator semigroups.
\newblock {\em Math. Ann.}, 347(2):455--478, 2010.

\bibitem{caputo}
M.~{Caputo}.
\newblock {Linear Models of Dissipation whose Q is almost Frequency
  Independent-II}.
\newblock {\em Geophysical Journal}, 13(5):529--539, Nov. 1967.

\bibitem{Caputo-Fab}
M.~Caputo and M.~Fabrizio.
\newblock A new definition of fractional derivative without singular kernel.
\newblock {\em Prog Fract Differ Appl}, 1:73--85, 04 2015.

\bibitem{Ch1987}
G.~Chen, M.~Delfour, A.~Krall, and G.~Payre.
\newblock Modeling, stabilization and control of serially connected beams.
\newblock {\em Siam Journal on Control and Optimization - SIAM J CONTR
  OPTIMIZAT}, 25, 05 1987.

\bibitem{chen1}
G.~Chen, S.~A. Fulling, F.~J. Narcowich, and S.~Sun.
\newblock Exponential decay of energy of evolution equations with locally
  distributed damping.
\newblock {\em SIAM Journal on Applied Mathematics}, 51(1):266--301, 1991.

\bibitem{Denk}
R.~Denk and F.~Kammerlander.
\newblock {Exponential stability for a coupled system of damped–undamped
  plate equations}.
\newblock {\em IMA Journal of Applied Mathematics}, 83(2):302--322, 02 2018.

\bibitem{Alabau-Cannarsa-Guglielmi2011}
R.~G. Fatiha Alabau-Boussouira, Piermarco~Cannarsa.
\newblock Indirect stabilization of weakly coupled systems with hybrid boundary
  conditions.
\newblock {\em Mathematical Control \& Related Fields}, 1(4):413--436, 2011.

\bibitem{lu}
X.~Fu and Q.~Lu.
\newblock Stabilization of the weakly coupled wave-plate system with one
  internal damping, 2017.

\bibitem{Guo2020}
B.-Z. Guo and H.-J. Ren.
\newblock Stability and regularity transmission for coupled beam and wave
  equations through boundary weak connections.
\newblock {\em ESAIM Control Optimisation and Calculus of Variations}, 26:29
  pp, 09 2020.

\bibitem{wang}
Y.-P. Guo, J.-M. Wang, and D.-X. Zhao.
\newblock Energy decay estimates for a two-dimensional coupled wave-plate
  system with localized frictional damping.
\newblock {\em ZAMM - Journal of Applied Mathematics and Mechanics /
  Zeitschrift für Angewandte Mathematik und Mechanik}, 100(2):e201900030,
  2020.

\bibitem{LIU-HAN2019}
{Han, Zhong-Jie} and {Liu, Zhuangyi}.
\newblock Regularity and stability of coupled plate equations with indirect
  structural or kelvin-voigt damping.
\newblock {\em ESAIM: COCV}, 25:51, 2019.

\bibitem{Fathi-2015}
F.~Hassine.
\newblock Energy decay estimates of elastic transmission wave/beam systems with
  a local kelvin-voigt damping.
\newblock {\em International Journal of Control}, pages 1--29, 12 2015.

\bibitem{Fathi-2016-N-dim}
F.~Hassine.
\newblock Asymptotic behavior of the transmission euler-bernoulli plate and
  wave equation with a localized kelvin-voigt damping.
\newblock {\em Discrete and Continuous Dynamical Systems - Series B},
  21:1757--1774, 06 2016.

\bibitem{Huang-falun}
F.~Huang.
\newblock On the mathematical model for linear elastic systems with analytic
  damping.
\newblock {\em SIAM Journal on Control and Optimization}, 26(3):714--724, 1988.

\bibitem{Huang01}
F.~L. Huang.
\newblock Characteristic conditions for exponential stability of linear
  dynamical systems in {H}ilbert spaces.
\newblock {\em Ann. Differential Equations}, 1(1):43--56, 1985.

\bibitem{LAS1999}
G.~Ji and I.~Lasiecka.
\newblock Nonlinear boundary feedback stabilization for a semilinear kirchhoff
  plate with dissipation acting only via moments-limiting behavior.
\newblock {\em Journal of Mathematical Analysis and Applications}, 229(2):452
  -- 479, 1999.

\bibitem{K-R-1987}
J.~U. Kim and Y.~Renardy.
\newblock Boundary control of the timoshenko beam.
\newblock {\em SIAM Journal on Control and Optimization}, 25(6):1417--1429,
  1987.

\bibitem{LAS1987}
J.~E. Lagnese.
\newblock Uniform boundary stabilization of homogeneous isotropic plates.
\newblock pages 204--215, 1987.

\bibitem{LAS1989}
I.~Lasiecka.
\newblock Stabilization of wave and plate-like equations with nonlinear
  dissipation on the boundary.
\newblock {\em Journal of Differential Equations}, 79(2):340 -- 381, 1989.

\bibitem{LAS1990}
I.~Lasiecka.
\newblock Asymptotic behavior of solutions to plate equations with nonlinear
  dissipation occurring through shear forces and bending moments.
\newblock {\em Applied Mathematics and Optimization}, 21:167--189, 1990.

\bibitem{Lebau}
J.~Le~Rousseau and G.~Lebeau.
\newblock On carleman estimates for elliptic and parabolic operators.
  applications to unique continuation and control of parabolic equations.
\newblock {\em ESAIM: Control, Optimisation and Calculus of Variations},
  18(3):712--747, 2012.

\bibitem{Li-Han-Xu}
Y.-F. Li, Z.-J. Han, and G.-Q. Xu.
\newblock Explicit decay rate for coupled string-beam system with localized
  frictional damping.
\newblock {\em Applied Mathematics Letters}, 78:51 -- 58, 2018.

\bibitem{chen-1998}
K.~Liu, S.~Chen, and Z.~Liu.
\newblock Spectrum and stability for elastic systems with global or local
  kelvin--voigt damping.
\newblock {\em SIAM Journal on Applied Mathematics}, 59(2):651--668, 1998.

\bibitem{liu-liu-1998}
K.~Liu and Z.~LIU.
\newblock Exponential decay of energy of the euler--bernoulli beam with locally
  distributed kelvin--voigt damping.
\newblock {\em SIAM Journal on Control and Optimization}, 36:1086--1098, 05
  1998.

\bibitem{RaoLiu01}
Z.~Liu and B.~Rao.
\newblock Characterization of polynomial decay rate for the solution of linear
  evolution equation.
\newblock {\em Z. Angew. Math. Phys.}, 56(4):630--644, 2005.

\bibitem{liu-zhang-2016}
Z.~Liu and Q.~Zhang.
\newblock Stability of a string with local kelvin--voigt damping and nonsmooth
  coefficient at interface.
\newblock {\em SIAM Journal on Control and Optimization}, 54:1859--1871, 01
  2016.

\bibitem{LiuZheng01}
Z.~Liu and S.~Zheng.
\newblock {\em Semigroups associated with dissipative systems}, volume 398 of
  {\em Chapman \& Hall/CRC Research Notes in Mathematics}.
\newblock Chapman \& Hall/CRC, Boca Raton, FL, 1999.

\bibitem{Bonetti}
M.~Mainardi and E.~Bonetti.
\newblock The application of real-order derivatives in linear viscoelasticity.
\newblock pages 64--67, 1988.

\bibitem{benotti}
M.~Mainardi and E.~Bonetti.
\newblock The application of real-order derivatives in linear viscoelasticity.
\newblock In H.~Giesekus and M.~F. Hibberd, editors, {\em Progress and Trends
  in Rheology II}, pages 64--67, Heidelberg, 1988. Steinkopff.

\bibitem{app}
D.~Matignon.
\newblock Asymptotic stability of webster-lokshin equation.
\newblock {\em Mathematical Control and Related Fields}, 4:481, 2014.

\bibitem{Mbodje1}
B.~Mbodje.
\newblock {Wave energy decay under fractional derivative controls}.
\newblock {\em IMA Journal of Mathematical Control and Information},
  23(2):237--257, 06 2006.

\bibitem{Mbodje2}
B.~{Mbodje} and G.~{Montseny}.
\newblock Boundary fractional derivative control of the wave equation.
\newblock {\em IEEE Transactions on Automatic Control}, 40(2):378--382, 1995.

\bibitem{NirenbergPaper}
L.~Nirenberg.
\newblock An extended interpolation inequality.
\newblock {\em Annali della Scuola Normale Superiore di Pisa - Classe di
  Scienze}, Ser. 3, 20(4):733--737, 1966.

\bibitem{Pazy01}
A.~Pazy.
\newblock {\em Semigroups of linear operators and applications to partial
  differential equations}, volume~44 of {\em Applied Mathematical Sciences}.
\newblock Springer-Verlag, New York, 1983.

\bibitem{Igor}
I.~Podlubny.
\newblock {\em {Fractional differential equations: an introduction to
  fractional derivatives, fractional differential equations, to methods of
  their solution and some of their applications}}.
\newblock Mathematics in Science and Engineering. Academic Press, London, 1999.

\bibitem{pruss01}
J.~Pr\"uss.
\newblock On the spectrum of {$C_{0}$}-semigroups.
\newblock {\em Trans. Amer. Math. Soc.}, 284(2):847--857, 1984.

\bibitem{Raposo}
C.~Raposo, W.~Bastos, and J.~Avila.
\newblock A transmission problem for euler-bernoulli beam with kelvin-voigt
  damping.
\newblock {\em Applied Mathematics and Information Sciences}, 5:17--28, 01
  2011.

\bibitem{Rivera-Santos2011}
J.~Rivera and M.~Santos.
\newblock Analytic property of a coupled system of wave-plate type with thermal
  effect.
\newblock {\em Differential and Integral Equations}, 24, 09 2011.

\bibitem{RUSSELL1993}
D.~Russell.
\newblock A general framework for the study of indirect damping mechanisms in
  elastic systems.
\newblock {\em Journal of Mathematical Analysis and Applications},
  173(2):339--358, 1993.

\bibitem{Teb12}
L.~Tebou.
\newblock Energy decay estimates for some weakly coupled euler-bernoulli and
  wave equations with indirect damping mechanisms.
\newblock {\em Mathematical Control \& Related Fields}, 1, 03 2012.

\bibitem{peter}
P.~Torvik and R.~Bagley.
\newblock On the appearance of the fractional derivative in the behavior of
  real materials.
\newblock {\em Journal of Applied Mechanics}, 51, 06 1984.

\bibitem{Torvik}
P.~J. Torvik and R.~L. Bagley.
\newblock {On the Appearance of the Fractional Derivative in the Behavior of
  Real Materials}.
\newblock {\em Journal of Applied Mechanics}, 51(2):294--298, 06 1984.

\bibitem{akil2020stability}
A.~Wehbe, I.~Issa, and M.~Akil.
\newblock Stability results of an elastic/viscoelastic transmission problem of
  locally coupled waves with non smooth coefficients.
\newblock {\em Acta Applicandae Mathematicae}, 171(1), Feb. 2021.

\bibitem{zhang-2010}
Q.~Zhang.
\newblock Exponential stability of an elastic string with local kelvin–voigt
  damping.
\newblock {\em Zeitschrift Fur Angewandte Mathematik Und Physik - ZAMP},
  61:1009--1015, 12 2010.

\bibitem{Zhang-Zuazua}
X.~Zhang and E.~Zuazua.
\newblock Long-time behavior of a coupled heat-wave system arising in
  fluid-structure interaction.
\newblock {\em Archive for Rational Mechanics and Analysis}, 184:49--120, 04
  2007.

\end{thebibliography}
\end{document}